\documentclass[psamsfonts,reqno]{amsart}
\usepackage{amssymb,eucal}
\usepackage[usenames]{color}

\usepackage{amscd}

\usepackage{graphics}

\usepackage{mathrsfs}

\usepackage[all]{xy}

\theoremstyle{plain}

\newtheorem{lemma}{Lemma}[section]
\newtheorem{theorem}[lemma]{Theorem}
\newtheorem{proposition}[lemma]{Proposition}
\newtheorem{corollary}[lemma]{Corollary}

\newtheorem*{sclaim}{Claim}

\newtheorem*{proposition*}{Proposition}
\newtheorem*{theorem*}{Theorem}

\newtheorem*{stat}{\name}
\newcommand{\name}{testing}

\theoremstyle{definition}
\newtheorem{definition}[lemma]{Definition}
\newtheorem{example}[lemma]{Example}

\theoremstyle{remark}
\newtheorem{remark}[lemma]{Remark}
\newtheorem{notation}[lemma]{Notation}

\newtheorem*{remark*}{Remark}

\newcommand{\qedc}{{\qed}~{\rm Claim~{\theclaim}.}}
\newcommand{\qedsc}{{\qed}~{\rm Claim.}}

\newenvironment{scproof}
{\begin{proof}[Proof of Claim.]}
{\qedsc\renewcommand{\qed}{}\end{proof}}

\numberwithin{equation}{section}

\newcommand{\pup}[1]{\textup{(}{#1}\textup{)}}

\newcommand{\set}[1]{\{#1\}}
\newcommand{\setm}[2]{\set{#1\mid#2}}
\newcommand{\Set}[1]{\left\{#1\right\}}
\newcommand{\Setm}[2]{\Set{#1\mid#2}}

\newcommand{\famm}[2]{(#1\mid#2)}

\newcommand{\gen}[1]{\langle{#1}\rangle}

\newcommand{\Pow}{\mathfrak{P}}

\newcommand{\dnw}{\mathbin{\downarrow}}

\newcommand{\Qone}{\mathbf{(Q1)}}
\newcommand{\Qtwo}{\mathbf{(Q2)}}
\newcommand{\Qonep}{\mathbf{(Q1')}}
\newcommand{\Qtwop}{\mathbf{(Q2')}}

\newcommand{\cA}{\mathcal{A}}

\newcommand{\cK}{\mathcal{K}}
\newcommand{\cL}{\mathcal{L}}
\newcommand{\cM}{\mathcal{M}}

\newcommand{\cS}{\mathcal{S}}
\newcommand{\cV}{\mathcal{V}}
\newcommand{\cW}{\mathcal{W}}
\newcommand{\rB}{\mathrm{B}}

\newcommand{\into}{\hookrightarrow}

\newcommand{\zero}{\mathbf{0}}

\newcommand{\two}{\mathbf{2}}

\newcommand{\bx}{\boldsymbol{x}}
\newcommand{\by}{\boldsymbol{y}}

\newcommand{\bu}{\boldsymbol{u}}

\newcommand{\bpi}{\boldsymbol{\pi}}
\newcommand{\brho}{\boldsymbol{\rho}}
\newcommand{\bff}{\boldsymbol{f}}
\newcommand{\bgg}{\boldsymbol{g}}
\newcommand{\bhh}{\boldsymbol{h}}
\newcommand{\bxi}{\boldsymbol{\xi}}
\newcommand{\bchi}{\boldsymbol{\chi}}

\newcommand{\bA}{\boldsymbol{A}}
\newcommand{\bB}{\boldsymbol{B}}
\newcommand{\bC}{\boldsymbol{C}}

\newcommand{\bU}{\boldsymbol{U}}

\newcommand{\bX}{\boldsymbol{X}}

\newcommand{\xF}{\mathbf{F}}

\newcommand{\CG}{\boldsymbol{C}}
\newcommand{\GA}{\boldsymbol{G}}
\newcommand{\PGGL}{\boldsymbol{P}_{\mathrm{gl}}}
\newcommand{\PGGR}{\boldsymbol{P}_{\mathrm{gr}}}
\newcommand{\PGA}{\boldsymbol{P}_{\mathrm{ga}}}
\newcommand{\CPG}{\boldsymbol{C}_{\mathrm{pg}}}

\newcommand{\sL}{\mathscr{L}}

\newcommand{\id}{\mathrm{id}}
\newcommand{\jz}{$(\vee,0)$}

\newcommand{\jzs}{\jz-semi\-lat\-tice}

\newcommand{\jzh}{\jz-ho\-mo\-mor\-phism}

\newcommand{\jze}{\jz-em\-bed\-ding}

\newcommand{\res}{\mathbin{\restriction}}
\newcommand{\pregamp}{pregamp}
\newcommand{\pregamps}{pregamps}
\newcommand{\subpregamp}{sub-pregamp}
\newcommand{\subpregamps}{sub-pregamps}

\DeclareMathOperator{\card}{card}

\DeclareMathOperator{\crita}{crit}
\newcommand{\crit}[2]{\crita({{#1};{#2}})}

\DeclareMathOperator{\Ids}{Id_s}

\DeclareMathOperator{\Def}{Def}

\DeclareMathOperator{\rng}{rng}
\DeclareMathOperator{\Max}{Max}
\DeclareMathOperator{\Con}{Con}
\newcommand{\Conc}{\Con_{\mathrm{c}}}
\DeclareMathOperator{\Id}{Id}
\DeclareMathOperator{\ari}{ar}

\DeclareMathOperator{\Var}{{\bf{Var}}}

\newcommand{\tosurj}{\mathbin{\twoheadrightarrow}}
\newcommand{\toinj}{\mathbin{\hookrightarrow}}

\newcommand{\SET}{\mathbf{Set}}
\newcommand{\PALG}{\mathbf{PAlg}}
\newcommand{\SEM}{\mathbf{Sem}_{\vee,0}}
\newcommand{\GAMP}{\mathbf{Gamp}}
\newcommand{\MPALG}{\mathbf{PGamp}}

\newcommand{\ignorer}[1]{}

\newcommand{\kposet}[4]{{{#1}} \mathbin{\boxtimes}_{{#3}} {{#4}} }

\newcommand{\dual}[1]{{#1}^{\mathrm{d}}}

\begin{document}

\title[Critical points]{Categories of partial algebras for critical points between varieties of algebras}

\author[P.~Gillibert]{Pierre Gillibert}
\address{Charles University in Prague, Faculty of Mathematics and Physics, Department of Algebra, Sokolovska 83, 186 00 Prague, Czech Republic.}
\email{gilliber@karlin.mff.cuni.cz}
\email{pgillibert@yahoo.fr}


\thanks{This work was partially supported by the institutional grant MSM 0021620839}

\date{\today}

\begin{abstract}
We denote by $\Conc A$ the \jzs\ of all finitely generated congruences of an algebra~$A$. A \emph{lifting} of a \jzs\ $S$ is an algebra~$A$ such that $S\cong \Conc A$.

The aim of this work is to give a categorical theory of partial algebras endowed with a partial subalgebra together with a semilattice-valued distance, that we call \emph{gamps}. This part of the theory is formulated in any variety of (universal) algebras.

Let~$\cV$ and $\cW$ be varieties of algebras (on a finite similarity type). Let $P$ be a finite lattice of order-dimension $d>0$. Let $\vec A$ be a $P$-indexed diagram of finite algebras in~$\cV$. If $\Conc\circ\vec A$ has no \emph{partial lifting} in the category of gamps of $\cW$, then there is an algebra $A\in\cV$ of cardinality $\aleph_{d-1}$ such that $\Conc A$ is not isomorphic to $\Conc B$ for any $B\in\cW$.

We already knew a similar result for diagrams $\vec A$ such that $\Conc\circ\vec A$ has no lifting in $\cW$, however the algebra~$A$ constructed here has cardinality $\aleph_d$.

Gamps are also used to study congruence-preserving extensions. Denote by~$\cM_3$ the variety generated by the lattice of length two, with three atoms. We construct a lattice $A\in\cM_3$ of cardinality $\aleph_1$ with no congruence $n$-permutable, congruence-preserving extension, for each $n\geq 2$.
\end{abstract}
\maketitle

\section{Introduction}

For an algebra~$A$ we denote by $\Con A$ the lattice of all congruences of~$A$ under inclusion. Given $x,y\in A$, we denote by $\Theta_A(x,y)$ the smallest congruence of~$A$ that identifies $x$ and $y$, such a congruence is called \emph{principal}. A congruence is \emph{finitely generated} if it is a finite join of principal congruences. The lattice $\Con A$ is algebraic and the compact element of $\Con A$ are the finitely generated congruences.

The lattice $\Con A$ is determined by the \jzs\ $\Conc A$ of compact congruences of~$A$. In this paper we mostly refer to $\Conc A$ instead of $\Con A$. If $\Conc A$ is isomorphic to a \jzs\ $S$, we call~$A$ a \emph{lifting} of $S$.

Given a class of algebras $\cK$ we denote by $\Conc\cK$ the class of all \jzs s with a lifting in $\cK$. In general, even if $\cK$ is a variety of algebras, there is no good description of $\Conc\cK$. The negative solution to the congruence lattice problem (CLP) in \cite{CLP} is a good example of the difficulty to find such a description.

The study of CLP led to the following related questions. Fix two classes of algebras~$\cV$ and~$\cW$.

\begin{itemize}
\item[$\Qone$] Given $A\in\cV$, does there exist $B\in\cW$ such that $\Conc A\cong\Conc B$?

\item[$\Qtwo$] Given $A\in\cV$, does there exist a congruence-preserving extension $B\in\cW$ of~$A$?
\end{itemize}

A positive answer to $\Qone$ is equivalent to $\Conc\cV\subseteq\Conc\cW$. The ``containment defect'' of $\Conc\cV$ into $\Conc\cW$ is measured (cf. \cite{CLPSurv, G1}) by the \emph{critical point} between~$\cV$ and~$\cW$, defined as
\begin{align*}
 \crit{\cV}{\cW}=
  \begin{cases}
  \min\setm{\card S}{S\in(\Conc\cV)-(\Conc\cW)}, &\text{if $\Conc\cV\not\subseteq\Conc\cW$,}\\
  \infty, &\text{if $\Conc\cV\subseteq\Conc\cW$.}
 \end{cases}
\end{align*}
This critical point has been already studied, for different families of varieties of lattices, in \cite{Ploscica00, Ploscica03, G1}.

We now give an illustration of $\Qtwo$. Every countable locally finite lattice has a relatively complemented, congruence-preserving extension (cf. \cite{GLWe}). In particular every countable locally finite lattice has a congruence-permutable, congruence-preserving extension. However, in every non-distributive variety of lattices, the free lattice on~$\aleph_1$ generators has no congruence-permutable, congruence-preserving extension (cf. \cite[Chapter 5]{GiWe1}). A precise answer to $\Qtwo$ also depends on the cardinality of~$A$.

In order to study a similar problem, Pudl\'ak in \cite{Pudl} uses an approach based on liftings of diagrams. The assignment $A\mapsto\Conc A$ can be extended to a functor. This leads to the following questions:

\begin{itemize}
\item[$\Qonep$] Given a diagram $\vec A$ in~$\cV$, does there exist a diagram $\vec B$ in~$\cW$ such that $\Conc \circ\vec A\cong\Conc \circ \vec B$?

\item[$\Qtwop$] Given a diagram $\vec A$ in~$\cV$, does there exist a diagram $\vec B$ in~$\cW$ which is a congruence-preserving extension of $\vec A$?
\end{itemize}

The functor $\Conc$ preserves directed colimits, thus, in many cases, a positive answer for the finite case of $\Qonep$ implies a positive answer to $\Qone$.

\begin{proposition}\label{P:intro}
Assume that~$\cV$ and~$\cW$ are varieties of algebras. If~$\cV$ is locally finite and for every lattice-indexed diagram $\vec A$ of finite algebras in~$\cV$ there exists a diagram~$\vec B$ in~$\cW$ such that $\Conc\circ\vec A\cong \Conc\circ\vec B$, then $\Qone$ has a positive answer.
\end{proposition}

In this proposition, we consider infinite diagrams of finite algebras. However if~$\cW$ is finitely generated and congruence-distributive, a compactness argument makes it possible to restrict the assumptions to finite diagrams of finite algebras.

In order to study the converse of Proposition~\ref{P:intro}, we shall use the construction of \emph{condensate}, introduced in~\cite{G1}. This construction was introduced in order to turn diagram counterexamples to object counterexamples. We use it here to turn a diagram counterexample of $\Qonep$ to a counterexample of $\Qone$.

\begin{theorem}\label{T1:intro}
Assume that~$\cV$ and~$\cW$ are varieties of algebras. Let $P$ be a finite lattice. Let $\vec A$ be a $P$-indexed diagram in~$\cV$. If $\Conc\circ\vec A$ is not liftable in~$\cW$ then there is a condensate $A\in\cV$ of $\vec A$, such that $\Conc A$ is not liftable in~$\cW$.

Moreover, if~$\cW$ has a countable similarity type and all algebras of $\vec A$ are countable, then the condensate~$A$ can be chosen of cardinal~$\aleph_d$, where $d$ is the order-dimension of $P$. In particular, $\crit{\cV}{\cW}\le\aleph_d$.

If every algebra of $\vec A$ is finite and~$\cW$ is finitely generated and congruence-distributive, then~$A$ can be chosen of cardinal~$\aleph_{d-1}$, so $\crit{\cV}{\cW}\le\aleph_{d-1}$.
\end{theorem}

The cardinality bound, in case~$\cW$ is finitely generated and congruence-distributive, is optimal, in the following sense: There are finitely generated varieties~$\cV$ and~$\cW$ of lattices such that every countable \jzs\ liftable in~$\cV$ is liftable in~$\cW$ and there is a square-indexed diagram of \jzs s that has a lifting in~$\cV$ but no lifting in~$\cW$. In particular there is a \jzs\ of cardinal~$\aleph_1$ that is liftable in~$\cV$ but not liftable in~$\cW$, thus $\crit{\cV}{\cW}=\aleph_1$. This example appears in \cite[Section~8]{G1}.

Later, we generalized the condensate construction in \cite{GiWe1}, to a larger categorical context. The best bound of Theorem~\ref{T1:intro} is obtained in a more general case (cf. \cite[Theorem~4-9.2]{GiWe1}), namely if~$\cW$ is both \emph{congruence-proper} (cf. \cite[Definition 4-8.1]{GiWe1}) and locally finite, for example~$\cW$ is a finitely generated congruence-modular variety. Using the tools introduced in this paper, we give a new version (cf. Theorem~\ref{T:diagcrit}), we assume that~$\cW$ is congruence-proper and has finite similarity type.

This categorical version of condensate can also apply to turn a counterexample of $\Qtwop$ to a counterexample of $\Qtwo$. For example in \cite[Chapter~5]{GiWe1} we give a square~$\vec A$ of finite lattices, that has no congruence-permutable, congruence-preserving extension. A condensate of this square has cardinality~$\aleph_1$ and it has no congruence-permutable, congruence-preserving extension.

The largest part of this paper is the introduction of \emph{pregamps} and \emph{gamps}, it is a generalization of semilattice-metric spaces and semilattice-metric covers given in \cite[Chapter~5]{GiWe1}. The category of gamps of a variety~$\cV$ has properties similar to a finitely generated congruence-distributive variety.

A \emph{pregamp} is a triplet $\bA=(A,\delta,S)$, where~$A$ is a partial algebra, $S$ is a \jzs\ and $\delta\colon A^2\to S$ is a \emph{distance}, \emph{compatible with the operations}. A typical example of pregamp is $(A,\Theta_A,\Conc A)$, for an algebra~$A$. This generalizes to partial algebras the notion of a congruence.

A \emph{gamp} $\bA$ is a pregamp $(A,\delta,S)$ with a partial subalgebra $A^*$ of~$A$. There are many natural properties that a gamp can satisfy (cf. Section~\ref{S:LFP}), for example $\bA$ is \emph{full} if all operations with parameters in $A^*$ can be evaluated in~$A$. A \emph{morphism} of gamps is a morphism of partial algebras with a morphism of \jzs s satisfying a compatibility condition with the distances (cf. Definition~\ref{D:gamp}).

The class of all gamps (on a given type), with morphisms of gamps, forms a category. Denote by $\CG$ the forgetful functor from the category of gamps to the category of \jzs s. A \emph{partial lifting} of a diagram $\vec S$ of \jzs\ is a diagram $\vec\bB$ of gamps, with some additional properties, such that $\CG\circ\vec\bB\cong\vec S$.

The category of gamps has properties similar to locally finite, congruence-proper varieties. Let $S$ be a finite \jzs, let $\bB$ be a gamp such that $\CG(\bB)\cong S$, there are (arbitrary large) finite subgamps $\bB'$ of $\bB$ such that $\CG(\bB')\cong S$. There is no equivalent result for algebras: for example, the three-element chain is the congruence lattice of a modular lattice, but not the congruence lattice of any finite modular lattice.

Assume that~$\cV$ and $\cW$ are varieties of algebras. Let $P$ be a finite lattice of order-dimension $d$. Suppose that we find~$\vec A$ a $P$-indexed diagram of finite algebras in~$\cV$, such that $\Conc\circ\vec A$ has no partial lifting (with maybe some additional \emph{``locally finite properties''}, cf. Section~\ref{S:LFP}) in the category of gamps of $\cW$, then $\crit{\cV}{\cW}\le\aleph_{d-1}$. Hence we obtain the optimal bound, with no assumption on $\cW$. However there is no (known) proof that a diagram with no lifting has no partial lifting, but no counterexample has been found.

The \emph{dual} of a lattice~$L$ is the lattice $\dual{L}$ with reverse order. The \emph{dual} of a variety of lattices~$\cV$ is $\dual{\cV}$ the variety of all duals of lattices in~$\cV$. Let~$\cV$ and $\cW$ be varieties of lattices, if $\cV\subseteq\cW$ or $\cV\subseteq\dual{\cW}$, then $\Conc\cV\subseteq\Conc\cW$. In a sequel to the present paper (cf. \cite{G4}), we shall prove the following result.
\begin{theorem*}
Let~$\cV$ and~$\cW$ be varieties of lattices. If every simple lattice in~$\cW$ contains a prime interval, then one of the following statements holds:
\begin{enumerate}
\item $\crit{\cV}{\cW}\le\aleph_2$.
\item $\cV\subseteq\cW$.
\item $\cV\subseteq\dual{\cW}$.
\end{enumerate}
\end{theorem*}
The $\aleph_2$ bound is optimal, as there are varieties~$\cV$ and $\cW$ of lattices such that $\crit{\cV}{\cW}=\aleph_2$. Without the use of gamps, we would have obtained an upper bound~$\aleph_3$ instead of~$\aleph_2$.

The gamps can also be used to study congruence-preserving extensions. Denote by $\PGGL$ the functor that maps a gamp $(A^*,A,\delta,\widetilde A)$ to the pregamp $(A^*,\delta,\widetilde A)$; we also denote by $\PGA$ the functor that maps an algebra~$A$ to the pregamp $(A,\Theta_A,\Conc A)$. Let~$B$ be a congruence-preserving extension of an algebra~$A$, then $(A,B,\Theta_B,\Conc B)$ is a gamp. Similarly, let $\vec B=\famm{B_p,g_{p,q}}{p\le q\text{ in }P}$ be a congruence-preserving extension of a diagram $\vec A=\famm{A_p,f_{p,q}}{p\le q\text{ in }P}$, denote by $\bB_p=(A,B,\Theta_B,\Conc B)$ and $\bgg_{p,q}=(g_{p,q},\Conc g_{p,q})$, for all $p\le q$ in $P$, then $\vec\bB= \famm{\bB_p,\bgg_{p,q}}{p\le q\text{ in }P}$ is a diagram of gamps. Moreover, $\PGGL\circ\vec \bB=\PGA\circ\vec A$, up to the identification of $\Conc B$ and $\Conc A$.

In Section~\ref{S:unlift}, given $n\ge 2$, we construct a square $\vec A$ of finite lattices in $\cM_3$, such that the diagram $\vec A$ has no congruence $n$-permutable, congruence-preserving extension. Another condensate construction gives a result proved in \cite{Ploscica08}, namely the existence of a lattice $A\in\cM_3$ of cardinality $\aleph_2$ with no congruence $n$-permutable, congruence-preserving extension.

Hopefully, once again, the diagram $\vec A$ satisfies a stronger statement, there is no \emph{operational} (cf. Definition~\ref{D:perminlatt}) diagram $\vec\bB$ of \emph{congruence $n$-permutable} gamps of lattices such that $\PGGL\circ\vec\bB\cong \PGA\circ\vec A$. Using a condensate, we obtain a lattice $\cM_3$ of cardinality $\aleph_1$ with no congruence $n$-permutable, congruence-preserving extension.

\section{Basic Concepts}

We denote by~$0$ (resp., $1$) the least (resp. largest) element of a poset if it exists. We denote by $\two=\set{0,1}$ the two-element lattice, or poset (i.e., partially ordered set), or \jzs, depending of the context. Given an algebra~$A$, we denote by $\zero_A$ the identity congruence of~$A$.

Given subsets $P$ and $Q$ of a poset $R$, we set
\[
P\dnw Q=\setm{p\in P}{(\exists q\in Q)(p\le q)}.
\]
If $Q=\set{q}$, we simply write $P\dnw q$ instead of $P\dnw \set{q}$.

Let~$\cV$ be a variety of algebras, let $\kappa$ be a cardinal, we denote by $F_\cV(\kappa)$ the free algebras in~$\cV$ with $\kappa$ generators. Given an algebra~$A$ we denote by $\Var A$ the variety of algebras generated by~$A$. If~$A$ is a lattice we also denote by $\Var^{0,1}A$ the variety of bounded lattices generated by~$A$. We denote by $\cL$ the variety of lattices.

We denote the \emph{range} of a function $f\colon X\to Y$ by $\rng f=\setm{f(x)}{x\in X}$. We use basic set-theoretical notation, for example~$\omega$ is the first infinite ordinal, and also the set of all nonnegative integers; furthermore, $n=\set{0,1,\dots,n-1}$ for every nonnegative integer~$n$. By ``countable'' we will always mean ``at most countable''.

Let $X$, $I$ be sets, we often denote $\vec x=(x_i)_{i\in I}$ an element of $X^I$. In particular, for~$n<\omega$, we denote by~$\vec x=(x_0,\dots,x_{n-1})$ an~$n$-tuple of $X$. If $f\colon Y\to Z$ is a function, where $Y\subseteq X$, we denote $f(\vec x)=(f(x_0),\dots,f(x_{n-1}))$ whenever it is defined. Similarly, if $f\colon Y\to Z$ is a function, where $Y\subseteq X^n$, we denote $f(\vec x)=f(x_0,\dots,x_{n-1})$ whenever it is defined. We also write $f(\vec x,\vec y)=f(x_0,\dots,x_{m-1},y_0,\dots,y_{n-1})$ in case $\vec x=(x_0,\dots,x_{m-1})$ and $\vec y=(y_0,\dots,y_{n-1})$, and so on.

For example, let~$A$ and~$B$ be algebras of the same similarity type. Let $\ell$ be an $n$-ary operation. Let $f\colon A\to B$ a map. The map $f$ is \emph{compatible} with $\ell$ if $f(\ell(\vec x))=\ell(f(\vec x))$ for every $n$-tuple~$\vec x$ of $X$. Let $m\le n\le\omega$. Let~$\vec X$ be an $n$-tuple of~$X$, we denote by $\vec x\res m$ the $m$-tuple $(x_k)_{k<m}$.

If~$X$ is a set and $\theta$ is an equivalence relation on~$X$, we denote by $X/\theta$ the set of all equivalence classes of $\theta$. Given $x\in X$ we denote by $x/\theta$ the equivalence class of $\theta$ containing $x$. Given an $n$-tuple~$\vec x$ of~$X$, we denote~$\vec x/\theta=(x_0/\theta,\dots,x_{n-1}/\theta)$. Given $Y\subseteq X$, we set $Y/\theta=\setm{x/\theta}{x\in Y}$.

Let $n\ge 2$ an integer. An algebra~$A$ is \emph{congruence $n$-permutable} if the following equality holds:
\[
\underbrace{\alpha\circ\beta\circ\alpha\circ\dots}_{n\text{ times}}=\underbrace{\beta\circ\alpha\circ\beta\circ\dots}_{n\text{ times}},\text{ for all $\alpha,\beta\in\Con A$}.
\]
If $n=2$ we say that~$A$ is \emph{congruence-permutable} instead of \emph{congruence~$2$-permutable}. 

The following statement is folklore.

\begin{proposition}\label{P:CP-PR}
Let~$A$ be an algebra, let $n\ge 2$ an integer. The following conditions are equivalent:
\begin{enumerate}
\item The algebra~$A$ is congruence $n$-permutable.
\item For all $x_0,x_1,\dots,x_{n}\in A$, there are $x_0=y_0,y_1,\dots,y_{n}=x_{n}\in A$ such that the following containments hold:
\begin{align*}
\Theta_A(y_k,y_{k+1})&\subseteq\bigvee\famm{\Theta_A(x_{i},x_{i+1})}{i< n\text{ even}},&&\text{for all $k< n$ odd},\\
\Theta_A(y_k,y_{k+1})&\subseteq\bigvee\famm{\Theta_A(x_{i},x_{i+1})}{i< n\text{ odd}},&&\text{for all $k< n$ even}.
\end{align*}
\end{enumerate}
\end{proposition}

Let $n\ge 2$. The class of all congruence $n$-permutable algebras of a given similarity type is closed under directed colimits and quotients. Moreover the class of congruence $n$-permutable algebras of a congruence-distributive variety is also closed under finite products (the latter statement is known not to extend to arbitrary algebras).

\section{Semilattices}

In this section we give some well-known facts about \jzs s. Most notions and results will have later a generalization involving \pregamps\ and gamps.

\begin{proposition}\label{P:semilatiso}
Let $S,T$ be \jzs s, let $X$ be a set, and let $f\colon X\to S$ and $g\colon X\to T$ be maps. Assume that for every $x\in X$, for every positive integer $n$, and for every $n$-tuple $\vec y$ of $X$ the following implication holds:
\begin{equation}\label{E:semilatiso}
f(x)\le\bigvee_{k<n}f(y_k)\quad \Longrightarrow\quad g(x)\le\bigvee_{k<n}g(y_k).
\end{equation}
If~$S$ is join-generated by $f(X)$, then there exists a unique \jzh\ $\phi\colon S\to T$ such that $\phi(f(x))=g(x)$ for each $x\in X$.

If the converse of~\eqref{E:semilatiso} also holds and $g(X)$ also join-generates $T$, then $\phi$ is an isomorphism.
\end{proposition}

\begin{definition}\label{D:idealsem}
An \emph{ideal} of a \jzs~$S$ is a lower subset~$I$ of~$S$ such that $0\in I$ and $u\vee v\in I$ for all $u,v\in I$. We denote by $\Id S$ the lattice of ideals of~$S$.

Let $\phi\colon S\to T$ be a \jzh. The \emph{$0$-kernel} of~$\phi$ is $\ker_0 \phi=\setm{a\in S}{\phi(a)=0}$; it is an ideal of~$S$. We say that $\phi$ \emph{separates zero} if $\ker_0\phi=\set{0}$.

Let $P$ be a poset, let $\vec S=\famm{S_p,\phi_{p,q}}{p\le q\text{ in }P}$ be a diagram of \jzs s. An \emph{ideal} of $\vec S$ is a family $(I_p)_{p\in P}$ such that $I_p$ is an ideal of $S_p$ and $\phi_{p,q}(I_p)\subseteq I_q$ for all $p\le q$ in $P$.

Let $\vec \phi=(\phi_p)_{p\in P}\colon \vec S\to \vec T$ be a natural transformation of $P$-indexed diagrams of \jzs s. The \emph{$0$-kernel} of~$\vec\phi$ is $\ker_0\vec\phi=(\ker_0\phi_p)_{p\in P}$, it is an ideal of $\vec S$.
\end{definition}

\begin{lemma}\label{L:quotSem}
Let~$S$ be a \jzs, let $I\in\Id S$. Put:
\[\theta_I=\setm{(x,y)\in S^2}{(\exists u\in I)(x\vee u=y\vee u)}.\]
The relation $\theta_I$ is a congruence of~$S$.
\end{lemma}

\begin{notation}
We denote by $S/I$ the \jzs\ $S/\theta_I$, where $\theta_I$ is the congruence defined in Lemma~\ref{L:quotSem}. Given $a\in S$, we denote by $a/I$ the equivalent class of $a$ for~$\theta_I$. The \jzh\ $\phi\colon S\to S/I$, $a\mapsto a/I$ is \emph{the canonical projection}. Notice that $\ker_0\phi=I$.

If $I=\set{0}$, we identify $S/I$ and~$S$.
\end{notation}

\begin{lemma}\label{L:quotMorphSem}
Let $\phi\colon S\to T$ be a \jzh, and let $I\in\Id S$ and $J\in\Id T$ such that $\phi(I)\subseteq J$. There exists a unique map $\psi\colon S/I\to T/J$ such that $\psi(a/I)=\phi(a)/J$ for each $a\in S$. Moreover, $\psi$ is a \jzh.
\end{lemma}

\begin{notation}
We say that $\phi$ \emph{induces} the \jzh\ $\psi\colon S/I\to T/J$ in Lemma~\ref{L:quotMorphSem}.
\end{notation}

\begin{lemma}\label{L:quotDiagSem}
Let $P$ be a poset, let $\vec S=\famm{S_p,\phi_{p,q}}{p\le q\text{ in }P}$ be a diagram of \jzs s, and let $\vec I$ be an ideal of $\vec S$. Denote by $\psi_{p,q}\colon S_p/I_p\to S_q/I_q$ the \jzh\ induced by $\phi_{p,q}$, then $\famm{S_p/I_p,\psi_{p,q}}{p\le q\text{ in }P}$ is a diagram of \jzs s.
\end{lemma}

\begin{notation}
We denote by $\vec S/\vec I$ the diagram $\famm{S_p/I_p,\psi_{p,q}}{p\le q\text{ in }P}$ introduced in Lemma~\ref{L:quotDiagSem}.
\end{notation}

\begin{lemma}\label{L:quotMorphSemDiag}
Let $P$ be a poset, let $\vec \phi\colon \vec S\to \vec T$ be a natural transformation of $P$-indexed diagrams of \jzs s, let $\vec I\in\Id \vec S$ and $\vec J\in\Id \vec T$ such that $\phi_p(I_p)\subseteq J_p$ for all $p\in P$. Denote by $\psi_p\colon S_p/I_p\to T_p/J_p$ the \jzh\ induced by $\phi_p$. Then $\vec\psi$ is a natural transformation from $\vec S/\vec I$ to $\vec T/\vec J$.
\end{lemma}

\begin{notation}
We say that $\vec\phi$ \emph{induces} $\vec \psi\colon \vec S/\vec I\to \vec T/\vec J$, the natural transformation defined in Lemma~\ref{L:quotMorphSemDiag}.
\end{notation}

\begin{definition}
A \jzh\ $\phi\colon S\to T$ is \emph{ideal-induced} if $\phi$ is surjective and for all $x,y\in S$ with $\phi(x)=\phi(y)$ there exists $z\in S$ such that $x\vee z=y\vee z$ and $\phi(z)=0$.

Let $P$ be a poset, let $\vec S=\famm{S_p,\phi_{p,q}}{p\le q\text{ in }P}$ and $\vec T=\famm{T_p,\psi_{p,q}}{p\le q\text{ in }P}$ be $P$-indexed diagrams of \jzs s. A natural transformation $\vec\pi=(\pi_p)_{p\in P}\colon\vec S\to\vec T$ is \emph{ideal-induced} if $\pi_p$ is ideal-induced for each $p\in P$.
\end{definition}

\begin{remark}
Let $I$ be an ideal of a \jzs\ $A$, denote by $\pi\colon A\to A/I$ the canonical projection, then $\pi$ is ideal-induced.
\end{remark}

The next lemmas give a characterization of ideal-induced \jzh s.

\begin{lemma}\label{L:IIsem}
Let $\phi\colon S\to T$ be a \jzh. The following statements are equivalent
\begin{enumerate}
\item $\phi$ is ideal-induced.
\item The \jzh\ $\psi\colon S/\ker_0\phi\to T$ induced by $\phi$ is an isomorphism.
\end{enumerate}
\end{lemma}

The following lemma expresses that, given a diagram $\vec S$ of \jzs s, the colimits of quotients of $\vec S$ are the quotients of the colimits of $\vec S$.

\begin{lemma}\label{L:colimitquot2}
Let $P$ be a directed poset, let~$\vec S=\famm{S_p,\phi_{p,q}}{p\le q\text{ in }P}$ be a $P$-indexed diagram in $\SEM$, and let $\famm{S,\phi_p}{p\in P}=\varinjlim \vec S$ be a directed colimit cocone in $\SEM$. The following statements hold:
\begin{enumerate}
\item Let $\vec I$ be an ideal of $\vec S$. Then $I=\bigcup_{p\in P}\phi_p(I_p)$ is an ideal of~$S$. Moreover, denote by $\psi_p\colon S_p/I_p\to S/I$ the \jzh\ induced by $\phi_p$, for each $p\in P$. The following is a directed colimit cocone:
\[
\famm{S/I,\psi_p}{p\in P}=\varinjlim \vec S/\vec I\quad\text{ in $\SEM$}.
\]
\item Let $I\in\Id S$. Put $I_p=\phi_p^{-1}(I)$ for each $p\in P$. Then~$\vec I=(I_p)_{p\in P}$ is an ideal of~$\vec S$, moreover $I=\bigcup_{p\in P} \phi_p(I_p)$.
\end{enumerate}
\end{lemma}

\begin{lemma}\label{L:quotDiagAlgAndSem}
Let $\pi\colon A\tosurj B$ be a surjective morphism of algebras. The \jzh\ $\Conc\pi$ is ideal-induced. Moreover, $\ker_0(\Conc\pi)=(\Conc A)\dnw\ker\pi$.
\end{lemma}

\begin{proposition}\label{P:IdeIdu}
Let $S$ and~$T$ be \jzs s with~$T$ finite, let $\phi\colon S\to T$ be an ideal-induced \jzh, and let $X\subseteq S$ finite. There exists a finite \jz-subsemilattice $S'$ of~$S$ such that $X\subseteq S'$ and $\phi\res S'\colon S'\to T$ is ideal-induced.
\end{proposition}

\begin{proof}
As $\phi$ is surjective and~$X$ is finite, there exists a finite \jz-subsemilattice~$Y$ of~$S$ such that $X\subseteq Y$ and $\phi(Y)=T$. Given $x,y\in Y$ with $\phi(x)=\phi(y)$ we fix $u_{x,y}\in S$ such that $\phi(u_{x,y})=0$ and $x\vee u_{x,y}=y\vee u_{x,y}$. Let~$U$ be the \jz-subsemilattice of~$S$ generated by $\setm{u_{x,y}}{x,y\in Y\text{ and }\phi(x)=\phi(y)}$. As $\phi(u)=0$ for all generators, $\phi(u)=0$ for each $u\in U$.

Let $S'$ be the \jz-subsemilattice of~$S$ generated by $Y\cup U$. As $S'$ is finitely generated, it is finite. As $Y\subseteq S'$, $\phi(S')=T$. Let $a,b\in S'$ such that $\phi(a)=\phi(b)$. There exist $x,y\in Y$ and $u,v\in U$ such that $a=x\vee u$ and $b=y\vee v$, thus $\phi(a)=\phi(x\vee u)=\phi(x)\vee\phi(u)=\phi(x)$. Similarly, $\phi(b)=\phi(y)$, hence $\phi(x)=\phi(y)$, moreover $x,y\in Y$, so $u_{x,y}\in U$. The element $w=u\vee v\vee u_{x,y}$ belongs to~$U$, hence $\phi(w)=0$, moreover $w\in S'$. {}From $x\vee u_{x,y}=y\vee u_{x,y}$ it follows that $a\vee w=b\vee w$. Therefore, $\phi\res S'$ is ideal-induced.
\end{proof}

\section{Partial algebras}\label{S:Palg}

In this section we introduce a few basic properties of partial algebras. We fix a similarity type~$\sL$. Given $\ell\in\sL$ we denote by $\ari(\ell)$ the arity of $\ell$.

\begin{definition}
A \emph{partial algebra}~$A$ is a set (the \emph{universe} of the partial algebra), given with a set $D_{\ell}=\Def_{\ell}(A)\subseteq A^{\ari(\ell)}$ and a map $\ell^A\colon D_{\ell}\to A$ called a \emph{partial operation}, for each $\ell\in\sL$.

Let $\ell\in\sL$ be an $n$-ary operation. If~$\vec x\in\Def_{\ell}(A)$ we say that $\ell^A(\vec x)$ is \emph{defined in~$A$}. We generalize this notion to terms in the usual way. For example, given binary operations~$\ell_1$ and~$\ell_2$ of a partial algebra~$A$ and $x,y,z\in A$, $\ell_1^A(\ell_2^A(x,y),\ell_1^A(y,z))$ is defined in~$A$ if and only if $(x,y)\in\Def_{\ell_2}(A)$, $(y,z)\in\Def_{\ell_1}(A)$, and $(\ell_2^A(x,y),\ell_1^A(y,z))\in\Def_{\ell_1}(A)$.

Given a term $t$, we denote by $\Def_t(A)$ the set of all tuples~$\vec x$ of~$A$ such that $t(\vec x)$ is defined in~$A$.
\end{definition}

We denote $\ell(\vec x)$ instead of $\ell^A(\vec x)$ when there is no ambiguity. Any algebra~$A$ has a natural structure of partial algebra with $\Def_{\ell}(A)=A^{\ari(\ell)}$ for each $\ell\in\sL$.

\begin{definition}
Let~$A$,~$B$ be partial algebras. A \emph{morphism of partial algebras} is a map $f\colon A\to B$ such that $f(\vec x)\in\Def_{\ell}(B)$ and $\ell(f(\vec x))=f(\ell(\vec x))$, for all $\ell\in\sL$ and all~$\vec x\in\Def_{\ell}(A)$.

The \emph{category of partial algebras}, denoted by $\PALG_{\sL}$, is the category in which the objects are the partial algebras and the arrows are the above-mentioned morphisms of partial algebras.

A morphism $f\colon A\to B$ of partial algebras is \emph{strong} if $(f(A))^{\ari(\ell)}\subseteq\Def_{\ell}(B)$ for each $\ell\in\sL$.

A partial algebra~$A$ is \emph{finite} if its universe is finite.
\end{definition}

\begin{remark}\label{R:isopalg}
A morphism $f\colon A\to B$ of partial algebras is an isomorphism if and only if the following conditions are both satisfied
\begin{enumerate}
\item The map~$f$ is bijective.
\item If $\ell(f(\vec x))$ is defined in~$B$ then $\ell(\vec x)$ is defined in~$A$, for each operation $\ell\in\sL$ and each tuple~$\vec x$ of~$A$.
\end{enumerate}

We remind the reader that the converse of $(2)$ is always true.
\end{remark}

\begin{definition}\label{D:StrPartAlg}
Given a partial algebra~$A$, a \emph{partial subalgebra}~$B$ of~$A$ is a subset~$B$ of~$A$ endowed with a structure of partial algebra such that $\Def_{\ell}(B)\subseteq\Def_{\ell}(A)$ and $\ell^A(\vec x) = \ell^B(\vec x)$ for all $\ell\in\sL$ and all~$\vec x\in\Def_{\ell}(B)$. The inclusion map from~$A$ into~$B$ is a morphism of partial algebras called \emph{the inclusion morphism}.

A partial subalgebra~$B$ of~$A$ is \emph{full} if whenever~$\ell\in\sL$ and~$\vec x\in B^{\ari(\ell)}$ are such that $\ell^A(\vec x)$ is defined and belongs to~$B$, then $\ell(\vec x)$ is defined in~$B$. It is equivalent to the following equality:
\[
\Def_{\ell}(B)=\setm{\vec x\in\Def_{\ell}(A)\cap B^{\ari(\ell)}}{\ell(\vec x)\in B},\quad\text{ for each $\ell\in\sL$.}
\]

A partial subalgebra~$B$ of~$A$ is \emph{strong} if the inclusion map is a strong morphism, that is,~$B^{\ari(\ell)}\subseteq\Def_{\ell}(A)$ for each $\ell\in\sL$.

An \emph{embedding} of partial algebras is a one-to-one morphism of partial algebras.
\end{definition}

\begin{notation}\label{N:imagepartsubalgebra}
Let $f\colon A\to B$ be a morphism of partial algebras, let~$X$ be a partial subalgebra of~$A$. The set $f(X)$ can be endowed with a natural structure of partial algebra, by setting $\Def_{\ell}(f(X))=f(\Def_{\ell}(X))=\setm{f(\vec x)}{\vec x\in\Def_{\ell}(X)}$, for each $\ell\in\sL$. Similarly, let~$Y$ be a partial subalgebra of~$B$. The set $f^{-1}(Y)$ can be endowed with a natural structure of partial algebra, by setting $\Def_{\ell}(f^{-1}(Y))=f^{-1}(\Def_{\ell}(Y))=\setm{\vec x\in A}{f(\vec x)\in\Def_{\ell}(Y)}$, for each $\ell\in\sL$.
\end{notation}

\begin{remark*}
Let $f\colon A\to B$ and $g\colon B\to C$ be morphisms of partial algebras, let $X$ a sub-partial algebra of~$A$, then $(g\circ f)(X)=g(f(X))$ as partial algebras. Let $Z$ be a partial subalgebra of $X$, then $(g\circ f)^{-1}(Z)=f^{-1}(g^{-1}(Z))$ as partial algebras.

Let $f\colon A\to B$ be a morphism of partial algebras, let $X$ be a partial subalgebra of~$A$. Then $X$ is a partial subalgebra of $f^{-1}(f(X))$. In particular $f^{-1}(f(A))=A$ as partial algebras. Let $Y$ be a partial subalgebra of~$B$, then $f(f^{-1}(Y))$ is a partial subalgebra of $Y$.

If~$\sL$ is infinite, then there are a finite partial algebra~$A$ (even with one element) and an infinite chain of partial subalgebras of~$A$ with union~$A$. In particular,~$A$ is not finitely presented in the category $\PALG_{\sL}$.

A morphism $f\colon A\to B$ of partial algebras is strong if and only if $f(A)$ is a strong partial subalgebra of~$B$.
\end{remark*}

\begin{lemma}\label{L:caractiso}
An embedding $f\colon A\to B$ of partial algebras is an isomorphism if and only if $f(A)=B$ as partial algebras.
\end{lemma}

\begin{proof}
Assume that $f$ is an isomorphism, let $g$ be its inverse. Notice that $f(A)$ is a partial subalgebra of~$B$ and $B=f(g(B))$ is a partial subalgebra of $f(A)$, therefore $B=f(A)$ as partial algebras.

Conversely, assume that $B=f(A)$ as partial algebras. Then $f$ is surjective, moreover $f$ is an embedding, so $f$ is a bijection. Let $g=f^{-1}$ in $\SET$. Let $\vec y\in\Def_{\ell}(B)$. As $B=f(A)$, there exists $\vec x\in\Def_{\ell}(A)$ such that $f(\vec x)=\vec y$, thus $g(\vec y)=\vec x\in\Def_{\ell}(A)$. Moreover $g(\ell(\vec y))=g(\ell(f(\vec x)))=g(f(\ell(\vec x)))=\ell(\vec x)=\ell(g(\vec y)).$
\end{proof}

\begin{notation}\label{N:gen}
Let~$A$ be a partial algebra, let~$X$ be a subset of~$A$. We define inductively, for each $n<\omega$,
\begin{align*}
\gen{X}^0_A&=X\cup\setm{c}{\text{$c$ is a constant of~$\sL$}}\\
\gen{X}^{n+1}_A&=\gen{X}^{n}_A
\cup\big\{\ell(\vec x)\mid \ell\in\sL,\ \vec x\in\Def_{\ell}(A),\ \vec x\text{ is an $\ari(\ell)$-tuple of}\gen{X}^{n}_A\big\}
\end{align*}
We endow $\gen{X}^n_A$ with the induced structure of full partial subalgebra of~$A$. If~$\sL$ and~$X$ are both finite, then $\gen{X}^n_A$ is finite for each $n<\omega$. If~$A$ is understood, we shall simply denote this partial algebra by $\gen{X}^n$.
\end{notation}

\begin{definition}
A partial algebra~$A$ \emph{satisfies an identity $t_1=t_2$} if $t_1(\vec x)=t_2(\vec x)$ for each tuple~$\vec x$ of~$A$ such that both $t_1(\vec x)$ and $t_2(\vec x)$ are defined in~$A$. Otherwise we say that~$A$ \emph{fails} $t_1=t_2$.

Let~$\cV$ be a variety of algebras, a partial algebra~$A$ is \emph{a partial algebra of~$\cV$} if~$A$ satisfies all identities of~$\cV$.
\end{definition}

\begin{remark*}
Let~$A$ be a partial algebra, let $\ell\in\sL$. If $\Def_{\ell}(A)=\emptyset$ then~$A$ satisfies $\ell(\vec x)=y$, vacuously.

If~$A$ fails $t_1=t_2$, then there exists a tuple~$\vec x$ of~$A$ such that $t_1(\vec x)$ and $t_2(\vec x)$ are both defined and $t_1(\vec x)\not=t_2(\vec x)$.
\end{remark*}

\begin{lemma}\label{L:PALGhasLimit}
The category $\PALG_{\sL}$ has all directed colimits. Moreover, given a directed poset $P$, a $P$-indexed diagram~$\vec A=\famm{A_p,f_{p,q}}{p\le q\text{ in }P}$ in $\PALG_{\sL}$, and a directed colimit cocone:
\begin{equation}\label{E:PALGL-col-set}
\famm{A,f_p}{p\in P}=\varinjlim \famm{A_p,f_{p,q}}{p\le q\text{ in }P},\quad\text{in $\SET$,}
\end{equation}
the set~$A$ can be uniquely endowed with a structure of partial algebra such that:
\begin{itemize}
\item $\Def_{\ell}(A)=\setm{f_p(\vec x)}{p\in P\text{ and }\vec x\in \Def_{\ell}(A_p)}$, for each $\ell\in\sL$;
\item $\ell(f_p(\vec x))=f_p(\ell(\vec x))$ for each $p\in P$, all $\ell\in\sL$, and all~$\vec x\in\Def_{\ell}(A_p)$.
\end{itemize}
Moreover, if~$A$ is endowed with this structure of partial algebra, the following statements hold:
\begin{enumerate}
\item $\famm{A,f_p}{p\in P}=\varinjlim\famm{A_p,f_{p,q}}{p\le q\text{ in }P}$ in $\PALG_{\sL}$.
\item Assume that for each $\ell\in\sL$, each $p\in P$, and each $\ari(\ell)$-tuple~$\vec x$ of~$A_p$ there exists $q\ge p$ such that $f_{p,q}(\vec x)\in\Def_{\ell}(A_q)$. Then~$A$ is an algebra, that is, $\Def_{\ell}(A)=A^{\ari(\ell)}$ for each $\ell\in\sL$.
\item If $P$ has no maximal element and $f_{p,q}$ is strong for all $p<q$ in $P$, then~$A$ is an algebra.
\item $\Def_t(A)=\setm{f_p(\vec x)}{p\in P\text{ and }\vec x\in \Def_{t}(A_p)}$ for each term $t$ of~$\sL$.
\item Let $t_1=t_2$ be an identity. If $A_p$ satisfies $t_1=t_2$ for all $p\in P$, then~$A$ satisfies $t_1=t_2$.
\end{enumerate}
\end{lemma}

\begin{proof}
Put $\Def_{\ell}(A)=\setm{f_p(\vec x)}{p\in P\text{ and }\vec x\in \Def_{\ell}(A_p)}$, for each $\ell\in\sL$.

Let $\ell\in\sL$, let~$\vec x\in\Def_{\ell}(A)$. There exist $p\in P$ and~$\vec y\in \Def_{\ell}(A_p)$ such that~$\vec x=f_p(\vec y)$. We first show that $f_p(\ell(\vec y))$ does not depend on the choice of $p$ and~$\vec y$. Let $q\in P$ and~$\vec z\in \Def_{\ell}(A_q)$ such that~$\vec x=f_q(\vec z)$. As $f_p(\vec y)=\vec x=f_q(\vec z)$, it follows from~\eqref{E:PALGL-col-set} that there exists $r\ge p,q$ such that $f_{p,r}(\vec y)=f_{q,r}(\vec z)$. Therefore the following equalities hold:
\[
f_p(\ell(\vec y))=f_r(f_{p,r}(\ell(\vec y)))=f_r(\ell(f_{p,r}(\vec y)))=f_r(\ell(f_{q,r}(\vec z)))=f_r(f_{q,r}(\ell(\vec z)))=f_q(\ell(\vec z)).
\]
Hence $\ell(f_p(\vec y))=f_p(\ell(\vec y))$ for all $p\in P$ and all~$\vec y\in\Def_{\ell}(A_p)$ uniquely define a partial operation $\ell\colon \Def_{\ell}(A)\to A$. Moreover $f_p$ is a morphism of partial algebras for each $p\in P$.

Let $\famm{B,g_p}{p\in P}$ be a cocone over $\famm{A_p,f_{p,q}}{p\le q\text{ in }P}$ in $\PALG_{\sL}$. In particular, it is a cocone in $\SET$, so there exists a unique map $h\colon A\to B$ such that $h\circ f_p=g_p$ for each $p\in P$. Let $\ell\in \sL$, let~$\vec x\in\Def_{\ell}(A)$. There exist $p\in P$ and~$\vec y\in\Def_{\ell}(A_p)$ such that~$\vec x=f_p(\vec y)$, thus $h(\vec x)=h(f_p(\vec y))=g_p(\vec y)$. As~$g_p$ is a morphism of partial algebras and~$\vec y\in\Def_{\ell}(A_p)$, we obtain that $h(\vec x)\in\Def_{\ell}(B)$. Moreover the following equalities hold:
\[
\ell(h(\vec x))=\ell(g_p(\vec y))=g_p(\ell(\vec y))=h(f_p(\ell(\vec y)))=h(\ell(f_p(\vec y)))=h(\ell(\vec x)).
\]
Hence $h$ is a morphism of partial algebras. Therefore:
\[
\famm{A,f_p}{p\in P}=\varinjlim\famm{A_p,f_{p,q}}{p\le q\text{ in }P}\text{\quad in $\PALG_{\sL}$.}
\]

Assume that for each $\ell\in\sL$, for all $p\in P$, and for all $\ari(\ell)$-tuples~$\vec x$ of~$A_p$, there exists $q\ge p$ such that $f_{p,q}(\vec x)\in\Def_{\ell}(A_q)$.

Let $\ell\in\sL$, let~$\vec x$ be an $\ari(\ell)$-tuple of~$A$. There exist $p\in P$ and a tuple~$\vec y$ of~$A_p$ such that~$\vec x=f_p(\vec y)$. Let $q\ge p$ such that $f_{p,q}(\vec y)\in\Def_{\ell}(A_q)$. It follows that~$\vec x=f_p(\vec y)=f_q(f_{p,q}(\vec y))$ belongs to $\Def_{\ell}(A)$. Therefore~$A$ is an algebra.

The statement $(3)$ follows directly from $(2)$. The statement $(4)$ is proved by a straightforward induction on terms, and $(5)$ is an easy consequence of $(4)$.
\end{proof}

\section{Pregamps}\label{S:Pregamp}

A \emph{\pregamp} is a partial algebra endowed with a semilattice-valued ``distance'' (cf. $(1)$-$(3)$) compatible with all operations of~$A$ (cf. $(4)$). It is a generalization of the notion of \emph{semilattice-metric space} defined in \cite[Section 5-1]{GiWe1}.

\begin{definition}\label{D:distance}
Let~$A$ be a partial algebra, let~$S$ be a \jzs. A \emph{$S$-valued partial algebra distance} on~$A$ is a map $\delta\colon A^2\to S$ such that:
\begin{enumerate}
\item $\delta(x,y)=0$ if and only if $x=y$, for all $x,y\in A$.
\item $\delta(x,y)=\delta(y,x)$, for all $x,y\in A$.
\item $\delta(x,y)\le \delta(x,z)\vee \delta(z,y)$, for all $x,y,z\in A$.
\item $\delta(\ell(\vec x),\ell(\vec y))\le\bigvee_{k<\ari(\ell)}\delta(x_k,y_k)$, for all $\ell\in\sL$ and all~$\vec x,\vec y\in\Def_{\ell}(A)$.
\end{enumerate}

Then we say that $\bA=(A,\delta,S)$ is a \emph{\pregamp}. We shall generally write $\delta_{\bA}=\delta$ and~$\widetilde{A}=S$.

The \pregamp\ is \emph{distance-generated} if it satisfies the following additional property:
\begin{enumerate}
\item[$(5)$]~$S$ is join-generated by $\delta_{\bA}(A^2)$. That is, for all $\alpha\in S$ there are $n\ge 0$ and $n$-tuples $\vec x,\vec y$ of~$A$ such that $\alpha=\bigvee_{k<n}\delta_{\bA}(x_k,y_k)$.
\end{enumerate}
\end{definition}

\begin{example}
Let~$A$ be an algebra. We remind the reader that $\Theta_A(x,y)$ denotes the smallest congruence that identifies $x$ and $y$, for all $x,y\in A$. This defines a distance $\Theta_A\colon A^2\to\Conc A$. Moreover, $(A,\Theta_A,\Conc A)$ is a distance-generated \pregamp.
\end{example}

A straightforward induction argument on the length of the term~$t$ yields the following lemma.

\begin{lemma}\label{L:DistCompTerms}
Let~$\bA$ be a \pregamp, let $t$ be an $n$-ary term, and let~$\vec x,\vec y\in\Def_t(A)$. The following inequality holds:
\begin{equation*}
\delta_{\bA}(t(\vec x),t(\vec y))\le\bigvee_{k<n}\delta_{\bA}(x_k,y_k).
\end{equation*}
We say that~$\delta_{\bA}$ and $t$ are \emph{compatible}.
\end{lemma}

\begin{definition}\label{D:pregampcategoryandfunctor}
Let~$\bA$ and~$\bB$ be \pregamps. A \emph{morphism} from~$\bA$ to~$\bB$ is an ordered pair $\bff=(f,\widetilde{f})$ such that $f\colon A\to B$ is a morphism of partial algebras, $\widetilde{f}\colon \widetilde{A}\to \widetilde{B}$ is a \jzh, and $\delta_{\bB}(f(x),f(y))=\widetilde{f}(\delta_{\bA}(x,y))$ for all $x,y\in A$.

Given morphisms $\bff\colon\bA\to\bB$ and $\bgg\colon\bB\to\bC$ of \pregamps, the pair $\bgg\circ\bff=(g\circ f,\widetilde{g}\circ \widetilde{f})$ is a morphism from~$\bA$ to $\bC$.

We denote by $\MPALG_{\sL}$ the category of \pregamps\ with the morphisms defined above.

We denote by $\PGA$ the functor from the category of~$\sL$-algebras to $\MPALG_{\sL}$ that maps an algebra~$A$ to $(A,\Theta_A,\Conc A)$, and a morphism of algebras $f$ to $(f,\Conc f)$. We denote by $\CPG$ the functor from $\MPALG_{\sL}$ to $\SEM$ that maps a \pregamp~$\bA$ to $\widetilde A$, and maps a morphism of \pregamps\ $\bff\colon\bA\to\bB$ to the \jzh\ $\widetilde f$.
\end{definition}

\begin{remark}\label{R:isopregamp}
A morphism $\bff\colon\bA\to\bB$ of \pregamps\ is an isomorphism if and only if~$f$ is an isomorphism of partial algebras and~$\widetilde f$ is an isomorphism of \jzs s.

Notice that $\CPG\circ\PGA=\Conc$.
\end{remark}

We leave to the reader the straightforward proof of the following lemma.

\begin{lemma}\label{L:MPALGhasLimit}
The category $\MPALG_{\sL}$ has all directed colimits. Moreover, given a directed poset $P$, a $P$-indexed diagram~$\vec \bA=\famm{\bA_p,\bff_{p,q}}{p\le q\text{ in }P}$ in $\MPALG_{\sL}$, a directed colimit cocone $\famm{A,f_p}{p\in P}=\varinjlim \famm{A_p,f_{p,q}}{p\le q\text{ in }P}$ in $\PALG_{\sL}$, and a directed colimit cocone $\famm{\widetilde{A},\widetilde{f}_p}{p\in P}=\varinjlim \famm{\widetilde{A}_p,\widetilde{f}_{p,q}}{p\le q\text{ in }P}$ in $\SEM$, there exists a unique~$\widetilde{A}$-valued partial algebra distance~$\delta_{\bA}$ on~$A$ such that $\delta_{\bA}(f_p(x),f_p(y))=\widetilde{f}_p(\delta_{\bA_p}(x,y))$ for all $p\in P$ and all $x,y\in A_p$.

Furthermore $\bA=(A,\delta_{\bA},\widetilde{A})$ is a \pregamp, $\bff_{p}\colon\bA_p\to\bA$ is a morphism of \pregamps\ for each $p\in P$, and the following is a directed colimit cocone:
\[
\famm{\bA,\bff_p}{p\in P}=\varinjlim \famm{\bA_p,\bff_{p,q}}{p\le q\text{ in }P},\quad\text{ in }\MPALG_{\sL}.
\]

Moreover if $\bA_p$ is distance-generated for each $p\in P$, then~$\bA$ is distance-generated.
\end{lemma}

\begin{remark}\label{R:CPGandPGApreslim}
As an immediate application of Lemma~\ref{L:MPALGhasLimit}, and the fact that $\Conc$ preserves directed colimits, we obtain that both $\CPG$ and $\PGA$ preserve directed colimits.
\end{remark}

\begin{definition}
An \emph{embedding} $\bff\colon\bB\to \bA$ of \pregamps\ is a morphism of \pregamps\ such that~$f$ and~$\widetilde f$ are both one-to-one.

A \emph{\subpregamp} of a \pregamp~$\bA$ is a \pregamp~$\bB$ such that~$B$ is a partial subalgebra of~$A$, $\widetilde B$ is a \jz-subsemilattice of~$\widetilde A$, and $\delta_{\bB}=\delta_{\bA}\res B^2$.

If $f\colon B\to A$ and~$\widetilde f\colon\widetilde B\to\widetilde A$ denote the inclusion maps, the morphism of \pregamps\ $\bff=(f,\widetilde f)$ is called \emph{the canonical embedding}.
\end{definition}

\begin{notation}\label{N:fbC}
Let $\bff\colon\bB\to\bA$ be a morphism of \pregamps. Given a \subpregamp~$\bC$ of~$\bB$, the triple $\bff(\bC)=(f(C),\delta_{\bA}\res (f(C))^2,\widetilde f(\widetilde C))$ (see Notation~\ref{N:imagepartsubalgebra}) is a \subpregamp\ of~$\bA$.

For a \subpregamp~$\bC$ of~$\bA$, the triple $\bff^{-1}(\bC)=(f^{-1}(C),\delta_{\bB}\res (f^{-1}(C))^2,\widetilde f^{-1}(\widetilde C))$ is a \subpregamp\ of~$\bB$.
\end{notation}

We leave to the reader the straightforward proof of the following description of \subpregamps\ and embeddings.

\begin{proposition}\label{P:descrsubpregamps}
The following statements hold.
\begin{enumerate}
\item Let~$\bA$ be a \pregamp, let~$B$ be a partial subalgebra of~$A$, let~$\widetilde B$ be a \jz-subsemilattice of~$\widetilde A$ that contains $\delta_{\bA}(B^2)$. Put $\delta_{\bB}=\delta_{\bA}\res B^2$. Then $(B,\delta_{\bB},\widetilde B)$ is a \subpregamp\ of~$\bA$. Moreover, all \subpregamps\ of~$\bA$ are of this form.
\item Let $\bff\colon\bB\to\bA$ be a morphism of \pregamps. Then~$f$ is an embedding of partial algebras if and only if~$\widetilde f$ separates~$0$. Moreover $\bff$ is an embedding if and only if~$\widetilde f$ is an embedding.
\item Let $\bff\colon\bB\to\bA$ be an embedding of \pregamps. The restriction $\bff\colon\bB\to\bff(\bB)$ is an isomorphism of \pregamps.
\end{enumerate}
\end{proposition}

The following result appears in \cite[Theorem~10.4]{UnivAlgebra}. It gives a description of finitely generated congruences of a general algebra.

\begin{lemma}\label{L:Condcompcongruences}
Let~$B$ be an algebra, let $m$ be a positive integer, let $x,y\in B$, and let~$\vec x,\vec y$ be $m$-tuples of~$B$. Then $\Theta_B(x,y)\leq\bigvee_{i<m}\Theta_B(x_i,y_i)$ if and only if there are a positive integer~$n$, a list~$\vec z$ of parameters from~$B$, and terms $t_0$, \dots, $t_n$ such that
\begin{align*}
x&=t_0(\vec x,\vec y,\vec z),\\
y&=t_n(\vec x,\vec y,\vec z),\\
t_j(\vec y,\vec x,\vec z)&=t_{j+1}(\vec x,\vec y,\vec z)\quad(\text{for all }j<n).
\end{align*}
\end{lemma}

The following lemma shows that the obvious direction of Lemma~\ref{L:Condcompcongruences} holds for \pregamps.

\begin{lemma}\label{L:TermsDist}
Let~$\bB$ be a \pregamp, let $m$ be a positive integer, let $x,y\in B$, and let~$\vec x,\vec y$ be $m$-tuples of~$B$. Assume that there are a positive integer~$n$, a list~$\vec z$ of parameters from~$B$, and terms $t_0$, \dots, $t_n$ such that the following equalities hold and all evaluations are defined
\begin{align*}
x&=t_0(\vec x,\vec y,\vec z),\\
y&=t_n(\vec x,\vec y,\vec z),\\
t_j(\vec y,\vec x,\vec z)&=t_{j+1}(\vec x,\vec y,\vec z),\quad(\text{for all }j<n).
\end{align*}
Then $\delta_{\bB}(x,y)\leq\bigvee_{i<m}\delta_{\bB}(x_i,y_i)$.
\end{lemma}

\begin{proof}
As~$\delta_{\bB}$ is compatible with terms (cf. Lemma~\ref{L:DistCompTerms}), and $\delta_{\bB}(u,u)=0$ for each $u\in B$, the following inequality holds:
\[
\delta_{\bB}(t_j(\vec x,\vec y,\vec z),t_j(\vec y,\vec x,\vec z))\le\bigvee_{k<n}\delta_{\bB}(x_k,y_k),\quad\text{for all $j<n$}.
\]
Hence:
\[
\delta_{\bB}(x,y)\le\bigvee_{j<n}\delta_{\bB}(t_j(\vec x,\vec y,\vec z),t_j(\vec y,\vec x,\vec z))\le\bigvee_{k<n}\delta_{\bB}(x_k,y_k).\tag*{\qed}
\]
\renewcommand{\qed}{}
\end{proof}

The following definition expresses that whenever two elements of~$A$ are identified by a ``congruence'' of~$A$, then there is a ``good reason'' for this in~$B$ (cf. Lemma~\ref{L:Condcompcongruences}).

\begin{definition}\label{D:CongTract}
A morphism $\bff\colon\bA\to\bB$ of \pregamps\ is \emph{congruence-tractable} if for all $m<\omega$ and for all $x,y,x_0,y_0,\dots,x_{m-1},y_{m-1}$ in~$A$ such that:
\[
\delta_{\bB}(x,y)\le\bigvee_{k<m}\delta_{\bB}(x_k,y_k),
\]
there are a positive integer~$n$, a list~$\vec z$ of parameters from~$B$, and terms $t_0$, \dots, $t_n$ such that the following equations are satisfied in~$B$ (in particular, all the corresponding terms are defined).
\begin{align*}
f(x)&=t_0(f(\vec x),f(\vec y),\vec z),\\
f(y)&=t_n(f(\vec x),f(\vec y),\vec z),\\
t_j(f(\vec y),f(\vec x),\vec z)&=t_{j+1}(f(\vec x),f(\vec y),\vec z)\quad(\text{for all }j<n).
\end{align*}
\end{definition}

\begin{lemma}\label{L:DLPlagCon}
Let $P$ be directed poset and let~$\vec \bA=\famm{\bA_p,\bff_{p,q}}{p\le q\text{ in }P}$ be a direct system of \pregamps. Assume that for each $p\in P$ there exists $q\ge p$ in $P$ such that~$\bff_{p,q}$ is congruence-tractable. Let:
\[
\famm{\bA,\bff_p}{p\in P}=\varinjlim \famm{\bA_p,\bff_{p,q}}{p\le q\text{ in }P},\quad\text{ in }\MPALG_{\sL}.
\]
If~$A$ is an algebra then the following statements hold:
\begin{enumerate}
\item Let $x,y\in A$, let $m<\omega$, let $x_0,y_0,\dots,x_{m-1},y_{m-1}$ in~$A$. The following two inequalities are equivalent:
\begin{equation}\label{E:DLPC1}
\delta_{\bA}(x,y)\le \bigvee_{k<m}\delta_{\bA}(x_k,y_k),
\end{equation}
\begin{equation}\label{E:DLPC2}
\Theta_A(x,y)\le\bigvee_{k<m}\Theta_A(x_k,y_k).
\end{equation}
\item There exists a unique \jzh\ $\phi\colon \Conc A\to\widetilde{A}$ such that: \[\phi(\Theta_A(x,y))=\delta_{\bA}(x,y),\quad\text{ for all $x,y\in A$.}\]
Moreover $\phi$ is an embedding.
\item If~$\bA$ is distance-generated, then the \jzh\ $\phi$ above is an isomorphism.
\end{enumerate}
\end{lemma}

\begin{proof}
Lemma~\ref{L:MPALGhasLimit} and Lemma~\ref{L:PALGhasLimit} imply that the following are directed colimits cocones
\begin{equation}\label{E:CL1}
\famm{A,f_p}{p\in P}=\varinjlim\famm{A_p,f_{p,q}}{p\le q\text{ in }P},\quad{\text{in $\PALG_{\sL}$.}}
\end{equation}
\begin{equation}\label{E:CL2}
\famm{A,f_p}{p\in P}=\varinjlim \famm{A_p,f_{p,q}}{p\le q\text{ in }P},\quad{\text{in $\SET$.}}
\end{equation}
\begin{equation}\label{E:CL3}
\famm{\widetilde{A},\widetilde{f}_p}{p\in P}=\varinjlim \famm{\widetilde{A}_p,\widetilde{f}_{p,q}}{p\le q\text{ in }P},\quad{\text{in $\SEM$.}}
\end{equation}

$(1)$ Let $x,y\in A$, let $m<\omega$, and let~$\vec x,\vec y$ be $m$-tuples of~$A$.

Assume that~\eqref{E:DLPC1} holds. It follows from~\eqref{E:CL2} that there are $p\in P$, $x',y'\in A$, and $m$-tuples~$\vec x',\vec y'$ of~$A_p$, such that $x=f_p(x')$, $y=f_p(y')$, $\vec x=f_p(\vec x')$, and~$\vec y=f_p(\vec y')$. The inequality~\eqref{E:DLPC1} can be written
\[
\delta_{\bA}(f_p(x'),f_p(y'))\le \bigvee_{k<m}\delta_{\bA}(f_p(x_k'),f_p(y_k')).
\]
This implies:
\[
\widetilde{f}_p\left(\delta_{\bA_p}(x',y')\right)\le \widetilde{f}_p\left(\bigvee_{k<m}\delta_{\bA_p}(x_k',y_k')\right).
\]
Hence, it follows from~\eqref{E:CL3} that there exists $q\ge p$ with: 
\[
\widetilde{f}_{p,q}\left(\delta_{\bA_p}(x',y')\right)\le \widetilde{f}_{p,q}\left(\bigvee_{k<m}\delta_{\bA_p}(x_k',y_k')\right),
\]
so, changing $p$ to $q$, $x'$ to $f_{p,q}(x')$, $y'$ to $f_{p,q}(y')$, $\vec x'$ to $f_{p,q}(\vec x')$, and~$\vec y'$ to $f_{p,q}(\vec y')$, we can assume that:
\[
\delta_{\bA_p}(x',y')\le \bigvee_{k<m}\delta_{\bA_p}(x_k',y_k').
\]

Let $q\ge p$ in $P$ such that $f_{p,q}$ is congruence-tractable. There are a positive integer~$n$, a list~$\vec z$ of parameters from~$A_q$, and terms $t_0$, \dots, $t_n$ such that the following equations are satisfied in~$A_q$:
\begin{align*}
f_{p,q}(x')&=t_0(f_{p,q}(\vec x'),f_{p,q}(\vec y'),\vec z),\\
f_{p,q}(y')&=t_n(f_{p,q}(\vec x'),f_{p,q}(\vec y'),\vec z),\\
t_k(f_{p,q}(\vec y'),f_{p,q}(\vec x'),\vec z)&=t_{k+1}(f_{p,q}(\vec x'),f_{p,q}(\vec y'),\vec z),\quad(\text{for all }k<n).
\end{align*}
Hence, applying $f_q$, we obtain
\begin{align*}
x&=t_0(\vec x,\vec y,f_q(\vec z)),\\
y&=t_n(\vec x,\vec y,f_q(\vec z)),\\
t_k(\vec y,\vec x,f_q(\vec z))&=t_{k+1}(\vec x,\vec y,f_q(\vec z)),\quad(\text{for all }k<n).
\end{align*}
Therefore, it follows from Lemma~\ref{L:Condcompcongruences} that~\eqref{E:DLPC2} holds.

Conversely, assume that~\eqref{E:DLPC2} holds. It follows from Lemma~\ref{L:Condcompcongruences} that there are a positive integer~$n$, a list~$\vec z$ of parameters from~$A$, and terms $t_0$, \dots, $t_n$ such that
\begin{align*}
x&=t_0(\vec x,\vec y,\vec z),\\
y&=t_n(\vec x,\vec y,\vec z),\\
t_j(\vec y,\vec x,\vec z)&=t_{j+1}(\vec x,\vec y,\vec z),\quad(\text{for all }j<n).
\end{align*}
We conclude, using Lemma~\ref{L:TermsDist}, that~\eqref{E:DLPC1} holds.

As $\Conc A$ is generated by $\setm{\Theta_A(x,y)}{x,y\in A}$, the statement $(2)$ follows from Proposition~\ref{P:semilatiso}. Moreover if we assume that~$\bA$ is distance-generated, that is~$\widetilde A$ is join-generated by $\setm{\delta_{\bA}(x,y)}{x,y\in A}$, then $\phi$ is an isomorphism.
\end{proof}

As an immediate application, we obtain that a ``true'' directed colimit of ``good'' \pregamps\ is an algebra together with its congruences.

\begin{corollary}\label{C:directlimisalgebra}
Let $P$ be directed poset with no maximal element and let~$\vec \bA=\famm{\bA_p,\bff_{p,q}}{p\le q\text{ in }P}$ be a $P$-indexed diagram of distance-generated \pregamps. If~$\bff_{p,q}$ is congruence-tractable and $f_{p,q}$ is strong for all $p<q$ in $P$, then there exists a unique \jzh\ $\phi\colon \Conc A\to \widetilde{A}$ such that: \[\phi(\Theta_A(x,y))=\delta_{\bA}(x,y)\quad\text{ for all $x,y\in A$.}\]
Moreover, $\phi$ is an isomorphism.
\end{corollary}

\begin{definition}
An \emph{ideal} of a \pregamp~$\bA$ is an ideal of $\widetilde A$. Denote by $\Id\bA=\Id\widetilde A$ the set of all ideals of~$\bA$.

Let $P$ be a poset, let~$\vec \bA=\famm{\bA_p,\bff_{p,q}}{p\le q\text{ in }P}$ be a $P$-indexed diagram in $\MPALG_{\sL}$. An \emph{ideal} of $\vec\bA$ is an ideal of $\famm{\widetilde A_p,\widetilde f_{p,q}}{p\le q\text{ in }P}$ (cf. Definition~\ref{D:idealsem}).
\end{definition}

\begin{definition}
Let $\bpi\colon\bA\to\bB$ be a morphism of \pregamps. The \emph{$0$-kernel of~$\bpi$}, denoted by $\ker_0\bpi$, is the $0$-kernel of $\widetilde\pi$ (cf. Definition~\ref{D:idealsem}).

Let $P$ be a poset and let $\vec\bpi=(\bpi_p)_{p\in P}\colon\vec\bA\to\vec\bB$ be a natural transformation of $P$-indexed diagrams of \pregamps. The \emph{$0$-kernel of~$\vec\bpi$} is $\vec I=(\ker_0\bpi_p)_{p\in P}$.
\end{definition}

\begin{remark}
The $0$-kernel of~$\bpi$ is an ideal of~$\bA$. Similarly the $0$-kernel of~$\vec\bpi$ is an ideal of $\vec\bA$.

If $\pi\colon A\to B$ is a morphism of algebras, then $\ker_0\PGA(\pi)$ is the set of all compact congruences of~$A$ below $\ker\pi$, that is, $\ker_0\PGA(\pi)=(\Conc A)\dnw\ker\pi$.
\end{remark}

\begin{definition}
A morphism $\bff\colon\bA\to\bB$ of \pregamps, is \emph{ideal-induced} if $f(A)=B$ as partial algebras and $\widetilde f$ is ideal-induced. In that case we say that~$\bB$ is an \emph{ideal-induced image of~$\bA$}.

Let $P$ be a poset, let $\vec\bA$ and~$\vec\bB$ be $P$-indexed diagrams of \pregamps. A natural transformation $\vec\bff=(\bff_p)_{p\in P}\colon\vec\bA\to\vec\bB$ is \emph{ideal-induced} if $\bff_p$ is ideal-induced for each $p\in P$.
\end{definition}

\begin{remark}
A morphism $\bff\colon\bA\to\bB$ of \pregamps\ is ideal-induced if $\widetilde f$ is ideal-induced, $f$ is surjective, and for each $\ell\in\sL$ and each tuple $\vec b$ of~$B$, $\ell(\vec b)$ is defined in~$B$ if and only if there exists a tuple~$\vec a$ in~$A$ such that $\vec b=f(\vec a)$ and $\ell(\vec a)$ is defined in~$A$.

If $f\colon A\tosurj B$ is a surjective morphism of algebras, then $\PGA(f)$ is ideal-induced.

If $\bff\colon\bA\to\bB$ and $\bgg\colon\bB\to\bC$ are ideal-induced morphisms of \pregamps, then $\bgg\circ\bff$ is ideal-induced.
\end{remark}

The following proposition gives a description of \emph{quotients} of \pregamps.

\begin{proposition}\label{P:partialquotien}
Let~$\bA$ be a \pregamp\ and let $I\in\Id\widetilde{A}$. The binary relation $\theta_I=\setm{(x,y)\in A^2}{\delta_{\bA}(x,y)\in I}$ is an equivalence relation on~$A$. Given $a\in A$ denote by $a/I$ the $\theta_I$-equivalence class containing~$a$, and set $A/I=A/{\theta_I}$. We can define a structure of partial algebra on~$A/I$ in the following way. Given $\ell\in\sL$, we put:
\begin{align*}
&\Def_{\ell}(A/I)=\setm{\vec x/I}{\vec x\in\Def_{\ell}(A)},\\
&\ell^{A/I}(\vec x/I)=\ell^A(\vec x)/I,\quad \text{for all~$\vec x\in \Def_{\ell}(A)$}.
\end{align*}

Moreover $\delta_{\bA/I}\colon (A/I)^2\to \widetilde{A}/I$, $(x/I,y/I)\mapsto \delta_{\bA}(x,y)/I$ defines an~$\widetilde{A}/I$-valued partial algebra distance, and the following statements hold:
\begin{enumerate}
\item $\bA/I=(A/I,\delta_{\bA/I},\widetilde{A}/I)$ is a \pregamp.
\item Put $\pi\colon A\to A/I$, $x\mapsto x/I$, and denote by~$\widetilde{\pi}\colon \widetilde{A}\to \widetilde{A}/I$, $d\mapsto d/I$ the canonical projection. Then $\bpi=(\pi,\widetilde{\pi})$ is an ideal-induced morphism of \pregamps\ from~$\bA$ to $\bA/I$.
\item The $0$-kernel of~$\bpi$ is~$I$.
\item If~$\bA$ is distance-generated, then $\bA/I$ is distance-generated.
\end{enumerate}
\end{proposition}

\begin{proof}
The relation $\theta_I$ is reflexive (it follows from Definition~\ref{D:distance}(1)), symmetric (see Definition~\ref{D:distance}(2)) and transitive (see Definition~\ref{D:distance}(3)), thus it is an equivalence relation.

Let $\ell\in\sL$, let~$\vec x,\vec y\in\Def_{\ell}(A)$ such that $x_k/I=y_k/I$ for each $k<\ari(\ell)$. It follows from Definition~\ref{D:distance}(4) that $\delta_{\bA}(\ell^A(\vec x),\ell^A(\vec y))\le\bigvee_{k<\ari(\ell)}\delta_{\bA}(x_k,y_k)\in I$, so $\ell^A(\vec x)/I=\ell^A(\vec y)/I$. Therefore the partial operation $\ell^{A/I}\colon \Def_{\ell}(A/I)\to A/I$ is well-defined.

Let $x,x',y,y'\in A$, assume that $x/I=x'/I$ and $y/I=y'/I$. The following inequality holds: \[\delta_{\bA}(x,y)\le\delta_{\bA}(x,x')\vee\delta_{\bA}(x',y')\vee\delta_{\bA}(y',y).
\]
However, $\delta_{\bA}(x,x')$ and $\delta_{\bA}(y,y')$ both belong to~$I$, hence $\delta_{\bA}(x,y)/I\le \delta_{\bA}(x',y')/I$. Similarly $\delta_{\bA}(x',y')/I\le \delta_{\bA}(x,y)/I$. So the map $\delta_{\bA/I}\colon (A/I)^2\to \widetilde{A}/I$ is well-defined.

Let $x,y\in A$, the following equivalences hold:
\[
\delta_{\bA/I}(x/I,y/I)=0/I\Longleftrightarrow\delta_{\bA}(x,y)\in I\Longleftrightarrow x/I=y/I.
\]
That is, Definition~\ref{D:distance}(1) holds. Each of the conditions of Definition~\ref{D:distance}(2)-(5) for~$\delta_{\bA}$ implies its analogue for~$\delta_{\bA/I}$.

It is easy to check that~$\bpi$ is well-defined and that it is a morphism of \pregamps.
\end{proof}

\begin{notation}\label{N:quotpregamp}
The notations $\bA/I$,~$A/I$, and $\delta_{\bA/I}$ used in Proposition~\ref{P:partialquotien} will be used throughout the paper. The map~$\bpi$ is \emph{the canonical projection}.

If $I=\set{0}$, we identify $\bA/I$ and~$\bA$.

If~$X$ is a partial subalgebra of~$A$, then we denote $X/I=\setm{x/I}{x\in X}$, with its natural structure of partial subalgebra of~$A$, inherited from~$X$, with $\Def_{\ell}(X/I)=\setm{\vec x/I}{\vec x\in \Def_{\ell}(X)}$ for each $\ell\in\sL$. That is $X/I=\pi(X)$ as partial algebras.

Let~$\bB$ be a \subpregamp\ of~$\bA$ and let $I$ be a common ideal of~$\bA$ and~$\bB$. Then we identify the quotient $\bB/I$ with the corresponding \subpregamp\ of $\bA/I$.
\end{notation}

\begin{remark*}
It is easy to construct a \pregamp~$\bA$, a term $t$, a tuple~$\vec x$ of~$A$, and an ideal~$I$ of~$\widetilde A$, such that $t(\vec x)$ is not defined in~$A$, but $t(\vec x/I)$ is defined in~$A/I$.
\end{remark*}

The following proposition gives a description of how morphisms of \pregamps\ factorize through quotients. It is related to Lemma~\ref{L:quotMorphSem}.

\begin{proposition}\label{P:projsmpa}
Let $\bff\colon\bA\to\bB$ be a morphism of \pregamps, let $I\in\Id\bA$, and let $J\in\Id\bB$. If~$\widetilde f(I)\subseteq J$, then the following maps are well-defined:
\begin{align*}
g\colon A/I&\to B/J\\
x/I&\mapsto f(x)/J,
\end{align*}
\begin{align*}
\widetilde g\colon \widetilde A/I&\to \widetilde B/J\\
\alpha/I&\mapsto \widetilde f(\alpha)/J.
\end{align*}
Moreover, $\bgg=(g,\widetilde g)$ is a morphism of \pregamps\ from~$\bA/I$ to~$\bB/J$. If $\bpi_I\colon\bA\to\bA/I$ and $\bpi_J\colon\bA\to\bA/I$ denote the canonical projections, then the following diagram commutes:
\[
\begin{CD}
\bA & @>\bff>> & \bB\\
@V{\bpi_I}VV & & @VV{\bpi_J}V\\
\bA/I & @>>\bgg> & \bB/J
\end{CD}
\]
\end{proposition}

\begin{proof}
Observe that $\widetilde g\colon\widetilde A/I\to \widetilde B/J$ is the \jzh\ induced by $\widetilde f$. Let $x,y\in A$ such that $x/I=y/I$, that is, $\delta_{\bA}(x,y)\in I$. It follows that $\delta_{\bB}(f(x),f(y))=\widetilde f( \delta_{\bA}(x,y))\in J$, so $f(x)/J=f(y)/J$. Therefore the map~$g$ is well-defined.

Let $\ell\in\sL$, let~$\vec a\in\Def_{\ell}(A/I)$, and let~$\vec x\in\Def_{\ell}(A)$ such that~$\vec a=\vec x/I$. The following equalities hold:
\[
g(\ell(\vec x/I))=g(\ell(\vec x)/I)=f(\ell(\vec x))/J=\ell(f(\vec x))/J=\ell(f(\vec x)/J)=\ell(g(\vec x/I)).
\]
Thus $g(\ell(\vec a))=\ell(g(\vec a))$. Therefore~$g$ is a morphism of partial algebras.

Let $x,y\in A$. It is easy to check~$\widetilde g(\delta_{\bA/I}(x/I,y/I))=\delta_{\bB/J}(g(x/J),g(y/J))$. Therefore $\bgg\colon\bA/I\to\bB/J$ is a morphism of \pregamps. Moreover $\bpi_J\circ\bff=\bgg\circ\bpi_I$ is obvious.
\end{proof}

\begin{notation}
We say that $\bff$ \emph{induces} $\bgg\colon \bA/I\to\bB/J$, the morphism of Proposition~\ref{P:projsmpa}.

Let $P$ be a poset, let~$\vec \bA=\famm{\bA_p,\bff_{p,q}}{p\le q\text{ in }P}$ be a $P$-indexed diagram in $\MPALG_{\sL}$, let $\vec I=(I_p)_{p\in P}$ be an ideal of $\vec\bA$, and let $\bgg_{p,q}\colon\bA_p/I_p\to \bA_q/I_q$ the morphism induced by~$\bff_{p,q}$, for all $p\le q$ in $P$. We denote by $\vec\bA/\vec I=\famm{\bA_p/I_p,\bgg_{p,q}}{p\le q\text{ in }P}$.

The diagram $\vec\bA/\vec I$ is a \emph{quotient} of $\vec\bA$.
\end{notation}

\begin{remark}
It is easy to check that $\vec\bA/\vec I$ is indeed a diagram. Given $p\le q\le r$ in $P$ and $x\in A_p$ the following equalities hold:
\[
g_{q,r}(g_{p,q}(x/I_p))=g_{q,r}(f_{p,q}(x)/I_q)=f_{q,r}(f_{p,q}(x))/I_r=f_{p,r}(x)/I_r=g_{p,r}(x/I_p).
\]
\end{remark}

Proposition~\ref{P:projsmpa} can be easily extended to diagrams in the following way. It is also related to Lemma~\ref{L:quotDiagSem}.

\begin{proposition}\label{P:projsmpadiag}
Let $P$ be a poset, let~$\vec \bA=\famm{\bA_p,\bff_{p,q}}{p\le q\text{ in }P}$ and~$\vec \bB=\famm{\bB_p,\bgg_{p,q}}{p\le q\text{ in }P}$ be $P$-indexed diagrams in $\MPALG_{\sL}$. Let~$\vec I$ be an ideal of~$\vec\bA$, let~$\vec J$ be an ideal of~$\vec\bB$. Let $\vec\bxi=(\bxi_p)_{p\in P}\colon\vec\bA\to\vec\bB$ be a natural transformation such that $\xi_p(I_p)\subseteq J_p$ for each $p\in P$. Denote by $\bchi_p\colon \bA_p/I_p\to\bB_p/J_p$ the morphism induced by $\bxi_p$, for each $p\in P$. Then $\vec\bchi=(\bchi_p)_{p\in P}$ is a natural transformation from $\vec\bA/\vec I$ to $\vec\bB/\vec J$.
\end{proposition}

\begin{notation}
With the notation of Proposition~\ref{P:projsmpadiag}. We say that $\vec\bchi\colon \vec\bA/\vec I\to\vec\bB/\vec J$ is \emph{induced by $\vec\bxi$}.
\end{notation}

The following lemma expresses that ideal-induced images of \pregamps\ correspond, up to isomorphism, to quotients of \pregamps. It is related to Lemma~\ref{L:IIsem}.

\begin{lemma}\label{L:IIPG}
Let $\bff\colon\bA\to\bB$ be a morphism of \pregamps. The following statements are equivalent:
\begin{enumerate}
\item $\bff$ is ideal-induced.
\item $\bff$ induces an isomorphism $\bgg\colon\bA/\ker_0\bff\to\bB$.
\end{enumerate}
\end{lemma}

\begin{proof}
Denote by $\bpi\colon \bA\to\bA/\ker_0\bff$ the canonical projection, so $\bgg\circ\bpi=\bff$.

Assume that $\bff$ is ideal-induced. As $\widetilde f\colon\widetilde A\to\widetilde B$ is ideal-induced, Lemma~\ref{L:IIsem} implies that $\widetilde f$ induces an isomorphism $\widetilde g\colon \widetilde A/\ker_0\bff\to\widetilde B$. It follows that $\widetilde g$ separates~$0$, thus (cf. Proposition~\ref{P:descrsubpregamps}(2)) $g$ is an embedding. Moreover $g(A/\ker_0\bff)=g(\pi(A))=f(A)=B$ as partial algebras. Therefore it follows from Lemma~\ref{L:caractiso} that $g$ is an isomorphism of partial algebras, thus $\bgg$ is an isomorphism of \pregamps\ (cf. Remark~\ref{R:isopregamp}).

Assume that $\bgg$ is an isomorphism. It follows that $\widetilde g$ is an isomorphism, so Lemma~\ref{L:IIsem} implies that $\widetilde f$ is ideal-induced. Moreover $g$ is an isomorphism, thus $f(A)=g(\pi(A))=g(A/\ker_0\bff)=B$ as partial algebras. Therefore $\bff$ is ideal-induced.
\end{proof}

The following proposition expresses that a quotient of a quotient is a quotient. It follows from Lemma~\ref{L:IIPG}, together with the fact that a composition of ideal-induced morphisms of \pregamps\ is ideal-induced.

\begin{proposition}\label{P:quotquotisquot}
Let~$\bA$ be a \pregamp, let~$I$ be an ideal of~$\bA$, let $J$ be an ideal of $\bA/I$. Then $(\bA/I)/J$ is isomorphic to a quotient of~$\bA$.
\end{proposition}

The following results expresses that, up to isomorphism, quotients of \subpregamps\ are \subpregamps\ of quotients.

\begin{proposition}\label{P:quotsubgamp}
Let~$\bA$ be a \pregamp, let~$\bB$ be a \subpregamp\ of~$\bA$, and let $I\in\Id\bB$. Then there exist $J\in\Id\bA$, a \subpregamp\ $\bC$ of $\bA/J$, and an isomorphism $\bff\colon\bB/I\to\bC$.

Let~$\bA$ be a \pregamp, let $I\in\Id\bA$, and let~$\bB$ be a \subpregamp\ of $\bA/I$. There exists a \subpregamp\ $\bC$ of~$\bA$ such that~$\bB$ is isomorphic to some quotient of $\bC$.
\end{proposition}

\begin{proof}
Let~$\bA$ be a \pregamp, let~$\bB$ be a \subpregamp\ of~$\bA$, let $I\in\Id\bB$. Put $J=\widetilde A\dnw I$. As~$I$ is an ideal of $\widetilde B$, it is directed, therefore $J$ is an ideal of~$\widetilde A$.

Let $\bff\colon \bB\to \bA$ be the canonical embedding. Notice that~$\widetilde{f}(I)\subseteq J$; denote by $\bgg\colon \bB/I\to\bA/J$ the morphism induced by $\bff$ (cf. Proposition~\ref{P:projsmpa}).

Let $d,d'\in \widetilde B$ such that~$\widetilde g(d/I)=\widetilde g(d'/I)$, that is, $d/J=d'/J$, so there exists $u\in J$ such that $d\vee u=d'\vee u$. As $J=\widetilde A\dnw I$, there exists $v\in I$ such that $u\le v$, hence $d\vee v=d'\vee v$, that is, $d/I=d'/I$. Therefore~$\widetilde g$ is an embedding. It follows from Proposition~\ref{P:descrsubpregamps} that $\bgg=(g,\widetilde g)$ is an embedding and induces an isomorphism $\bB/I\to \bgg(\bB/I)$; the latter is a \subpregamp\ of $\bA/J$.

Now let $I\in\Id\bA$ and let~$\bB$ be a \subpregamp\ of $\bA/I$. Denote by $\bpi\colon\bA\to\bA/I$ the canonical projection, put $\bC=\bpi^{-1}(\bB)$ (cf. Notation~\ref{N:fbC}). As $\bpi$ is ideal-induced, it is easy to check that $\bpi(\bC)=\bB$, and the restriction $\bpi\res\bC\to\bB$ is ideal-induced.
\end{proof}

The following lemma, in conjunction with Lemma~\ref{L:colimitquot2}, proves that, given a direct system~$\vec \bA$ of \pregamps, every quotient of the colimit of $\vec \bA$ is the colimit of a quotient of~$\vec \bA$.

\begin{lemma}\label{L:colimitquot1}
Let $P$ be directed poset and let~$\vec \bA=\famm{\bA_p,\bff_{p,q}}{p\le q\text{ in }P}$ be a $P$-indexed diagram in $\MPALG_{\sL}$. Let $\famm{\bA,\bff_p}{p\in P}=\varinjlim \famm{\bA_p,\bff_{p,q}}{p\le q\text{ in }P}$ be a directed colimit cocone in $\MPALG_{\sL}$. Let~$\vec I=(I_p)_{p\in P}$ be an ideal of~$\vec \bA$. Then~$I=\bigcup_{p\in P} \widetilde f_p(I_p)$ is an ideal of~$\bA$.

Let $\bgg_p\colon\bA_p/I_p\to\bA/I$ be the morphism induced by $\bff_p$, let $\bgg_{p,q}\colon\bA_p/I_p\to\bA_q/I_q$ be the morphism induced by~$\bff_{p,q}$, for all $p\le q$ in $P$. The following is a directed colimit cocone:
\[
\famm{\bA/I,\bgg_p}{p\in P}=\varinjlim \famm{\bA_p/I_p,\bff_{p,q}}{p\le q\text{ in }P} \text{ in $\MPALG_{\sL}$.}
\]
\end{lemma}

\begin{proof}
Lemma~\ref{L:PALGhasLimit} and Lemma~\ref{L:MPALGhasLimit} imply that the following are colimits cocones:
\begin{align}
\famm{A,f_p}{p\in P}&=\varinjlim \famm{A_p,f_{p,q}}{p\le q\text{ in }P}\text{ in $\SET$},\label{E:PLHL1}\\
\famm{\widetilde{A},\widetilde{f}_p}{p\in P}&=\varinjlim \famm{\widetilde{A}_p,\widetilde{f}_{p,q}}{p\le q\text{ in }P}\text{ in $\SEM$}\label{E:PLHL2}\\
\famm{A,f_p}{p\in P}&=\varinjlim \famm{A_p,f_{p,q}}{p\le q\text{ in }P}\text{ in $\PALG_{\sL}$}\label{E:PLHL3}
\end{align}

Moreover, Lemma~\ref{L:colimitquot2} implies that $I$ is an ideal of $\widetilde A$ and that the following is a directed colimit cocone:
\begin{equation}
\famm{\widetilde A/I,\widetilde g_p}{p\in P}=\varinjlim \famm{\widetilde A_p/I_p,\widetilde g_{p,q}}{p\le q\text{ in }P}\text{ in $\SEM$}\label{E:PLHL4}
\end{equation}

Let $p\in P$, let $x,y\in A_p$ such that $g_p(x/I_p)=g_p(y/I_p)$. It follows that $f_p(x)/I=f_p(y)/I$, that is, $\delta_{\bA}(f_p(x),f_p(y))\in I$. So there exist $q\in P$ and $\alpha\in\widetilde A_q$ such that~$\widetilde f_p(\delta_{\bA_p}(x,y))=\delta_{\bA}(f_p(x),f_p(y))=\widetilde f_q(\alpha)$. It follows from~\eqref{E:PLHL2} that there exists $r\geq p,q$ such that $\delta_{\bA_r}(f_{p,r}(x),f_{p,r}(y))=\widetilde f_{p,r}(\delta_{\bA_p}(x,y))=\widetilde f_{q,r}(\alpha)$. However, $\widetilde f_{q,r}(\alpha)\in \widetilde f_{q,r}(I_q)\subseteq I_r$, so $f_{p,r}(x)/I_r=f_{p,r}(y)/I_r$, and so $g_{p,r}(x/I_p)=g_{p,r}(y/I_p)$. Moreover~$A/I=\bigcup_{p\in P} f_p(A_p)/I=\bigcup_{p\in P} g_p(A_p/I_p)$. Hence the following is a directed colimit cocone:
\[
\famm{A/I,g_p}{p\in P}=\varinjlim \famm{A_p/I_p,g_{p,q}}{p\le q\text{ in }P}\text{ in $\SET$.}
\]

Let $\ell\in\sL$. The following equalities hold:
\[
\Def_{\ell}(A/I)=\Def_{\ell}(A)/I=\bigcup_{p\in P}f_p(\Def_{\ell}(A_p))/I=\bigcup_{p\in P}g_p(\Def_{\ell}(A_p/I_p)).
\]
Let $p\in P$, let~$\vec a\in \Def_{\ell}(A_p/I)$. Then, as $g_p$ is a morphism of partial algebras, $g_p(\ell(\vec a))=\ell(g_p(\vec a))$. So, by Lemma~\ref{L:PALGhasLimit}, the following is a directed colimit cocone:
\[
\famm{A/I,g_p}{p\in P}=\varinjlim \famm{A_p/I_p,g_{p,q}}{p\le q\text{ in }P}\text{ in $\PALG_{\sL}$}.
\]

As $\bff_p$ is a morphism of \pregamps, $\delta_{\bA/I}(g_p(x),g_p(y))=\widetilde g_p(\delta_{\bA_p/I_p}(x,y))$ for all $p\in P$ and all $x,y\in A_p/I_p$, thus Lemma~\ref{L:MPALGhasLimit} implies that the following is a directed colimit cocone:
\[
\famm{\bA/I,\bgg_p}{p\in P}=\varinjlim \famm{\bA_p/I_p,\bgg_{p,q}}{p\le q\text{ in }P}\text{ in $\MPALG_{\sL}$}.\tag*{\qed}
\]
\renewcommand{\qed}{}
\end{proof}

\begin{definition}\label{D:PregofV}
A \pregamp~$\bA$ \emph{satisfies} an identity $t_1=t_2$ if~$A/I$ satisfies $t_1=t_2$ for each $I\in\Id\bA$.

Let~$\cV$ be a variety of algebras. A \pregamp~$\bA$ is a \emph{\pregamp\ of~$\cV$} if it satisfies all identities of~$\cV$.
\end{definition}

\begin{remark*}
It is not hard to construct a \pregamp~$\bA$, an identity $t_1=t_2$, and an ideal~$I$ of~$\widetilde A$ such that~$A$ satisfies $t_1=t_2$, but~$A/I$ fails $t_1=t_2$.

A \pregamp~$\bA$ satisfies an identity $t_1=t_2$ if and only if for each ideal-induced morphism $\bff\colon\bA\to\bB$ of \pregamps, the partial algebra~$B$ satisfies $t_1=t_2$.
\end{remark*}

\begin{definition}
Let~$\cV$ be a variety of algebras. \emph{The category of \pregamps\ of~$\cV$}, denoted by $\MPALG(\cV)$, is the full subcategory of $\MPALG_{\sL}$ in which the objects are all the \pregamps\ of~$\cV$.
\end{definition}

As an immediate application of Lemma~\ref{L:colimitquot1} and Lemma~\ref{L:colimitquot2}, we obtain that the class of all \pregamps\ that satisfy a given identity is closed under directed colimits.

\begin{corollary}\label{C:limidentity}
Let~$\cV$ be a variety of algebras and let $P$ be a directed poset. Let~$\vec \bA=\famm{\bA_p,\bff_{p,q}}{p\le q\text{ in }P}$ be a $P$-indexed diagram in $\MPALG(\cV)$. Let $\famm{\bA,\bff_p}{p\in P}=\varinjlim \famm{\bA_p,\bff_{p,q}}{p\le q\text{ in }P}$ be a directed colimit cocone in $\MPALG_{\sL}$. Then~$\bA$ is a \pregamp\ of~$\cV$.
\end{corollary}

Similarly, it follows from Proposition~\ref{P:quotquotisquot} and Proposition~\ref{P:quotsubgamp} that the class of all \pregamps\ that satisfy a given identity is closed under ideal-induced images and \subpregamps.

\begin{corollary}\label{C:quotsubpregamp}
Let~$\bA$ be a \pregamp\ of a variety~$\cV$, let $\bB$ be a \pregamp, let $\bff\colon\bA\to\bB$ be an ideal-induced morphism of \pregamps, then $\bB$ is a \pregamp\ of~$\cV$. Furthermore, every \subpregamp\ of~$\bA$ is a \pregamp\ of~$\cV$.
\end{corollary}

\section{Gamps}\label{S:Gamp}

A \emph{gamp} of a variety~$\cV$ is a \pregamp\ that ``belongs'' to~$\cV$ (cf. $(1)$), together with a partial subalgebra (cf. $(2)$). The main interest of this new notion is to express later some additional properties that reflect properties of algebras (cf. Definition~\ref{D:propgamp}). It is a generalization of the notion of a \emph{semilattice-metric cover as defined in} \cite[Section 5-1]{GiWe1}.

\begin{definition}\label{D:gamp}
Let~$\cV$ be a variety of~$\sL$-algebras. A \emph{gamp} (resp., a \emph{gamp of~$\cV$}) is a quadruple $\bA=(A^*,A,\delta_{\bA},\widetilde{A})$ such that 
\begin{enumerate}
\item $(A,\delta_{\bA},\widetilde{A})$ is a \pregamp\ (resp., a \pregamp\ of~$\cV$) (cf. Definitions~\ref{D:distance} and~\ref{D:PregofV}).
\item~$A^*$ is a partial subalgebra of~$A$.
\end{enumerate}

A \emph{realization} of~$\bA$ is an ordered pair $(A',\chi)$ such that~$A'\in\cV$,~$A$ is a partial subalgebra of~$A'$, $\chi\colon \widetilde A\to\Conc A'$ is a \jze, and $\chi(\delta_{\bA}(x,y))=\Theta_{A'}(x,y)$ for all $x,y\in A$. A realization is \emph{isomorphic} if $\chi$ is an isomorphism.

A gamp~$\bA$ is \emph{finite} if both~$A$ and~$\widetilde{A}$ are finite.

Let~$\bA$ and~$\bB$ be gamps. A morphism $\bff\colon(A,\delta_{\bA},\widetilde{A})\to(B,\delta_{\bB},\widetilde{B})$ of \pregamps\ is a \emph{morphism of gamps} from~$\bA$ to~$\bB$ if $f(A^*)$ is a partial subalgebra of~$B^*$.

The \emph{category of gamps of~$\cV$}, denoted by $\GAMP(\cV)$, is the category in which the objects are the gamps of~$\cV$ and the arrows are the morphisms of gamps.

A \emph{subgamp} of a gamp~$\bA$ is a gamp $\bB=(B^*,B,\delta_{\bB},\widetilde{B})$ such that~$B^*$ is a partial subalgebra of~$A^*$,~$B$ is a partial subalgebra of~$A$, $\delta_{\bB}=\delta_{\bA}\res B^2$, and~$\widetilde{B}$ is a \jz-subsemilattice of~$\widetilde{A}$. Let $f\colon B\to A$ and~$\widetilde{f}\colon\widetilde{B}\to\widetilde{A}$ be the inclusion maps. The ordered pair $(f,\widetilde{f})$ is a morphism of gamps from~$\bB$ to~$\bA$, called \emph{the canonical embedding}.
\end{definition}

\begin{remark*}
A gamp might have no realization. A realization of a finite gamp does not need to be finite. 

Let $\bff\colon\bA\to\bB$ be a morphism of gamps, let $(A',\chi)$ be a realization of~$\bA$, and let $(B',\xi)$ be a realization of~$\bB$. There might not exist any morphism $g\colon A'\to B'$.
\end{remark*}

\begin{definition}\label{D:gampchain}
A \emph{gamp of lattices} is a gamp of the variety of all lattices.

Let~$\bB$ be a gamp of lattices. A \emph{chain} of~$\bB$ is a sequence $x_0,x_1,\dots,x_{n-1}$ of~$B^*$ such that $x_i\wedge x_j=x_i$ in~$B$ for all $i\le j< n$. We sometime denote such a chain as $x_0\le x_1\le \dots\le x_{n-1}$. If $x_i\not=x_j$ for all $ i<j< n$, we denote the chain as $x_0<x_1<\dots<x_{n-1}$.

Let $u<v$ be a chain of~$\bB$, we say that $v$ is a \emph{cover} of $u$, and then we write $u\prec v$, if there is no chain $u<x<v$ in~$\bB$.
\end{definition}

The following properties for a gamp come from algebra. It follows from Definition~\ref{D:propgamp}$(1)$ that there are many operations defined in~$A$. With $(2)$ or $(3)$ all ``congruences'' have a set of ``generators''. Condition $(4)$ expresses that whenever two elements are identified by a ``congruence'' of~$A^*$, then there is a ``good reason'' for this in~$A$ (cf. Lemma~\ref{L:Condcompcongruences}). Conditions $(6)$ and $(7)$ are related to the transitive closure of relations. Condition $(8)$ is related to congruence $n$-permutability (cf. Proposition~\ref{P:CP-PR}).

\begin{definition}\label{D:propgamp}
A gamp~$\bA$ is \emph{strong} if the following holds:
\begin{enumerate}
\item[$(1)$] $A^*$ is a strong partial subalgebra of~$A$ (cf. Definition~\ref{D:StrPartAlg}).
\end{enumerate}

A gamp~$\bA$ is \emph{distance-generated} if it satisfies the following condition:
\begin{enumerate}
\item[$(2)$] Every element of $\widetilde A$ is a finite join of elements of the form $\delta_{\bA}(x,y)$ where $x,y\in A^*$.
\end{enumerate}

A gamp~$\bA$ of lattices is \emph{distance-generated with chains} if
\begin{enumerate}
\item[$(3)$] For all $\alpha\in\widetilde A$ there are a positive integer $n$, and chains $x_0<y_0, x_1<y_1,\dots,x_{n-1}<y_{n-1}$ of~$\bA$ such that $\alpha=\bigvee_{k<n}\delta_{\bA}(x_k,y_k)$.
\end{enumerate}

A gamp~$\bA$ is \emph{congruence-tractable} (cf. Lemma~\ref{L:Condcompcongruences}) if
\begin{enumerate}
\item[$(4)$] For all $x,y, x_0,y_0,\dots,x_{m-1},y_{m-1}$ in~$A^*$, if $\delta_{\bA}(x,y)\le\bigvee_{k<m}\delta_{\bA}(x_k,y_k)$ then there are a positive integer~$n$, a list~$\vec z$ of parameters from~$A$, and terms $t_0$, \dots, $t_n$ such that the following equations are satisfied in~$A$.
\begin{align*}
x&=t_0(\vec x,\vec y,\vec z),\\
y&=t_n(\vec x,\vec y,\vec z),\\
t_k(\vec y,\vec x,\vec z)&=t_{k+1}(\vec x,\vec y,\vec z)\quad(\text{for all }k<n).
\end{align*}
\end{enumerate}

A morphism $\bff\colon\bA\to\bB$ of gamps is \emph{strong} if
\begin{enumerate}
\item[$(5)$] $f(A)$ is a strong partial subalgebra of~$B^*$ (cf. Definition~\ref{D:StrPartAlg}).
\end{enumerate}

A morphism $\bff\colon\bA\to\bB$ of gamps is \emph{congruence-cuttable} if
\begin{enumerate}
\item[(6)] $f(A)$ is a partial sublattice of~$B^*$ and given a finite subset~$X$ of~$\widetilde B$ and $x,y\in A$ with $\delta_{\bA}(f(x),f(y))\le\bigvee X$, there are $n<\omega$ and $f(x)=x_0,\dots,x_n=f(y)$ in~$B^*$ such that $\delta_{\bA}(x_k,x_{k+1})\in\widetilde B\dnw X$ for all $k<n$.
\end{enumerate}

A morphism $\bff\colon\bA\to\bB$ of gamps of the variety of all lattices is \emph{congruence-cuttable with chains} if
\begin{enumerate}
\item[$(7)$] $f(A)$ is a partial sublattice of~$B^*$ and given a finite subset~$X$ of~$\widetilde B$ and $x,y\in A$ with $\delta_{\bA}(f(x),f(y))\le\bigvee X$, there is a chain $x_0<\dots<x_n$ of~$\bB$ such that $x_0=f(x)\wedge f(y)$, $x_n=f(x)\vee f(y)$, and $\delta_{\bA}(x_k,x_{k+1})\in\widetilde B\dnw X$ for all $k<n$.
\end{enumerate}

Let $n\ge 1$ be an integer. A gamp~$\bA$ is \emph{congruence $n$-permutable} if the following statement holds:
\begin{enumerate}
\item[$(8)$] For all $x_0,x_1,\dots,x_n\in A^*$, there are $x_0=y_0,y_1,\dots,y_n=x_n$ in~$A$ such that:
 \begin{align*}
 \delta_{\bA}(y_k,y_{k+1})&\le\bigvee\famm{\delta_{\bA}(x_{i},x_{i+1})}{i< n\text{ even}},
 &&\text{for all $k<n$ odd},\\
 \delta_{\bA}(y_k,y_{k+1})&\le\bigvee\famm{\delta_{\bA}(x_{i},x_{i+1})}{i< n\text{ odd}},
 &&\text{for all $k<n$ even}.
 \end{align*}
\end{enumerate}
\end{definition}

The following lemma shows that chains in strong gamps of lattices behave the same way as chains in lattices.

\begin{lemma}\label{L:presqueordre}
Let $x_0<\dots<x_n$ a chain of a strong gamp of lattices~$\bB$. The equalities $x_i\wedge x_j=x_j\wedge x_i=x_i$ and $x_i\vee x_j=x_j\vee x_i=x_j$ hold in~$B$ for all $i\le j\le n$.

Moreover the following statements hold:
\begin{align}
\delta_{\bB}(x_k,x_{k'})&\le \delta_{\bB}(x_i,x_j),\quad\text{for all $i\le k\le k'\le j\le n$}.
\label{E:po1}\\
\delta_{\bB}(x_i,x_j)&=\bigvee_{i\le k<j}\delta_{\bB}(x_k,x_{k+1}),\quad\text{for all $i\le j\le n$}.\label{E:po2}
\end{align}
\end{lemma}

\begin{proof}
Let $i\le j\le n$. As $x_i,x_j\in B^*$, all the elements $x_i\wedge x_j$, $x_j\wedge x_i$, $x_i\vee x_j$, and $x_j\vee x_i$ are defined in~$B$.

As $u\wedge v=v\wedge u$ is an identity of lattices, it follows that $x_j\wedge x_i=x_i\wedge x_j=x_i$.

As $x_j\wedge x_i=x_i$, $(x_i\wedge x_j)\vee x_j=x_i\vee x_j$ in~$B$, and $(u\wedge v)\vee v=v$ is an identity of lattices, $x_i\vee x_j=(x_i\wedge x_j)\vee x_j=x_j$. Similarly $x_j\vee x_i=x_j$.

Let $i\le k\le k'\le j\le n$. As $x_k\wedge x_{k'}=x_k$ and $x_j\wedge x_{k'}=x_{k'}$, we obtain from Definition~\ref{D:distance}(4) the inequality:
\[
\delta_{\bB}(x_k,x_{k'})=\delta_{\bB}(x_k\wedge x_{k'},x_j\wedge x_{k'})\le \delta_{\bB}(x_k,x_j).
\]
Similarly, as $x_k=x_k\vee x_i$ and $x_j=x_j\vee x_i$, the inequality $\delta_{\bB}(x_k,x_j)\le\delta_{\bB}(x_i,x_j)$ holds. Therefore~\eqref{E:po1} holds.

Let $i<j\le n$. Definition~\ref{D:distance}(3) implies the inequality:
\[
\delta_{\bB}(x_i,x_j)\le\bigvee_{i\le k<j}\delta_{\bB}(x_k,x_{k+1}).
\]
Moreover~\eqref{E:po1} implies $\delta_{\bB}(x_k,x_{k+1})\le\delta_{\bB}(x_i,x_j)$ for all $i\le k<j$, it follows that~\eqref{E:po2} is true.
\end{proof}

The following proposition gives a description of \emph{quotient} of gamps.

\begin{proposition}\label{P:quot}
Let~$\bA$ be a gamp of~$\cV$ and let~$I$ be an ideal of~$\widetilde{A}$. Let $(A,\delta_{\bA},\widetilde A)/I=(A/I,\delta_{\bA/I},\widetilde A/I)$ be the quotient \pregamp\ and set~$A^*/I=\setm{a/I}{a\in A^*}$ \textup(cf. Notation~\textup{\ref{N:quotpregamp}}\textup). The following statements hold:
\begin{enumerate}
\item $\bA/I=(A^*/I,A/I,\delta_{\bA/I},\widetilde{A}/I)$ is a gamp of~$\cV$.
\item The canonical projection $\bpi\colon (A,\delta_{\bA},\widetilde A)\to(A/I,\delta_{\bA/I},\widetilde A/I)$ of \pregamps\ is a morphism of gamps from~$\bA$ to $\bA/I$.
\item If $(A',\chi)$ is a realization of~$\bA$ in~$\cV$, then $(A'/\bigvee\chi(I),\chi')$ is a realization of~$\bA/I$ in~$\cV$, where 
\begin{align*}
\chi'\colon \widetilde{A}/I&\to\Conc(A'/\bigvee\chi(I))\\
d/I&\mapsto \chi(d)/\bigvee\chi(I).
\end{align*}
Moreover, if $(A',\chi)$ is an isomorphic realization of~$\bA$, then $(A'/\bigvee\chi(I),\chi')$ is an isomorphic realization of~$\bA/I$.
\item If~$\bA$ is strong, then $\bA/I$ is strong.
\item If~$\bA$ is distance-generated \pup{resp., distance-generated with chains} then $\bA/I$ is distance-generated \pup{resp., distance-generated with chains}.
\item If~$\bA$ is distance-generated and congruence-tractable, then $\bA/I$ is congruence-tractable.
\item Let $n\ge 2$ an integer. If~$\bA$ is congruence $n$-permutable then $\bA/I$ is congruence $n$-permutable.
\item If~$\bA$ is a gamp of lattices and $x_0\le x_1\le\dots\le x_n$ is a chain of~$\bA$, then $x_0/I\le x_1/I\le\dots\le x_n/I$ is a chain of $\bA/I$.
\end{enumerate}
\end{proposition}

\begin{proof}
The statement $(1)$ follows from Corollary~\ref{C:quotsubpregamp}. Denote by $\bpi\colon (A,\delta_{\bA},\widetilde A)\to(A/I,\delta_{\bA/I},\widetilde A/I)$ the canonical projection of \pregamps. The fact that $\pi(A^*)=A^*/I$ as partial algebras follows from the definition of~$A^*/I$. Thus $(3)$ holds.

Let $(A',\chi)$ be a realization of~$\bA$ in~$\cV$, let $d,d'\in \widetilde A$ such that $d/I=d'/I$. Hence there exists $u\in I$ such that $d\vee u=d'\vee u$, it follows that $\chi(d)/\chi(u)=\chi(d')/\chi(u)$, hence $\chi(d)/\bigvee\chi(I)=\chi(d')/\bigvee\chi(I)$. Therefore the map $\chi'$ is well-defined. It is easy to check that $\chi'$ is a \jzh. Assume that $\chi'(d/I)\le\chi'(d'/I)$ for some $d,d'\in \widetilde A$. Hence $\chi(d)/\bigvee\chi(I)\le\chi(d')/\bigvee\chi(I)$, so $\chi(d)\le\chi(d')\vee\bigvee\chi(I)$. However, $\chi(d)$ is a compact congruence of~$A'$, so there exist $u\in I$ such that $\chi(d)\le\chi(d')\vee\chi(u)=\chi(d'\vee u)$, as $\chi$ is an embedding, it follows that $d\le d'\vee u$, so $d/I\le d'/I$. Therefore $\chi'$ is an embedding.

Let $x,y\in A$. The following equivalences hold:
\begin{align*}
x/I=y/I &\Longleftrightarrow \delta_{\bA}(x,y)\in I\\
&\Longleftrightarrow \chi(\delta_{\bA}(x,y))\le\bigvee\chi(I)\\
&\Longleftrightarrow\Theta_{A'}(x,y)\le \bigvee\chi(I)\\
&\Longleftrightarrow x/\bigvee\chi(I)=y/\bigvee\chi(I).
\end{align*}
So we can identify~$A/I$ with the corresponding subset of~$A'/\bigvee\chi(I)$. Moreover, given~$\vec a\in\Def_{\ell}(A/I)$, there exists~$\vec x\in\Def_{\ell}(A)$ such that~$\vec a=\vec x/I$, hence $\ell(\vec a)=\ell(\vec x)/I$ is identified with $\ell(\vec x)/\bigvee\chi(I)=\ell(\vec x/\bigvee\chi(I))$. So this identification preserves the operations. 

Now assume that the realization is isomorphic. Then~$\chi$ is surjective, thus~$\chi'$ is surjective, and thus bijective, hence the realization $(A'/\bigvee\chi(I),\chi')$ is isomorphic. Therefore $(3)$ holds.

The proofs of the statements $(4)$, $(5)$, $(7)$, and $(8)$ are straightforward.

Assume that~$\bA$ is distance-generated and congruence-tractable. Let $x,y\in A^*$, let $m<\omega$, let~$\vec x,\vec y$ be $m$-tuples of~$A^*$. Assume that:
\[
\delta_{\bA/I}(x/I,y/I)\le\bigvee_{k<m}\delta_{\bA/I}(x_k/I,y_k/I).
\]
It follows that there exists $u\in I$ with:
\[
\delta_{\bA}(x,y)\le u\vee\bigvee_{k<m}\delta_{\bA}(x_k,y_k).
\]
However, as~$\bA$ is distance-generated, there exist $x_0',\dots,x_{p-1}',y_0',\dots,y_{p-1}'$ in~$A^*$ such that $u=\bigvee_{k<p}\delta_{\bA}(x_k',y_k')$. As $\delta_{\bA}(x_k',y_k')\le u\in I$, $x_k'/I=y_k'/I$ for all $k<p$. Moreover the following inequality holds:
\[
\delta_{\bA}(x,y)\le \bigvee_{k<m}\delta_{\bA}(x_k,y_k)\vee\bigvee_{k<p}\delta_{\bA}(x_k',y_k').
\]
As~$\bA$ is congruence-tractable, there are a positive integer~$n$, a list~$\vec z$ of parameters from~$A$, and terms $t_0$, \dots, $t_n$ such that, the following equations are satisfied in~$A$:
\begin{align*}
x&=t_0(\vec x,\vec x',\vec y,\vec y',\vec z),\\
y&=t_n(\vec x,\vec x',\vec y,\vec y',\vec z),\\
t_k(\vec y,\vec y',\vec x,\vec x',\vec z)&=t_{k+1}(\vec x,\vec x',\vec y,\vec y',\vec z)\quad(\text{for all }k<n).
\end{align*}
Put $t_k'(\vec a,\vec b,\vec c,\vec d)=t_k(\vec a,\vec d,\vec b,\vec d,\vec c)$, for all tuples $\vec a,\vec b,\vec c,\vec d$ and all $k\le n$. As~$\vec x'/I=\vec y'/I$ the following equations are satisfied in~$A/I$:
\begin{align*}
x/I&=t_0'(\vec x/I,\vec y/I,\vec z/I,\vec x'/I),\\
y/I&=t_n'(\vec x/I,\vec y/I,\vec z/I,\vec x'/I),\\
t_k'(\vec y/I,\vec x/I,\vec z/I,\vec x'/I)&=t_{k+1}'(\vec x/I,\vec y/I,\vec z/I,\vec x'/I)\quad(\text{for all }k<n).
\end{align*}
Therefore $\bA/I$ is congruence-tractable.
\end{proof}

\begin{definition}
The gamp $\bA/I$ described in Proposition~\ref{P:quot} is a \emph{quotient of~$\bA$}, the morphism~$\bpi$ is \emph{the canonical projection}.
\end{definition}

The following proposition describes how morphisms factorize through quotients of gamps.

\begin{proposition}\label{P:quotarrow}
Let $\bff\colon \bA\to\bB$ be a morphism of gamps, let~$I$ be an ideal of~$\widetilde{A}$, and let $J$ be an ideal of~$\widetilde{B}$. Assume that~$\widetilde{f}(I)\subseteq J$ and denote by
 \[
 \bgg\colon(A/I,\delta_{\bA/I},\widetilde A/I)\to(B/J,\delta_{\bB/J},\widetilde B/J)
 \]
the morphism of \pregamps\ induced by $\bff$. The following statements hold.
\begin{enumerate}
\item $\bgg\colon\bA/I\to\bB/J$ is a morphism of gamps.
\item If $\bff$ is strong, then $\bgg$ is strong.
\item If $\bff$ is congruence-cuttable then $\bgg$ is congruence-cuttable.
\item If~$\bA$ and~$\bB$ are gamps of lattices and $\bff$ is congruence-cuttable with chains, then $\bgg$ is congruence-cuttable with chains.
\end{enumerate}
\end{proposition}

\begin{proof}
The equality $g(A^*/I)=f(A^*)/J$ of partial algebras holds. Moreover, as $f(A^*)$ is a partial subalgebra of~$B$, $g(A^*/I)$ is a partial subalgebra of~$B/J$. Therefore $(1)$ holds.

The statement $(2)$ follows from the definitions of a quotient gamp (cf. Proposition~\ref{P:quot}) and of a strong morphism (Definition~\ref{D:propgamp}).

Assume that $\bff$ is congruence-cuttable. As $f(A)$ is a partial subalgebra of~$B^*$ it follows that $g(A/I)$ is a partial subalgebra of~$B^*/J$. Let~$X$ be a finite subset of~$\widetilde B$, let $x,y\in A$ such that $\delta_{\bB/J}(g(x/I),g(y/I))\le \bigvee X/I$. If $X=\emptyset$, then $x/I=y/I$, hence the case is immediate.

If $X\not=\emptyset$, let $u\in J$ such that $\delta_{\bB}(g(x),g(y))\le u\vee\bigvee X$. Put $X'=X\cup\set{u}$. There are $n<\omega$ and $f(x)=x_0,\dots,x_n=f(y)$ in~$B^*$ such that $\delta_{\bB}(x_k,x_{k+1})\in \widetilde B\dnw X'$ for each $k<n$. If $\delta_{\bB}(x_k,y_k)\le u$, then $\delta_{\bB/J}(x_k/J,y_k/J)=\delta_{\bB}(x_k,y_k)/J=0/J\in (\widetilde B/J)\dnw X/J$. Otherwise $\delta_{\bB}(x_k,y_k)\in \widetilde B\dnw X$, thus $\delta_{\bB/J}(x_k/J,y_k/J)=\delta_{\bB}(x_k,y_k)/J\in (\widetilde B/J)\dnw X/J$. Therefore $(3)$ holds.

The proof of $(4)$ is similar to the proof of $(3)$.
\end{proof}

We introduce in the following definitions a functor $\GA\colon\cV\to\GAMP(\cV)$, a functor $\CG\colon \GAMP(\cV)\to\SEM$ and functors $\PGGL,\PGGR\colon\GAMP(\cV)\to\MPALG(\cV)$.

\begin{definition}\label{D:gampfunctors}
Let~$A$ be a member of a variety~$\cV$ of algebras. Then the quadruple $\GA(A)=(A,A,\Theta_A,\Conc A)$ is a gamp of~$\cV$ (we recall that $\Theta_A(x,y)$ denotes the smallest congruence that identifies $x$ and $y$). If $f\colon A\to B$ is a morphism of algebras, then $\GA(f)=(f,\Conc f)$ is a morphism of gamps from $\GA(A)$ to $\GA(B)$. It defines a functor from the category~$\cV$ to the category $\GAMP(\cV)$.

A gamp~$\bA$ is an \emph{algebra} if~$\bA$ is isomorphic to $\GA(B)$ for some~$B$. A gamp~$\bA$ is an \emph{algebra} of a variety~$\cV$ if~$\bA$ is isomorphic to $\GA(B)$ for some~$B\in\cV$. 

Let~$\bA$ be a gamp of~$\cV$, we set $\CG(\bA)=\widetilde{A}$. Let $\bff\colon\bA\to\bB$ be a morphism of gamps of~$\cV$, we set $\CG(\bff)=\widetilde f$. This defines a functor $\CG\colon\GAMP(\cV)\to\SEM$.

Let~$\bA$ be a gamp of~$\cV$, we set $\PGGR(\bA)=(A,\delta_{\bA},\widetilde A)$. Let $\bff\colon\bA\to\bB$ be a morphism of gamps of~$\cV$, we put $\PGGR(\bff)=\bff$ as a morphism of \pregamps. This defines a functor $\PGGR \colon\GAMP(\cV)\to\MPALG(\cV)$.

Let~$\bA$ be a gamp of~$\cV$, we set $\PGGL(\bA)=(A^*,\delta_{\bA}\res (A^*)^2,\widetilde A)$. Let $\bff\colon\bA\to\bB$ be a morphism of gamps of~$\cV$, we denote by $\PGGL(\bff)$ the restriction $(f,\widetilde f)\colon \PGGL(\bA)\to\PGGL(\bB)$. This defines a functor $\PGGL \colon\GAMP(\cV)\to\MPALG(\cV)$.
\end{definition}

\begin{remark}\label{R:functor}
The following assertions hold.
\begin{enumerate}
\item The following equations, between the functors introduced in Definition~\ref{D:pregampcategoryandfunctor} and Definition~\ref{D:gampfunctors}, are satisfied:
\begin{align*}
\CG\circ\GA&=\Conc=\CPG\circ\PGA,\\
\PGGR\circ\GA&=\PGGL\circ\GA=\PGA,\\
\CPG\circ\PGGR&=\CPG\circ\PGGL=\CG.
\end{align*}

\item If~$A$ is a subalgebra of~$B$, then, in general, $\GA(A)$ is not a subgamp of $\GA(B)$. The different ``congruences'' of a subgamp can be extended in a natural way to different ``congruences'' of the gamp.

\item Let $\bA$ be a gamp. If~$\bA$ is an algebra, then there is a unique, up to isomorphism, algebra~$B$ such that $\bA\cong\GA(B)$. Moreover if $\bA$ is a gamp in a variety~$\cV$, then $B\in\cV$. Indeed~$A$ is an algebra and $\bA\cong\GA(A)$.

\item Let $\bff\colon\bA\to\bB$ be a morphism of gamps. If $\bA$ and $\bB$ are algebras, then $f\colon A\to B$ is a morphism of algebras.

\item Let $\vec\bA=\famm{\bA_p,\bff_{p,q}}{p\le q\text{ in }P}$ be a diagram of gamps. If $\bA_p$ is an algebra, for all $p\in P$, then $\vec A=\famm{A_p,f_{p,q}}{p\le q\text{ in }P}$ is a diagram of algebras, moreover $\vec\bA\cong \GA\circ\vec A$.

\item Let~$B\in\cV$ and let~$I$ be an ideal of $\Conc B$. There is an isomorphism $\bff\colon\GA(B)/I\cong\GA(B/\bigvee I)$ satisfying $f(x/I)=x/\bigvee I$ and~$\widetilde f(\alpha/I)=\alpha/\bigvee I$ for each $x\in B$ and each $\alpha\in\Conc B$.

\item Let~$B$ be an algebra. Then~$B$ is congruence $n$-permutable if and only if $\GA(B)$ is congruence $n$-permutable (this follows immediately from Proposition~\ref{P:CP-PR}).
\end{enumerate}
\end{remark}

\begin{lemma}\label{L:GAMPhasLimit}
Let~$\cV$ be a variety of algebras. The category $\GAMP(\cV)$ has all directed colimits. Suppose that we are given a directed poset $P$, a $P$-indexed diagram~$\vec \bA=\famm{\bA_p,\bff_{p,q}}{p\le q\text{ in }P}$ in $\GAMP(\cV)$, and a directed colimit cocone:
\[
\famm{(A,\delta_{\bA},\widetilde A),\bff_p}{p\in P}=\varinjlim \famm{(A_p,\delta_{\bA_p},\widetilde A_p),\bff_{p,q}}{p\le q\text{ in }P}\quad\text{ in $\MPALG_{\sL}$.}
\]
Put~$A^*=\bigcup_{p\in P}f(A_p^*)$ with its natural structure of partial algebra \textup(cf. Lemma~\textup{\ref{L:PALGhasLimit}}\textup), then $\bA=(A^*,A,\delta_{\bA},\widetilde{A})$ is a gamp of~$\cV$, $\bff_{p}\colon\bA_p\to\bA$ is a morphism of gamps, and the following is a directed colimit cocone:
\[
\famm{\bA,\bff_p}{p\in P}=\varinjlim \famm{\bA_p,\bff_{p,q}}{p\le q\text{ in }P},\quad\text{ in }\GAMP(\cV).
\]
Moreover the following statements hold:
\begin{enumerate}
\item If $\bA_p$ is distance-generated for each $p\in P$, then~$\bA$ is distance-generated.
\item If~$\cV$ is a variety of lattices and $\bA_p$ is distance-generated with chains for each $p\in P$, then~$\bA$ is distance-generated with chains.
\item Let~$n$ be a positive integer. If $\bA_p$ is congruence $n$-permutable for each $p\in P$, then~$\bA$ is congruence $n$-permutable.
\end{enumerate}
\end{lemma}

\begin{proof}
It follows from Corollary~\ref{C:limidentity} that $(A,\delta_{\bA},\widetilde A)$ is a \pregamp\ of~$\cV$. Moreover~$A^*$ is a partial subalgebra of~$A$. Hence $\bA=(A^*,A,\delta_{\bA},\widetilde{A})$ is a gamp of~$\cV$. As $f(A_p^*)$ is a partial subalgebra of~$A^*$, $\bff_{p}\colon\bA_p\to\bA$ is a morphism of gamps. It is easy to check that the following is a directed colimit cocone:
\[
\famm{\bA,\bff_p}{p\in P}=\varinjlim \famm{\bA_p,\bff_{p,q}}{p\le q\text{ in }P},\quad\text{ in }\GAMP(\cV).
\]

Assume that $\bA_p$ is distance-generated for each $p\in P$. Let $\alpha\in\widetilde A$, then there are $p\in P$ and $\beta\in\widetilde A_p$ such that $\alpha=\widetilde f_p(\beta)$. As $\bA_p$ is distance-generated, there are an integer $n\ge 0$ and $n$-tuple $\vec x,\vec y$ of~$A_p^*$ such that $\beta=\bigvee_{k<n}\delta_{\bA_p}(x_k,y_k)$. Therefore the following equalities hold:
\[
\alpha=\widetilde f_p(\beta)=\widetilde f_p\left(\bigvee_{k<n}\delta_{\bA_p}(x_k,y_k)\right)=\bigvee_{k<n}\widetilde f_p(\delta_{\bA_p}(x_k,y_k))=\bigvee_{k<n}\delta_{\bA}(f_p(x_k),f_p(y_k)).
\]
Thus~$\bA$ is distance-generated.

The proofs of $(2)$ and $(3)$ are similar.
\end{proof}

As an immediate application we obtain the following corollary.

\begin{corollary}\label{C:GCGpreslim}
The functors $\GA$, $\CG$, $\PGGL$, and $\PGGR$ preserves directed colimits.
\end{corollary}

\begin{proof}
It follows from the description of directed colimits of gamps (cf. Lemma~\ref{L:GAMPhasLimit}) and \pregamps\ (cf. Lemma~\ref{L:MPALGhasLimit}) that $\CG$, $\PGGL$, and $\PGGR$ preserve directed colimits. As $\Conc$ preserves directed colimits, $\GA$ also preserves directed colimits.
\end{proof}

\begin{definition}\label{D:quotdiag}
Let~$\cV$ be a variety of algebras, let $P$ be a poset, and let~$\vec \bA=\famm{\bA_p,\bff_{p,q}}{p\le q\text{ in }P}$ in $\GAMP(\cV)$. An \emph{ideal} of~$\vec\bA$ is an ideal of $\CG\circ\vec A$. It consists of a family~$\vec I=(I_p)_{p\in P}$ such that $I_p$ is an ideal of~$\widetilde A_p$ and~$\widetilde f_{p,q}(I_p)\subseteq I_q$ for all $p\le q$ in~$P$.

We denote by $\vec\bA/\vec I=\famm{\bA_p/I_p,\bgg_{p,q}}{p\le q\text{ in }P}$, where $\bgg_{p,q}\colon \bA_p/I_p\to \bA_q/I_q$ is induced by~$\bff_{p,q}$, for all $p\le q$ in $P$.

The diagram $\vec\bA/\vec I$ is a \emph{quotient} of~$\vec\bA$.
\end{definition}

\begin{remark}\label{R:quotassoc}
In the context of Definition~\ref{D:quotdiag} the following equalities hold:
\[
(\PGGR\circ\vec \bA)/\vec I=\PGGR\circ(\vec \bA/\vec I).
\]
\[
(\CG\circ\vec \bA)/\vec I=\CG\circ(\vec \bA/\vec I).
\]
Moreover, up to a natural identification (cf. Notation~\ref{N:quotpregamp})
\[
(\PGGL\circ\vec \bA)/\vec I=\PGGL\circ(\vec \bA/\vec I).
\]
\end{remark}

\begin{definition}\label{D:partiallifting}
Let~$\cV$ be a variety of algebras, let $P$ be a poset. A \emph{partial lifting} in~$\cV$ is a diagram $\vec\bA=\famm{\bA_p,\bff_{p,q}}{p\le q\text{ in }P}$ in $\GAMP(\cV)$ such that the following statements hold:
\begin{enumerate}
\item The gamp $\bA_p$ is strong, congruence-tractable, distance-generated, and has an isomorphic realization (cf. Definitions~\ref{D:propgamp} and \ref{D:gamp}), for each $p\in P$.
\item The morphisms~$\bff_{p,q}$ is strong and congruence-cuttable (cf. Definition~\ref{D:propgamp}), for all $p<q$ in $P$.
\end{enumerate}

The partial lifting is a \emph{lattice partial lifting} if~$\cV$ is a variety of lattices, $\bB_p$ is distance-generated with chains for each $p\in P$, and~$\bff_{p,q}$ is congruence-cuttable with chains for all $p<q$ in $P$.

A partial lifting~$\vec \bA$ is \emph{congruence $n$-permutable} if $\bA_p$ is congruence $n$-permutable for each $p\in P$.

Let~$\vec S=(S_p,\sigma_{p,q})$ be a diagram in $\SEM$. A \emph{partial lifting of~$\vec S$} is a partial lifting of~$\vec\bA$ such that $\CG\circ\vec \bA\cong\vec S$.
\end{definition}

\begin{remark}\label{Rk:Iso2Equal}
If~$\vec \bA$ is a partial lifting of~$\vec S$, then there exists a diagram $\vec{\bA}'\cong\vec\bA$ such that $\CG\circ\vec{\bA}'=\vec S$. Hence we can assume that $\CG\circ\vec\bA=\vec S$.
\end{remark}

The following result expresses the fact that a subdiagram or a quotient of a partial lifting is a partial lifting.

\begin{lemma}
Let~$\cV$ be a variety of algebras, let $P$ be a poset, and let~$\vec \bA$ be a partial lifting in~$\cV$ of a diagram~$\vec S$. The following statements hold:
\begin{enumerate}
\item Let~$\vec I$ be an ideal of~$\vec\bA$; then $\vec\bA/\vec I$ is a partial lifting of~$\vec S/I$.
\item Let $Q\subseteq P$; then $\vec\bA\res Q$ is a partial lifting of~$\vec S\res Q$.
\end{enumerate}
\end{lemma}

\begin{proof}
The statement $(1)$ follows from Proposition~\ref{P:quot} and Proposition~\ref{P:quotarrow}. The statement $(2)$ is immediate.
\end{proof}

\begin{lemma}\label{L:LimPartLiftIsAlgebra}
Let~$\cV$ be a variety of algebras, let $P$ be a directed poset with no maximal element, and let~$\vec\bA=\famm{\bA_p,\bff_{p,q}}{p\le q\text{ in }P}$ be a partial lifting in~$\cV$. Consider a colimit cocone:
 \[
 \famm{\bA,\bff_p}{p\in P}=\varinjlim\vec\bA\quad\text{ in }\GAMP(\cV).
\]
Then~$\bA$ is an algebra in~$\cV$. Moreover, for any~$n\ge 2$, if all~$\bA_p$ are congruence $n$-permutable, then the algebra corresponding to~$\bA$ is congruence $n$-permutable.
\end{lemma}

\begin{proof}
The morphism $\bff_{p,q}\colon(A_p,\delta_{\bA_p},\widetilde A_p)\to(A_q,\delta_{\bA_q},\widetilde A_q)$ of \pregamps\ is both strong and congruence-tractable for all $p<q$ in $P$. It follows from the description of colimits in $\GAMP(\cV)$ (cf. Lemma~\ref{L:GAMPhasLimit} and Corollary~\ref{C:directlimisalgebra}) that~$A$ is an algebra and there is an isomorphism $\phi\colon\Conc A\to\widetilde A$ satisfying:
\[\phi(\Theta_A(x,y))=\delta_{\bA}(x,y),\quad\text{ for all $x,y\in A$.}\]
As~$A^*=\bigcup_{p\in P}f_p(A_p^*)$ and $f_{p,q}(A_p)\subseteq A_q^*$ for all $p<q$ in $P$, it follows that~$A^*=A$. Therefore $(\id_A,\phi)\colon\GA(A)\to\bA$ is an isomorphism of gamps.

Now assume that all~$\bA_p$ are congruence $n$-permutable. It follows from~Lemma~\ref{L:GAMPhasLimit} that~$\bA$ is congruence $n$-permutable. As~$\bA$ is an algebra, the conclusion follows from Proposition~\ref{P:CP-PR}.
\end{proof}

\section{Locally finite properties}\label{S:LFP}

The aim of this section is to prove Lemma~\ref{L:LSGamp}, which is a special version of the Buttress Lemma of~\cite{GiWe1} adapted to gamps for the functor $\CG$.

We use the following generalizations of $(2)$, $(3)$, $(4)$, $(6)$, and $(7)$ of Definition~\ref{D:propgamp}.

\begin{definition}
Fix a gamp~$\bA$ in a variety~$\cV$, a \jzs~$S$, and a \jzh\ $\phi\colon\widetilde A\to S$.
The gamp~$\bA$ is \emph{distance-generated through} $\phi$ if the following statement holds:
\begin{enumerate}
\item[$(2')$] For each $s\in S$ there are $n<\omega$ and $n$-tuple~$\vec x,\vec y$ of~$A^*$ such that:
\[
s=\bigvee_{k<n}\phi(\delta_{\bA}(x_k,y_k))
\]
\end{enumerate}
If~$\bA$ is a gamp of lattices, we say~$\bA$ is \emph{distance-generated with chains through $\phi$} if 
\begin{enumerate}
\item[$(3')$] For each $s\in S$ there are $n<\omega$ and chains $x_0<y_0,\dots,x_{n-1}<y_{n-1}$ of~$\bA$ such that:
\[
s=\bigvee_{k<n}\phi(\delta_{\bA}(x_k,y_k))
\]
\end{enumerate}
The gamp~$\bA$ is \emph{congruence-tractable through $\phi$} if the following statement holds:
\begin{enumerate}
\item[$(4')$] Let $x,y, x_0,y_0,\dots,x_{m-1},y_{m-1}$ in~$A^*$, if $\phi(\delta_{\bA}(x,y))\le\bigvee_{k<m}\phi(\delta_{\bA}(x_k,y_k))$ then there are a positive integer~$n$, a list~$\vec z$ of parameters from~$A$, and terms $t_0$, \dots, $t_n$ such that, the following equations are satisfied in~$A$:
\begin{align*}
\phi(\delta_{\bA}(x,t_0(\vec x,\vec y,\vec z)))&=0,\\
\phi(\delta_{\bA}(y,t_n(\vec x,\vec y,\vec z)))&=0,\\
\phi(\delta_{\bA}(t_j(\vec y,\vec x,\vec z),t_{j+1}(\vec x,\vec y,\vec z)))&=0\quad(\text{for all }j<n).
\end{align*}
\end{enumerate}
A morphism $\bff\colon\bU\to\bA$ of gamps is \emph{congruence-cuttable through $\phi$} if the following statement holds:
\begin{enumerate}
\item[$(6')$] Given $X\subseteq S$, given $x,y$ in~$U$, if $\phi(\widetilde{f}(\delta_{\bA}(x,y)))\le\bigvee X$ then there exist $n<\omega$ and $f(x)=x_0, x_1,\dots, x_n=f(y)$ in~$A^*$ such that $\delta_{\bA}(x_k,x_{k+1})\in S\dnw X$ for all $k<n$.
\end{enumerate}
A morphism $\bff\colon\bU\to\bA$ of gamps of lattices is \emph{congruence-cuttable with chains through $\phi$} if the following statement holds:
\begin{enumerate}
\item[$(7')$] Given $X\subseteq S$, given $x,y\in U$, if $\phi(\widetilde{f}(\delta_{\bA}(x,y)))\le\bigvee X$ then there are $n<\omega$ and a chain $x_0<x_1<\dots<x_n$ of~$\bA$ such that $x_0=f(x)\wedge f(y)$, $x_n=f(x)\vee f(y)$, and $\phi(\delta_{\bA}(x_k,x_{k+1}))\in S\dnw X$ for all $k<n$.
\end{enumerate}
\end{definition}

\begin{remark}
A gamp~$\bA$ is distance-generated through $\id_{\widetilde A}$ if and only if~$\bA$ is distance-generated.

A gamp~$\bA$ is congruence-tractable through $\id_{\widetilde A}$ if and only if~$\bA$ is congruence-tractable.

A morphism $\bff\colon\bU\to\bA$ of gamps is congruence-cuttable through $\id_{\widetilde A}$ if and only if $\bff$ is congruence-cuttable.

A morphism $\bff\colon\bU\to\bA$ of gamps of lattices is congruence-cuttable with chains through $\id_{\widetilde A}$ if and only if $\bff$ is congruence-cuttable with chains.
\end{remark}

\begin{lemma}\label{L:gampthrough1}
Let~$\bA$ be a gamp in a variety of algebras~$\cV$, let $\phi\colon\widetilde A\to S$ an ideal-induced \jzh, and put $I=\ker_0\phi$.
\begin{enumerate}
\item If~$\bA$ is distance-generated through $\phi$, then $\bA/I$ is distance-generated.
\item If~$\bA$ is distance-generated with chains through $\phi$, then $\bA/I$ is distance-generated with chains.
\item If~$\bA$ is congruence-tractable through $\phi$ then $\bA/I$ is congruence-tractable.
\end{enumerate}
\end{lemma}

\begin{proof}
As $\phi$ is ideal-induced, it induces an isomorphism $\xi\colon \widetilde A/I\to S$.

Assume that~$\bA$ is distance-generated through $\phi$. Let $d\in \widetilde A/I$, put $s=\xi(d)$. There are $n<\omega$ and $n$-tuples~$\vec x,\vec y$ in~$A^*$ such that $s=\bigvee_{k<n}\phi(\delta_{\bA}(x_k,y_k))$. It follows that $s=\bigvee_{k<n}\xi(\delta_{\bA/I}(x_k/I,y_k/I))$, so $d=\bigvee_{k<n}\delta_{\bA/I}(x_k/I,y_k/I)$. Therefore $\bA/I$ is distance-generated.

The case where~$\bA$ is distance-generated with chains through $\phi$ is similar.

Assume that~$\bA$ is congruence-tractable through $\phi$. Let $x,y, x_0,y_0,\dots,x_{m-1},y_{m-1}$ in~$A^*$ such that:
\[
\delta_{\bA/I}(x/I,y/I)\le\bigvee_{k<m}\delta_{\bA/I}(x_k/I,y_k/I).
\]
Thus $\phi(\delta_{\bA/I}(x,y))\le\bigvee_{k<m}\phi(\delta_{\bA}(x_k,y_k))$. Hence there are a positive integer~$n$, a list~$\vec z$ of parameters from~$A$, and terms $t_0$, \dots, $t_n$ such that the following equations are satisfied in~$A$:
\begin{align*}
\phi(\delta_{\bA}(x,t_0(\vec x,\vec y,\vec z)))&=0,\\
\phi(\delta_{\bA}(y,t_n(\vec x,\vec y,\vec z)))&=0,\\
\phi(\delta_{\bA}(t_k(\vec y,\vec x,\vec z),t_{k+1}(\vec x,\vec y,\vec z)))&=0\quad(\text{for all }k<n).
\end{align*}
Those equations imply that the following equations are satisfied in~$A/I$:
\begin{align*}
x/I&=t_0(\vec x/I,\vec y/I,\vec z/I),\\
y/I&=t_n(\vec x/I,\vec y/I,\vec z/I),\\
t_k(\vec y/I,\vec x/I,\vec z/I)&=t_{k+1}(\vec x/I,\vec y/I,\vec z/I)\quad(\text{for all }k<n).
\end{align*}
Therefore $\bA/I$ is congruence-tractable.
\end{proof}

The proof of the following lemma is similar.

\begin{lemma}\label{L:gampthrough2}
Let $\bff\colon \bU\to\bA$ be a morphism of gamps, let~$I$ be an ideal of $\bU$, and let $\phi\colon\widetilde A\to S$ be an ideal-induced \jzh. Put $J=\ker_0\phi$. Denote by $\bgg\colon\bU/I\to\bA/J$ the morphism of gamps induced by $\bff$. The following statements hold:
\begin{enumerate}
\item If $\bff$ is congruence-cuttable through $\phi$, then $\bgg$ is congruence-cuttable.
\item Assume that $\bff$ is a morphism of gamps of lattices. If $\bff$ is congruence-cuttable with chains through $\phi$, then $\bgg$ is congruence-cuttable with chains.
\end{enumerate}
\end{lemma}

We shall now define \emph{locally finite properties} for an algebra~$B$, as properties that are satisfied by ``many'' finite subgamps of $\GA(B)$.

\begin{definition}\label{D:LFF}
Let~$B$ be an algebra. A \emph{locally finite property for~$B$} is a property~$(P)$ in a subgamp~$\bA$ of $\GA(B)$ such that there exists a finite $X\subseteq B$ satisfying that for every finite full partial subalgebra~$A^*$ of~$B$ that contains~$X$, there exists a finite $Y\subseteq B$ such that for every finite full partial subalgebra~$A$ of~$B$ that contains~$A^*\cup Y$ and every finite \jz-subsemilattice~$\widetilde{A}$ of $\Conc B$ that contains $\Conc^A(B)$, the subgamp $(A^*,A,\Theta_B,\widetilde{A})$ of $\GA(B)$ satisfies $(P)$.
\end{definition}

\begin{proposition}\label{P:cunjlff}
Any finite conjunction of locally finite properties is locally finite.
\end{proposition}

\begin{lemma}\label{L:exLFF}
Let~$B$ be an algebra and denote by~$\sL$ the similarity type of~$B$. The properties $(1)$-$(8)$ in~$\bA$ are locally finite for~$B$:

Assume that~$\sL$ is finite.
\begin{enumerate}
\item[$(1)$]~$\bA$ is strong.
\end{enumerate}
Fix an ideal-induced \jzh\ $\phi\colon\Conc B\to S$, with~$S$ finite.
\begin{enumerate}
\item[$(2)$]~$\bA$ is distance-generated through $\phi\res\widetilde A$.
\end{enumerate}
Fix an ideal-induced \jzh\ $\phi\colon\Conc B\to S$ with~$S$ finite, and assume that~$B$ is a lattice.
\begin{enumerate}
\item[$(3)$]~$\bA$ is distance-generated with chains through $\phi\res\widetilde A$.
\end{enumerate}
Fix an ideal-induced \jzh\ $\phi\colon\Conc B\to S$, with~$S$ finite.
\begin{enumerate}
\item[$(4)$]~$\bA$ is congruence-tractable through $\phi\res\widetilde A$.
\end{enumerate}
Assume that~$\sL$ is finite, fix a finite gamp $\bU$, fix a morphism $\bff\colon\bU\to\GA(B)$ of gamps.
\begin{enumerate}
\item[$(5)$] The restriction $\bff\colon\bU\to\bA$ is strong.
\end{enumerate}
Fix an ideal-induced \jzh\ $\phi\colon \Conc B\to S$ where~$S$ is finite, a finite gamp $\bU$, and a morphism $\bff\colon\bU\to\GA(B)$ of gamps.
\begin{enumerate}
\item[$(6)$] The restriction of $\bff\colon\bU\to\bA$ is congruence-cuttable through $\phi\res\widetilde A$.
\end{enumerate}
Assume that~$B$ is a lattice. Fix an ideal-induced \jzh\ $\phi\colon \Conc B\to S$, where~$S$ is finite. Fix a finite gamp $\bU$, fix a morphism $\bff\colon\bU\to\GA(B)$ of gamps. 
\begin{enumerate}
\item[$(7)$] The restriction of $\bff\colon\bU\to\bA$ is congruence-cuttable with chains through $\phi\res\widetilde A$.
\end{enumerate}
Let $n\ge 2$ be an integer. Assume that~$B$ is congruence $n$-permutable.
\begin{enumerate}
\item[$(8)$]~$\bA$ is congruence $n$-permutable.
\end{enumerate}
\end{lemma}

\begin{proof}
$(1)$ Put $X=\emptyset$ and let~$A^*$ be a finite full partial subalgebra of~$B$. Put $Y=\setm{\ell^B(\vec x)}{\ell\in\sL\text{ and~$\vec X$ is an $\ari(\ell)$-tuple of $A^*$}}$. As~$A^*$ and $\sL$ are both finite, $Y$ is also finite. Let~$A$ be a finite full partial subalgebra of~$B$ that contains~$A^*\cup Y$, let~$\widetilde A$ be a finite \jz-subsemilattice of $\Conc B$ containing $\Conc^A(B)$. As $\ell(\vec x)\in Y\subseteq A$, it is defined in~$A$, for each $\ell\in\sL$ and each $\ari(\ell)$-tuple~$\vec X$ of~$A$.

$(2)$ Let $s\in S$. As $\phi$ is surjective, there exists $\theta\in\Conc B$ such that $s=\phi(\theta)$. So there exist $n<\omega$ and $n$-tuple~$\vec x,\vec y$ of~$B$ such that $s=\phi(\bigvee_{k<n}\Theta_B(x_k,y_k))$. Put $X_s=\set{x_0,\dots,x_{n-1},y_0,\dots,y_{n-1}}$.

Put $X=\bigcup_{s\in S} X_s$. As~$S$ is finite and $X_s$ is finite for all $s\in S$,~$X$ is finite. Let~$A^*$ be a finite full partial subalgebra of~$B$ that contains~$X$. Put $Y=\emptyset$. Let~$A$ be a finite full partial subalgebra of~$B$ that contains~$A^*\cup Y$. Let~$\widetilde A$ be a finite \jz-subsemilattice of $\Conc B$ containing $\Conc^A(B)$. By construction $(A^*,A,\Theta_B,\widetilde A)$ satisfies $(2)$.

The proof that $(3)$ is a locally finite property is similar.

$(4)$ Put $X=\emptyset$, let~$A^*$ be a finite full partial subalgebra of~$B$. Denote by $E$ the set of all quadruples $(x,y,\vec x,\vec y)$ such that the following statements are satisfied:
\begin{itemize}
\item $x,y\in A^*$.
\item~$\vec X$ and $\vec y$ are $m$-tuples of~$A^*$, for some $m<\omega$.
\item $\phi(\Theta_B(x,y))\le\bigvee_{k<m}\phi(\Theta_B(x_k,y_k))$.
\item $(x_i,y_i)\not=(x_j,y_j)$ for all $i<j<m$.
\end{itemize}

Put $I=\ker_0\phi$ and let $(x,y,\vec x,\vec y)\in E$. As $\phi$ is ideal-induced, there exists $\alpha\in I$ such that $\Theta_B(x,y)\le \bigvee_{k<m}\Theta_B(x_k,y_k)\vee\alpha$. Let~$\vec x',\vec y'$ be $p$-tuples of~$B$ such that $\alpha=\bigvee_{k<p}\Theta(x'_k,y'_k)$. Hence:
\[\Theta_B(x,y)\le \bigvee_{k<m}\Theta_B(x_k,y_k)\vee\bigvee_{k<p}\Theta_B(x_k',y_k').
\]
It follows from Lemma~\ref{L:Condcompcongruences} that there are a positive integer~$n$, a list~$\vec z$ of parameters from~$B$, and terms $t_0$, \dots, $t_n$ such that the following equalities hold in~$B$:
\begin{align}
x&=t_0(\vec x,\vec x',\vec y,\vec y',\vec z),\label{EA:1}\\
y&=t_n(\vec x,\vec x',\vec y,\vec y',\vec z),\label{EA:2}\\
t_j(\vec y,\vec y',\vec x,\vec x',\vec z)&=t_{j+1}(\vec x,\vec x',\vec y,\vec y',\vec z),\quad\text{for all }j<n\label{EA:3}.
\end{align}
Put $t'_j(\vec a,\vec b,\vec c,\vec d)=t_j(\vec a,\vec d,\vec b,\vec d,\vec c)$ for all tuples~$\vec a,\vec b,\vec c,\vec d$ of~$B$ of appropriate length. The following inequalities follow from the compatibility of $\Theta_B$ with terms:
\begin{align}
\Theta_B(t'_j(\vec x,\vec y,\vec z,\vec x'),t_j(\vec x,\vec x',\vec y,\vec y',\vec z))\le\bigvee_{k<p}\Theta_B(x'_k,y'_k)=\alpha,\label{EA:4}\\
\Theta_B(t'_j(\vec y,\vec x,\vec z,\vec x'),t_j(\vec y,\vec y',\vec x,\vec x',\vec z))\le\bigvee_{k<p}\Theta_B(x'_k,y'_k)=\alpha.\label{EA:5}
\end{align}
As $\phi(\alpha)=0$, it follows from~\eqref{EA:4} and~\eqref{EA:1} that:
\[
\phi(\Theta_B(x,t_0'(\vec x,\vec y,\vec z,\vec x')))=0
\]
Set~$\vec z'=(\vec z,\vec x')$. As $\phi(\alpha)=0$, it follows from~\eqref{EA:1}-\eqref{EA:5} that
\begin{align*}
\phi(\Theta_B(x,t_0'(\vec x,\vec y,\vec z')))&=0,\\
\phi(\Theta_B(y,t_n'(\vec x,\vec y,\vec z')))&=0,\\
\phi(\Theta_B(t_j'(\vec y,\vec x,\vec z'),t_{j+1}'(\vec x,\vec y,\vec z')))&=0\quad(\text{for all }j<n).
\end{align*}
Let $Y_{(x,y,\vec x,\vec y)}$ be a finite partial subalgebra of~$B$ such that both $t_j'(\vec x,\vec y,\vec z,\vec x')$ and $t_j'(\vec y,\vec x,\vec z,\vec x')$ are defined in $Y_{(x,y,\vec x,\vec y)}$, for each $j\le n$.

Put $Y=\bigcup\famm{Y_{e}}{e\in E}$. As $Y_{e}$ is finite for each ${e}\in E$ and $E$ is finite, $Y$ is finite. Let~$A$ be a finite full partial subalgebra of~$B$ that contains $Y\cup A^*$, and let~$\widetilde A$ be a finite \jz-subsemilattice of $\Conc B$ containing $\Conc^A(B)$. It is not hard to verify that $(A^*,A,\Theta_B,\widetilde A)$ satisfies $(4)$.

$(5)$ As~$U$ is finite, the set $X_1=\gen{f(U)}^1_B$ (cf. Notation~\ref{N:gen}) is also finite. As~$\widetilde U$ is finite and each element of~$\widetilde{f}(\widetilde U)$ is a compact congruence of~$B$, we can choose a finite subset~$X_2$ of~$B$ such that~$\widetilde{f}(\widetilde U)\subseteq\Conc^{X_2}(B)$. The set $X=X_1\cup X_2$ is finite. Let~$A^*$ be a finite full partial subalgebra of~$B$ containing~$X$. Put $Y=\emptyset$, let~$A$ be a full partial subalgebra of~$B$ containing $Y\cup A^*$, and let~$\widetilde A$ be a finite \jz-subsemilattice of $\Conc B$ containing $\Conc^A(B)$. The following containments hold:
\[
\widetilde{f}(\widetilde U)\subseteq \Conc^{X_2}(B)\subseteq \Conc^{A}(B)\subseteq \widetilde A.
\]
Moreover $\gen{f(U)}^1_B\subseteq A^*$.

The proofs that $(6),(7)$, and $(8)$ are locally finite properties are similar.
\end{proof}

The following lemma is an analogue for gamps of the \emph{Buttress Lemma} (cf. \cite{GiWe1}).

\begin{lemma}\label{L:LSGamp}
Let~$\cV$ be a variety of algebras in a finite similarity type, let $P$ be a lower finite poset, let $(S_p)_{p\in P}$ be a family of finite \jzs s, let~$B\in\cV$, and let $(\phi_p)_{p\in P}$ be a family of \jzh s where $\phi_p\colon \Conc B\to S_p$ is ideal-induced for each $p\in P$. There exists a diagram~$\vec \bA=\famm{\bA_p,\bff_{p,q}}{p\le q\text{ in }P}$ of finite subgamps of $\GA(B)$ such that the following statements hold:
\begin{enumerate}
\item $\phi_p\res\widetilde A_p$ is ideal-induced for each $p\in P$.
\item $\bA_p$ is strong, distance-generated through $\phi_p\res\widetilde A_p$, and congruence-tractable through $\phi_p\res\widetilde A_p$, for each $p\in P$.
\item $\bA_p$ is a subgamp of $\bA_q$ and $\bff_{p,q}$ is the canonical embedding, for all $p\le q$ in $P$. 
\item $\bff_{p,q}$ is strong and congruence-cuttable through $\phi_q\res\widetilde A_q$, for all $p<q$ in $P$.
\end{enumerate}

If~$B$ is a lattice, then we can construct~$\vec\bA$ such that~$\bff_{p,q}$ is congruence-cuttable with chains through $\phi_q\res\widetilde A_q$ for all $p<q$ in $P$, and $\bB_p$ is distance-generated with chains through $\phi_p$ for each $p\in P$.

If~$B$ is congruence $n$-permutable \pup{where $n$ is a positive integer}, then we can construct~$\vec\bA$ such that $\bA_p$ is congruence $n$-permutable for each $p\in P$.
\end{lemma}

\begin{proof}
Let $r\in P$, suppose having constructed a diagram $\famm{\bA_p,\bff_{p,q}}{p\le q<r}$ of finite subgamps of $\GA(B)$ such that the following statements hold:
\begin{enumerate}
\item $\phi_p\res\widetilde A_p$ is ideal-induced for each $p<r$.
\item $\bA_p$ is strong, distance-generated through $\phi_p\res\widetilde A_p$, and congruence-tractable through $\phi_p\res\widetilde A_p$, for each $p<r$.
\item $\bA_p$ is a subgamp of $\bA_q$ and $\bff_{p,q}$ is the canonical embedding, for all $p\le q<r$. 
\item $\bff_{p,q}$ is strong and congruence-cuttable through $\phi_q\res\widetilde A_q$, for all $p<q<r$.
\end{enumerate}

The following property in~$\bA$ a subgamp of $\GA(B)$ is locally finite (see Proposition~\ref{P:cunjlff} and Lemma~\ref{L:exLFF})
\begin{itemize}
\item[$(F)$]~$\bA$ is strong, distance-generated through $\phi_r\res\widetilde A$, congruence-tractable through $\phi_r\res\widetilde A$, and the canonical embedding $\bff\colon\bA_p\to\bA$ is strong and congruence-cuttable through $\phi_r\res\widetilde A$, for all $p<r$.
\end{itemize}
Thus there exist finite partial subalgebras~$A^*_r$ and~$A_r$ of~$B$ such that for each~$\widetilde A$ containing $\Conc^{A_r}(B)$, the gamp $(A^*_r,A_r,\Theta_B,\widetilde A)$ satisfies $(F)$. Moreover it follows from Proposition~\ref{P:IdeIdu} that there exists a finite \jz-subsemilattice~$\widetilde A_r$ of $\Conc B$, such that $\Conc^{A_r}(B)\subseteq \widetilde A_r$ and $\phi_r\res\widetilde A_r$ is ideal-induced.

Set $\bA_r=(A^*_r,A_r,\Theta_B,\widetilde A_r)$, and $\bff_{p,r}$ the canonical embedding for each $p\le r$. By construction, the diagram $\famm{\bA_p,\bff_{p,q}}{p\le q\le r}$ satisfies the required conditions. The construction of $\famm{\bA_p,\bff_{p,q}}{p\le q\text{ in }P}$ follows by induction.

If~$B$ is a lattice, then we can add to the property $(F)$ the condition $\bff\colon\bA_p\to\bA$ is congruence-cuttable with chains through $\phi_r\res\widetilde A$ for each $p<r$, and $\bB_r$ is distance-generated with chains through $\phi_r$.

If~$B$ is congruence $n$-permutable, then we can add to the property $(F)$ the condition $\bA_r$ is congruence $n$-permutable.
\end{proof}

\begin{remark}\label{R:LSGamp}
In the context of Lemma~\ref{L:LSGamp}, if we have a locally finite property for~$B$, then we can assume that any $A_p$ satisfies this property.
\end{remark}

\section{Norm-coverings and lifters}

The aim of this section is to construct a (family) of posets, which we shall use later as an index for a diagram (in \cite{G4}). We also give a combinatorial statement that is satisfied by this poset.

We introduced the following definition in~\cite{G1}.

\begin{definition}\label{D:KerSupp}
A finite subset $V$ of a poset~$U$ is a \emph{kernel} if for every $u\in U$, there exists a largest element $v\in V$ such that $v\le u$. We say that~$U$ is \emph{supported} if every finite subset of~$U$ is contained in a kernel of~$U$.
\end{definition}

It is not hard to verify that this definition of a supported poset is equivalent to the one used in~\cite{GiWe1}.

The following definition introduced in \cite{G1} also appears in \cite{GiWe1} in a weaker form. Nevertheless, in the context of $\aleph_0$-lifters (cf. Definition~\ref{D:Lifter}), all these definitions are equivalent.

\begin{definition}\label{D:normcovering}
A \emph{norm-covering} of a poset $P$ is a pair $(U,\partial)$, where~$U$ is a supported poset and $\partial\colon U\to P$, $u\mapsto\partial u$ is an isotone map.

We say that an ideal~$\bu$ of~$U$ is \emph{sharp} if the set $\setm{\partial u}{u\in\bu}$ has a largest element, which we shall then denote by~$\partial{\bu}$. We shall denote by $\Ids U$ the set of all sharp ideals of~$U$, partially ordered by inclusion.
\end{definition}

We remind the reader about the following definition introduced in \cite{GiWe1}.

\begin{definition}\label{D:Lifter}
Let~$P$ be a poset. An \emph{$\aleph_0$-lifter} of~$P$ is a pair $(U,\bU)$, where~$U$ is a norm-covering of~$P$ and~$\bU$ is a subset of~$\Ids U$ satisfying the following properties:
\begin{enumerate}
\item The set $\bU^==\setm{\bu\in\bU}{\partial\bu\text{ is not maximal in }P}$ is lower finite, that is, the set $\bU\dnw\bu$ is finite for each~$\bu\in\bU^=$.
 
\item For every map $S\colon\bU^=\to[U]^{<\omega}$, there exists an isotone map $\sigma\colon P\to\bU$ such that
\begin{enumerate}
\item the map $\sigma$ is a \emph{section} of $\partial$, that is, $\partial\sigma(p)=p$ holds for each $p\in P$;
\item the containment $S(\sigma(p))\cap\sigma(q)\subseteq\sigma(p)$ holds for all $p<q$ in~$P$.
(\emph{Observe that~$\sigma(p)$ belongs to~$\bU^=$}.)
\end{enumerate}
\end{enumerate}
\end{definition}

The existence of lifters is related to the following infinite combinatorial statement introduced in~\cite{GiWe2}.

\begin{definition}\label{D:InfCombP}
For cardinals $\kappa$, $\lambda$ and a poset~$P$, let $(\kappa,{<}\lambda)\leadsto P$ hold if for every mapping $F\colon\Pow(\kappa)\to[\kappa]^{<\lambda}$, there exists a one-to-one map $f\colon P\toinj\nobreak\kappa$ such that
 \begin{equation*}
 F(f(P\dnw p))\cap f(P\dnw q)\subseteq f(P\dnw p)\,,
 \qquad\text{for all }p\le q\text{ in }P\,.
 \end{equation*}
 
Notice that in case $P$ is lower finite, it is sufficient to verify the conclusion above for all $F\colon[\kappa]^{<\omega}\to [\kappa]^{<\lambda}$ isotone and all $p\prec q$ in $P$.
\end{definition}

\begin{lemma}\label{L:squarehaslifter}
The square poset has an~$\aleph_0$-lifter $(X,\bX)$ such that $\card X=\aleph_1$.
\end{lemma}

\begin{proof}
By \cite[Proposition~4.7]{GiWe2}, the Kuratowski index (cf. \cite[Definition~4.1]{GiWe2}) of the square~$P$ is less than or equal to its order-dimension, which is equal to 2. Hence, by the definition of the Kuratowski index, $(\kappa^+,{<}\kappa)\leadsto P$ for every infinite cardinal~$\kappa$. Therefore, by \cite[Corollary~3-5.8]{GiWe1}, $P$ has an $\aleph_0$-lifter $(X,\bX)$ such that $\card X=\card\bX=\aleph_1$.
\end{proof}

Given a poset $P$, we introduce a new poset which looks like a lexicographical product of $P$ with a tree. This construction is mainly used in \cite{G4}.

\begin{definition}\label{D:newposet}
Let~$P$ be a poset with a smallest element, let $X\subseteq P$, let $\vec R=(R_x)_{x\in X}$ be a family of sets, let $\alpha\le\omega$. Consider the following poset:
\[
T=\setm{(n,\vec x,\vec r)}{n<\alpha,\ \vec x\in X^n,\text{ and }\vec r\in R_{x_0}\times\dots\times R_{x_{n-1}}},
\]
ordered by $(m,\vec x,\vec r)\le (n,\vec y,\vec s)$ if and only if $m\le n$, $\vec x=\vec y\res m$, and $\vec r=\vec s\res m$. Recall that $\vec y\res m=(y_0,\dots,y_{m-1})$. Given $t=(n,\vec x,\vec r)\in T$ and $m\le n$, we set $t\res m=(m,\vec x\res m,\vec r\res m)$.

Put:
\[
A=\kposet{P}{X}{\vec R}{\alpha} = T\times P=\bigcup_{n\in\alpha}\bigcup_{\vec x\in X^n}\Big(\set{n}\times \set{\vec x}\times( R_{x_0}\times\dots\times R_{x_{n-1}})\times P\Big).
\]
Any element of $A$ can be written $(n,\vec x,\vec r,p)$ with $n<\alpha$, $\vec x\in X^n$ and $\vec r\in R_{x_0}\times\dots\times R_{x_{n-1}}$.

We define an order on $A$ by $(m,\vec x,\vec r,p)\le (n,\vec y,\vec s,q)$ if and only if the following conditions hold:
\begin{enumerate}
\item $(m,\vec x,\vec r)\le (n,\vec y,\vec s)$.
\item If $m=n$ then $p\le q$.
\item If $m<n$ then $p\le y_{m}$.
\end{enumerate}
\end{definition}

\begin{remark}\label{R:prodlf}
In the context of Definition~\ref{D:newposet}, notice that~$T$ is a lower finite tree. Indeed, if $(n,\vec x,\vec r)\in T$, then $T\dnw (n,\vec x,\vec r)=\setm{(m,\vec x\res m,\vec r\res m)}{m\le n}$ is a chain of length~$n$. The tree $T$ is called \emph{the tree associated to $\kposet{P}{X}{\vec R}{\alpha}$}.

The following statements hold:
\begin{enumerate}
\item If~$P$ is lower finite, then $A$ is lower finite.
\item The inequality $\card A\le\aleph_0 +\card P+\sum_{x\in X}\card R_x$ holds.
\item The inequality $\card T\le\aleph_0+\sum_{x\in X}\card R_x$ holds.
\end{enumerate}
\end{remark}

\begin{remark}\label{R:coverinkprod}
In the context of Definition~\ref{D:newposet}, if $a<b$ in $A$, then there are $t=(n,\vec x,\vec r)\in T$, $p,q\in P$, and $m\le n$ such that $a=(t\res m,p)$ and $b=(t,q)$. Moreover, if $m<n$, then $(t\res m,p)<(t\res(m+1),0)\leq(t,q)$. It follows easily that $a\prec b$ if and only if exactly one of the following statements holds:
\begin{enumerate}
\item $m=n$ and $p\prec q$.
\item $n=m+1$, $p=x_m$, and $q=0$.
\end{enumerate}
As a consequence, we obtain immediately that \emph{If $P$, $X$, and all $R_x$ are finite, then each $a\in A$ has only finitely many covers}.
\end{remark}

\begin{lemma}\label{L:prodhascomb}
Let $\kappa\ge \lambda$ be infinite cardinals, let~$P$ be a lower finite $\kappa$-small poset with a smallest element, let $X\subseteq P$, let $\vec R=(R_x)_{x\in X}$ be a family of $\kappa$-small sets, and let $\alpha\le\omega$. If $(\kappa,<\lambda)\leadsto P$, then $(\kappa,<\lambda)\leadsto \kposet{P}{X}{\vec R}{\alpha}$.
\end{lemma}

\begin{proof}
Denote by $T$ the tree associated to $A=\kposet{P}{X}{\vec R}{\alpha}$. It follows from Remark~\ref{R:prodlf} together with the assumptions on cardinalities that the following inequalities hold:
\[
\card T\le\aleph_0 +\sum_{x\in X}\card R_x\le\aleph_0+\sum_{x\in X}\kappa\le\kappa.
\]
Thus there exists a partition $(K_t)_{t\in T}$ of~$\kappa$ such that $\card K_t=\kappa$ for each $t\in T$.

Notice that $A$ is lower finite. Let $F\colon[\kappa]^{<\omega}\to[\kappa]^{<\lambda}$ isotone, let $t\in T$. Assume having constructed, for each $s<t$, a one-to-one map $\sigma_s\colon P\toinj K_s$ such that setting
\[
S_{s}=\setm{\sigma_{s\res m}(p)}{m<n\text{ and }p\le x_m},\text{ for all $s=(n,\vec x,\vec r)\le t$,}
\]
the following containments hold:
\begin{align*}
\rng\sigma_s&\subseteq K_s-F(S_s),&&\text{for each $s<t$,}\\
F\big(\sigma_s(P\dnw p)\cup S_s\big)\cap \sigma_s(P\dnw q)&\subseteq \sigma_s(P\dnw p),&&\text{for all $p\le q$ in $P$ and all $s<t$.}
\end{align*}
Put $F_t(U)=F(U\cup S_t) - F(S_t)$ for each $U\in[K_t-F(S_t)]^{<\omega}$. As $S_t$ is finite, this defines a map $F_t\colon[K_t-F(S_t)]^{<\omega}\to[K_t-F(S_t)]^{<\lambda}$. As $F(S_t)$ is $\lambda$-small, $\card(K_t-F(S_t))=\kappa$, moreover $(\kappa,{<}\lambda)\leadsto P$, so there exists a one-to-one map $\sigma_t\colon P \toinj K_t-F(S_t)$ such that:
\[
F\big(\sigma_t(P\dnw p)\cup S_t\big)\cap \sigma_t(P\dnw q)\subseteq \sigma_t(P\dnw p),\quad \text{for all $p\le q$ in $P$.}
\]
Therefore we construct, by induction on~$t$, a one-to-one map $\sigma_t\colon P \toinj K_t$ for each $t\in T$, such that setting 
\begin{equation}\label{E:plldefS}
S_{t}=\setm{\sigma_{t\res m}(p)}{m<n\text{ and }p\le x_m},\text{ for all $t=(n,\vec x,\vec r)\in T$,}
\end{equation}
the following containments hold:
\begin{align}
\rng\sigma_t&\subseteq K_t-F(S_t),&&\text{for each $t\in T$,}\label{E:pllt1}\\
F\big(\sigma_t(P\dnw p)\cup S_t\big)\cap \sigma_t(P\dnw q)&\subseteq \sigma_t(P\dnw p),
&& \text{for all $p\le q$ in $P$ and all $t\in T$.}\label{E:pllt2}
\end{align}
For $(t,p)\in A$, set $\sigma(t,p)=\sigma_{t}(p)$. This defines a map $\sigma\colon A\to\kappa$. Let $a=(s,p)$ and $b=(t,q)$ in $A$ such that $\sigma(a)=\sigma(b)$. It follows from~\eqref{E:pllt1} that $\sigma(a)\in K_s$ and $\sigma(b)\in K_t$. As $(K_u)_{u\in T}$ is a partition of $\kappa$, we obtain $s=t$. Moreover $\sigma_t(p)=\sigma(a)=\sigma(b)=\sigma_t(q)$, so, as $\sigma_t$ is one-to-one, $p=q$, and so $a=b$. Therefore $\sigma$ is one-to-one.

Let $a=(t,p)\in A$, with $t=(n,\vec x,\vec r)\in T$. It follows from the definition of $A$ that:
\[
A\dnw a=\setm{(t\res m,q)\in A}{m< n\text{ and }q\in P\dnw x_{m}}\cup\setm{(t,q)\in A}{q\in P\dnw p}.
\]
Thus, from~\eqref{E:plldefS}, we see that
\begin{equation}\label{E:sigmaAdnwa}
\sigma(A\dnw a)=S_{t}\cup \sigma_{t}(P\dnw p),\quad\text{for each $a=(t,p)\in A$}.
\end{equation}
As $(K_t)_{t\in T}$ is a partition, it follows from~\eqref{E:plldefS} and~\eqref{E:pllt1} that $K_t\cap S_t=\emptyset$, thus, by~\eqref{E:sigmaAdnwa},

\begin{equation}\label{E:sigmaAdnwa2}
\sigma(A\dnw a)\cap K_{t}=\sigma_{t}(P\dnw p),\quad\text{for each $a=(t,p)\in A$}.
\end{equation}

Let $a\prec b$ in $A$. There are two cases to consider (cf. Remark~\ref{R:coverinkprod}). First assume that $a=(t,p)$ and $b=(t,q)$ with $p\prec q$ and $t=(n,\vec x,\vec r)\in T$. Let $c\leq b$ with $\sigma(c)\in F(\sigma(A\dnw a))$. We can write $c=(t\res m,p')$, with $m\le n$. Suppose first that $m<n$. As $c\leq b$, $p'\leq x_m$, thus $c<a$. So $\sigma(c)\in\sigma(A\dnw a)$. Now suppose that $m=n$. It follows from~\eqref{E:pllt1} that
\[
\sigma(c)=\sigma_t(p') \in F(\sigma(A\dnw a)) \cap \sigma(A\dnw b)\cap (K_{t}-F(S_{t})).
\]
So, from~\eqref{E:sigmaAdnwa} and~\eqref{E:sigmaAdnwa2} we obtain
\[
\sigma(c)\in F(\sigma_{t}(P\dnw p)\cup S_{t}) \cap \sigma_{t}(P\dnw q).
\]
Thus~\eqref{E:pllt2} implies that $\sigma(c)\in \sigma_{t}(P\dnw p)$, from~\eqref{E:sigmaAdnwa} we obtain $\sigma(c)\in \sigma(A\dnw a)$. Therefore the containment $F(\sigma(A\dnw a))\cap \sigma(A\dnw b)\subseteq \sigma(A\dnw a)$ holds.

Now assume that $t=(n+1,\vec x,\vec r)\in T$, $b=(t,0)$ and $a=(t\res n,x_n)$. Let $c\le b$ such that $\sigma(c)\in F(\sigma(A\dnw a))$. As $c\le b$ there are $m\le n+1$ and $p\in P$ such that $c=(t\res m,p)$. If $m\leq n$, then $p\le x_m$, thus $c\le a$, so $\sigma(c)\in \sigma(A\dnw a)$. If $m=n+1$, then $c=b$. {}From~\eqref{E:plldefS} and~\eqref{E:sigmaAdnwa} we obtain $\sigma(A\dnw a)=S_{t\res n}\cup\sigma_{t\res n}(P\dnw x_n)=S_t$. Therefore $\sigma_{t}(0)=\sigma(c)\in F(S_t)$, in contradiction with~\eqref{E:pllt1}. So the containment $F(\sigma(A\dnw a))\cap \sigma(A\dnw b)\subseteq \sigma(A\dnw a)$ holds.
\end{proof}

\begin{corollary}\label{C:norcov}
For an integer $m>1$, put
 \[
 \rB_m({\le}2)=\setm{X\in\Pow(m)}{\text{either }\card X\le 2\text{ or }X=m}.
 \]
Let $P$ be a \jzs\ embeddable, as a poset, into $\rB_m({\le}2)$. Let $X\subseteq P$, let $\vec R=(R_x)_{x\in X}$ be a family of finite sets, and let $\alpha\le\omega$. There exists an $\aleph_0$-lifter $(U,\bU)$ of $A=\kposet{P}{X}{\vec R}{\alpha}$ such that $U$ has cardinality~$\aleph_2$.
\end{corollary}

\begin{proof}
It follows from \cite{HajMat}, see also \cite[Theorem~46.2]{EHMR}, that $(\aleph_2,2,\aleph_0)\to m$, so \cite[Proposition~5.2]{GiWe2} implies $(\aleph_2,<\aleph_0)\leadsto\rB_m({\le}2)$, and so, from \cite[Lemma~3.2]{GiWe2} we obtain $(\aleph_2,<\aleph_0)\leadsto P$. It follows from Lemma~\ref{L:prodhascomb} that $(\aleph_2,<\aleph_0)\leadsto A$. The conclusion follows from Remark~\ref{R:coverinkprod} together with \cite[Lemma~3-5.5]{GiWe1}.
\end{proof}

\section{The Condensate Lifting Lemma for gamps}

In this section we apply the Armature Lemma from \cite{GiWe1} together with Lemma~\ref{L:LSGamp} to prove a special case of the Condensate Lifting Lemma for gamps.

In order to use the \emph{condensate} constructions, we need categories and functors that satisfy the following conditions.

\begin{definition}
Let $\cA$ and~$\cS$ be categories, let $\Phi\colon\cA\to\cS$ be a functor. We introduce the following statements:
\begin{itemize}
\item[(CLOS)] $\cA$ has all small directed colimits.
\item[(PROD)] Any two objects of~$\cA$ have a product in~$\cA$.
\item[(CONT)] The functor~$\Phi$ preserves all small directed colimits.
\end{itemize}
\end{definition}

\begin{remark}\label{R:DefCondensate}
Given a norm-covering~$X$ of a poset $P$ and a category $\cA$ that satisfies both (CLOS) and (PROD), we can construct an object $\xF(X)\otimes\vec A$ which is a directed colimit of finite products of objects in~$\vec A$, together with morphisms
 \[
 \pi_{\bx}^X\otimes \vec A\colon\xF(X)\otimes\vec A\to A_{\partial\bx}
 \]
for each $\bx\in\Ids X$.

Moreover if $\cA$ is a class of algebras closed under finite products and directed colimits, then $\pi_{\bx}^X\otimes \vec A$ is surjective, and $\card(\xF(X)\otimes\vec A)\le \card X+\sum_{p\in P} A_p$. For more details about this construction, we refer the reader to \cite[Chapter~2]{GiWe1}.
\end{remark}

In the following theorem, we refer the reader to Definition~\ref{D:partiallifting} for the definition of a partial lifting.

\begin{theorem}\label{T:CLLGamp}
Let~$\cV$ and~$\cW$ be varieties of algebras such that~$\cW$ has finite similarity type, let $(X,\bX)$ be an~$\aleph_0$-lifter of a poset $P$, let~$\vec A=\famm{A_p,f_{p,q}}{p\le q\text{ in }P}$ be a diagram in~$\cV$ such that $\Conc A_p$ is finite for each $p\in P^=$, let~$B\in\cW$ such that $\Conc B\cong\Conc\bigl(\xF(X)\otimes\vec A\bigr)$. Then there exists a partial lifting $\vec \bB=\famm{\bB_p,\bgg_{p,q}}{p\le q\text{ in }P}$ of $\Conc\circ\vec A$ in~$\cW$ such that $\bB_p$ is finite for each $p\in P^=$ and $\bB_p$ is a quotient of $\GA(B)$ for each $p\in \Max P$.
\end{theorem}

\begin{proof}
Denote by $\cS$ the category of all \jzs s with \jzh s. The category~$\cV$ satisfies (CLOS) and (PROD), so $\xF(X)\otimes\vec A$ is well-defined (cf. \cite[Section~3-1]{GiWe1}). The functor $\Conc$ satisfies (CONT). Let $\chi\colon \Conc B\to \Conc\bigl(\xF(X)\otimes\vec A\bigr)$ be an isomorphism and put $\rho_{\bx}=\Conc (\pi_{\bx}^X\otimes\vec A)\circ\chi$ for each $\bx\in\bX$. As $\pi_{\bx}^X\otimes\vec A$ is surjective and $\chi$ is an isomorphism, it follows that $\rho_{\bx}$ is ideal-induced. Notice that $\rng\rho_{\bx}=\Conc A_{\partial\bx}$ is finite.

Lemma~\ref{L:LSGamp} implies that there exists a diagram~$\vec \bB=\famm{\bB_{\bx},\bgg_{\bx,\by}}{\bx\le \by\text{ in }\bX^=}$ of finite subgamps of $\GA(B)$ in~$\cW$ such that the following statements hold:
\begin{enumerate}
\item $\rho_{\bx}\res\widetilde B_{\bx}$ is ideal-induced for each $\bx\in \bX^=$.
\item $\bB_{\bx}$ is strong, distance-generated through $\rho_{\bx}\res\widetilde B_{\bx}$, and congruence-tractable through $\rho_{\bx}\res\widetilde B_{\bx}$, for each $\bx\in \bX^=$.
\item $\bgg_{\bx,\by}$ is the canonical embedding, for all $\bx\le \by$ in~$\bX^=$. 
\item $\bgg_{\bx,\by}$ is strong and congruence-cuttable through $\rho_{\by}\res\widetilde B_{\by}$, for all $\bx<\by$ in~$\bX^=$.
\end{enumerate}
We extend this diagram to an~$\bX$-indexed diagram, with $\bB_{\by}=\GA(B)$ and defining $\bgg_{\bx,\by}$ as the canonical embedding for each $\by\in\bX-\bX^=$ and each $\bx\le\by$. Thus $\CG\circ\bB$ is an $\bX$-indexed diagram in the comma category $\cS\dnw \Conc B$, moreover $\CG(\bB_{\bx})=\widetilde B_{\bx}$ is finite for each $\bx\in\bX^=$. Therefore, it follows from the Armature Lemma \cite{GiWe1} that there exists an isotone section $\sigma\colon P\into\bX$ such that the family $(\rho_{\sigma(p)}\res\widetilde B_{\sigma(p)})_{p\in P}$ is a natural transformation from $\famm{\CG(\bB_{\sigma(p)}),\CG(\bgg_{\sigma(p),\sigma(q)})}{p\le q\text{ in }P}$ to $\Conc\circ\vec A$.

Put $\vec\bB'=\famm{\bB_{\sigma(p)},\bgg_{\sigma(p),\sigma(q)}}{p\le q\text{ in }P}$ and $I_p=\ker_0\rho_{\sigma(p)}$, for each $p\in P$. This defines an ideal~$\vec I=(I_p)_{p\in P}$ of~$\vec \bB'$. Moreover, as $\rho_{\sigma(p)}\res\widetilde B_{\sigma(p)}$ is ideal-induced for each $p\in P$, these morphisms induce a natural equivalence $(\CG\circ\vec\bB')/\vec I\cong\Conc\circ\vec\bA$ (cf. Lemma~\ref{L:IIsem}).

Denote by $\bhh_{p,q}\colon \bB_{\sigma(p)}/I_p\to\bB_{\sigma(q)}/I_q$ the morphism induced by $\bgg_{\sigma(p),\sigma(q)}$. It follows from Proposition~\ref{P:quot} that $\bB_{\sigma(p)}/I_p$ is strong for each $p\in P$, and it follows from Proposition~\ref{P:quotarrow} that $\bhh_{p,q}$ is strong for all $p<q$ in~$P$. Lemma~\ref{L:gampthrough1} implies that $\bB_{\sigma(p)}/I_p$ is distance-generated and congruence-tractable for each $p\in P$. {}From Lemma~\ref{L:gampthrough2} we obtain that $\bhh_{p,q}$ is congruence-cuttable for all $p<q$ in~$P$.

Let $p\in P$, let $\chi\colon \widetilde B_{\sigma(p)}\to\Conc B$ be the inclusion map. As $\bB_{\sigma(p)}$ is a subgamp of $\GA(B)$, $(B,\chi)$ is a realization of $\bB_{\sigma(p)}$, thus it induces a realization $(B/\bigvee I_p,\chi')$ where $\chi'\colon \widetilde B_p/I_p\to B/\bigvee I$ satisfies $\chi'(d/I_p)=d/\bigvee I_p$ (cf. Proposition~\ref{P:quot}). As $\rho_{\sigma(p)}\res\widetilde B_{\sigma(p)}$ is ideal-induced it is easy to check that $\chi'$ is surjective, hence it defines an isomorphic realization.

Let $p$ a maximal element of~$P$. {}From $\bB_{\sigma(p)}=\GA(B)$ it follows that $\bB_{\sigma(p)}/I_p$ is a quotient of $\GA(B)$.

Therefore $\vec\bB'/\vec I$ is a partial lifting of $\Conc\circ\vec A$ in~$\cW$. Moreover, if $p\in P^=$ then $\bB_{\sigma(p)}$ is finite, thus $\bB'_p/I_p=\bB_{\sigma(p)}/I_p$ is finite.
\end{proof}

\begin{remark}\label{R:CLLGamp}
Use the notation of Theorem~\ref{T:CLLGamp}. A small modification of the proof above shows that if~$\cW$ is a variety of lattices, then we can construct a lattice partial lifting~$\vec\bB$ of $\Conc\circ\vec A$ in~$\cW$. Moreover, for any integer $n\ge 2$, if~$B$ is congruence $n$-permutable, then all~$\bB_p$ can be chosen congruence $n$-permutable.
\end{remark}

\begin{corollary}
Let~$\cV$ be a locally finite variety of algebras, let~$\cW$ be a variety of algebras with finite similarity type. The following statements are equivalent:
\begin{enumerate}
\item $\crit{\cV}{\cW}>\aleph_0$.
\item Let $T$ be a countable lower finite tree and let~$\vec A$ be a $T$-indexed diagram of finite algebras in~$\cV$. Then $\Conc\circ\vec A$ has a partial lifting in~$\cW$.
\item Let~$\vec A$ be a $\omega$-indexed diagram of finite algebras in~$\cV$. Then $\Conc\circ\vec A$ has a partial lifting in~$\cW$.
\end{enumerate}
\end{corollary}

\begin{proof}
Assume that $(1)$ holds. Let $T$ be a countable lower finite tree and let~$\vec A$ be a $T$-indexed diagram of finite algebras in~$\cV$. It follows from \cite[Corollary 4.7]{G1} that there exists an~$\aleph_0$-lifter $(X,\bX)$ of $T$ such that $\card X=\aleph_0$. Hence the following inequalities hold:
\[
\card(\xF(X)\otimes\vec A)\le \card X+\sum_{p\in T} A_{p} \le\aleph_0\sum_{p\in T} \aleph_0=\aleph_0.
\]
Thus, as $\crit{\cV}{\cW}>\aleph_0$, there exists~$B\in\cW$ such that $\Conc B\cong\Conc\bigl(\xF(X)\otimes\vec A\bigr)$. It follows from Theorem~\ref{T:CLLGamp} that there exists a partial lifting of $\Conc\circ\vec A$ in~$\cW$.

The implication $(2)\Longrightarrow(3)$ is immediate.

Assume that $(3)$ holds and let~$A\in\cV$ such that $\card\Conc A\le\aleph_0$. By replacing~$A$ with one of its subalgebras we can assume that~$\card A\le\aleph_0$ (see \cite[Lemma~3.6]{G1}). As~$\cV$ is locally finite, there exists an increasing sequence $(A_k)_{k<\omega}$ of finite subalgebras of~$A$ with union~$A$. Denote by $f_{i,j}\colon A_i\to A_j$ the inclusion map, for all $i\le j<\omega$. Put $\vec A=\famm{A_i,f_{i,j}}{i\le j<\omega}$. Let~$\vec\bB$ be a partial lifting of $\Conc\circ\vec A$ in~$\cW$, let~$\bB$ be the directed colimit of~$\vec\bB$ in $\GAMP(\cV)$. As $\CG$ and $\Conc$ both preserve directed colimits, the following gamps are isomorphic:
\[
\CG(\bB)\cong\CG(\varinjlim\vec\bB)\cong\varinjlim(\CG\circ\vec\bB)\cong\varinjlim(\Conc\circ\vec A)\cong\Conc\varinjlim\vec A\cong\Conc A.
\]
Moreover, it follows from Lemma~\ref{L:LimPartLiftIsAlgebra} that~$\bB$ is an algebra in~$\cW$, that is, there exists~$B\in\cW$ such that $\bB\cong\GA(B)$, so $\Conc B =\CG(\GA(B))\cong\CG(\bB)\cong\Conc A$. Therefore $\crit{\cV}{\cW}>\aleph_0$.
\end{proof}

A variety of algebras is \emph{congruence-proper} if each of its member with a finite congruence lattice is finite (cf. \cite[Definition 4-8.1]{GiWe1}).

The following theorem is similar to Theorem~\ref{T:CLLGamp} if $\cW$ is a congruence-proper variety with a finite similarity type, then we no longer need partial liftings in the statement of the theorem. There is a similar theorem in \cite[Theorem~4-9.2]{GiWe1}. This new version applies only to varieties (not quasivarieties), but the assumption that~$\cW$ is locally finite is no longer needed.

\begin{theorem}\label{T:diagcrit}
Let~$\cV$ be a variety of algebras. Let~$\cW$ be a congruence-proper variety of algebras with a finite similarity type. Let $(X,\bX)$ be an~$\aleph_0$-lifter of a poset $P$. Let~$\vec A=\famm{A_p,f_{p,q}}{p\le q\text{ in }P}$ be a diagram in~$\cV$ such that $\Conc A_p$ is finite for each $p\in P^=$. Let~$B\in\cW$ such that $\Conc B\cong\Conc\bigl(\xF(X)\otimes\vec A\bigr)$. Then there exists a lifting $\vec B$ of $\Conc\circ\vec A$ in~$\cW$ such that $B_p$ is a quotient of~$B$ for each $p\in P$.
\end{theorem}

\begin{proof}
We notice, as in the beginning of the proof of Theorem~\ref{T:CLLGamp}, that $\xF(X)\otimes\vec A$ is well-defined. We also uses the same notations. Let $\chi\colon \Conc B\to \Conc\bigl(\xF(X)\otimes\vec A\bigr)$ be an isomorphism. Put $\rho_{\bx}=\Conc (\pi_{\bx}^X\otimes\vec A)\circ\chi$ for each $\bx\in\bX$. As $\pi_{\bx}^X\otimes\vec A$ is surjective and $\chi$ is an isomorphism, it follows that $\rho_{\bx}$ is ideal-induced.

Notice that $\rng\rho_{\bx}=\Conc A_{\partial\bx}$ is finite. Hence $\Conc B/\ker_0\rho_{\bx}\cong \Conc A_{\partial\bx}$ is finite. As $\cW$ is congruence-proper, it follows that $B/\ker_0\rho_{\bx}$ is finite.

Hence, the property $\bA/\ker_0\rho_{\bx}=\GA(B)/\ker_0\rho_{\bx}$, in $\bA$ a subgamp of $\GA(B)$ is a locally finite property for~$B$. It follows from Lemma~\ref{L:LSGamp} and Remark~\ref{R:LSGamp} that there exists a diagram~$\vec \bB=\famm{\bB_{\bx},\bgg_{\bx,\by}}{\bx\le \by\text{ in }\bX^=}$ of finite subgamps of $\GA(B)$ in~$\cW$, as in the proof of Theorem~\ref{T:CLLGamp}, but that satisfies the additional condition:
\begin{enumerate}
\item[$(5)$] $\bB_{\bx}/\ker_0\rho_{\bx}=\GA(B)/\ker_0\rho_{\bx}$, for all $\bx\in\bX^=$.
\end{enumerate}
We continue with the same argument as the one in the proof of Theorem~\ref{T:CLLGamp}. We obtain $\bB'/\vec I$ a partial lifting of $\Conc\circ\vec A$. The following isomorphisms hold
\begin{align*}
\bB'_p/I_p&=\bB_{\sigma(p)}/\ker_0\rho_{\sigma(p)} &&\text{see proof of Theorem~\ref{T:CLLGamp}.}\\
&=\GA(B)/\ker_0\rho_{\sigma(p)}&&\text{by $(5)$}.\\
&\cong\GA(B/\ker_0\rho_{\sigma(p)}) &&\text{by Remark~\ref{R:functor}(6).}
\end{align*}
Hence $\bB'_p/I_p$ is an algebra of~$\cV$ (cf. Remark~\ref{R:functor}(3)), for all $p\in P^=$, it is also true for $p\in\Max P$. It follows from Remark~\ref{R:functor}(5) that $\vec\bB'/\vec I\cong\GA\circ\vec B$, for a diagram of algebras $\vec B$. Hence:
\begin{align*}
\Conc\circ\vec B&=\CG\circ\GA\circ\vec B &&\text{by Remark~\ref{R:functor}(1).}\\
&\cong\CG\circ\vec\bB'/\vec I\\
&\cong\Conc\circ\vec A &&\text{as $\vec\bB'/\vec I$ is a partial lifting of $\Conc\circ\vec A$.}
\end{align*}
Hence $\vec B$ is a lifting of $\Conc\circ\vec A$.
\end{proof}

The following corollary is an immediate application of Theorem~\ref{T:diagcrit}.

\begin{corollary}
Let~$\cV$ be a variety of algebras. Let~$\cW$ be a congruence-proper variety of algebras with a finite similarity type. Let $(X,\bX)$ be an~$\aleph_0$-lifter of a poset $P$. Let~$\vec A=\famm{A_p,f_{p,q}}{p\le q\text{ in }P}$ be a diagram of finite algebras in~$\cV$. Assume that $\Conc\circ\vec A$ has no lifting in $\cW$, then $\crit{\cV}{\cW}\le\aleph_0+\card X$.
\end{corollary}

\section{An unliftable diagram}\label{S:unlift}

Each countable locally finite lattice has a congruence-permutable, congruence-preserving extension (cf.~\cite{GLWe}). This is not true for all locally finite lattices. Given a non-distributive variety~$\cV$ of lattices, there is no congruence-permutable algebra~$A$ such that $\Conc A\cong\Conc F_{\cV}(\aleph_2)$ (cf. \cite{RTW}). Hence $\Conc F_{\cV}(\aleph_2)$ has no congruence-permutable, congruence-preserving extension. The latter result can be improved, by stating that the free lattice $F_\cV(\aleph_1)$ has no congruence-permutable, congruence-preserving extension (cf. \cite{GiWe1}).

Let~$\cV$ be a nondistributive variety of lattices. There is no congruence $n$-permutable lattices~$L$ such that $\Conc L\cong F_\cV(\aleph_2)$, for each $n\ge 2$ (cf. \cite{Ploscica08}). In particular $F_\cV(\aleph_2)$ has no congruence $n$-permutable, congruence-preserving extension, for each $n\ge 2$. The aim of this section is to improve the cardinality bound to~$\aleph_1$. We use gamps to find a lattice of cardinal~$\aleph_1$ with no congruence $n$-permutable, congruence-preserving extension, for each $n\ge 2$. This partially solve \cite[Problem~7]{GiWe1}. The proof is based on a square-indexed diagram of lattices with no congruence $n$-permutable, congruence-preserving extension.

Let $\vec A=\famm{A_p,f_{p,q}}{p\le q\text{ in }P}$ be a diagram of algebras, let $\vec B=\famm{B_p,g_{p,q}}{p\le q\text{ in }P}$ be a congruence-preserving extension of~$\vec A$. Then $\bB_p=(A_p,B_p,\Theta_{B_p},\Conc B_p)$ is a gamp and $\bgg_{p,q}=(g_{p,q},\Conc g_{p,q})\colon \bB_p\to\bB_q$ is a morphism of gamps, for all $p\le q$ in $P$. This defines a diagram $\vec\bB$ of gamps. Moreover, identifying $\Conc A_p$ and $\Conc B_p$ for all $p\in P$, we have $\PGGL\circ\vec \bB=\PGA\circ\vec A$.

Fix $n\ge 2$ an integer. Given a diagram~$\vec A$ of algebras, with a congruence $n$-permutable, congruence-preserving extension, there exists a diagram $\vec\bB$ of congruence $n$-permutable gamps such that $\PGGL\circ\vec \bB=\PGA\circ\vec A$. The converse might not hold in general.

However the square-indexed diagram~$\vec A$ of finite lattices with no congruence $n$-permutable, congruence-preserving extension, mentioned above, satisfies a stronger property. There is no operational diagram $\vec \bB$ of lattice congruence $n$-permutable gamps of lattices (cf. Definition~\ref{D:perminlatt}) such that $\PGA\circ\vec A\cong\PGGL\circ\vec \bB$ (cf. Lemma~\ref{L:unliftablediag}).

We conclude, in Theorem~\ref{T:NoCPCnP}, that there is a condensate of~$\vec A$, of cardinal~$\aleph_1$, with no congruence $n$-permutable, congruence-preserving extension.

For the purpose of this section, we need a stronger version of congruence $n$-permutable gamp, specific to gamp of lattices.

\begin{definition}\label{D:perminlatt}
A gamp $\bA$ of lattices is \emph{lattice congruence $n$-permutable} if for all $x_0,\dots,x_n$ in $A^*$ there exist $y_0,\dots,y_n$ in~$A$ such that $y_i\wedge y_j=y_j\wedge y_i=y_i$ in~$A$ for all $i\le j\le n$, $y_0=x_0\wedge x_n=x_n\wedge x_0$ in~$A$, $y_n=x_0\vee x_n=x_n\vee x_0$ in~$A$, and:
\[
\delta(y_k,y_{k+1})\le\bigvee\famm{\delta(x_{i},x_{i+1})}{i<n\text{ even}},\quad\text{for all $k<n$ odd},
\]
\[
\delta(y_k,y_{k+1})\le\bigvee\famm{\delta(x_{i},x_{i+1})}{i<n\text{ odd}},\quad\text{for all $k<n$ even}.
\]

A morphism $\bff\colon \bA\to\bB$ of gamps is \emph{operational} if $\ell(\vec x)$ is defined in~$B$ for all $\ell\in\sL$ and all $\ari(\ell)$-tuple~$\vec X$ in $f(A)\cup B^*$.

Let $P$ be a poset, a diagram $\vec\bA=\famm{\bA_p,\bff_{p,q}}{p\le q\text{ in }P}$ of gamps is \emph{operational} if $\bff_{p,q}$ is operational for all $p<q$ in $P$.
\end{definition}

\begin{remark}
In the context of Definition~\ref{D:perminlatt}, the elements $y_1,\dots,y_n$ doe not form a chain in general, as we might have $y_i\not\in A^*$ for some $i$.
\end{remark}

In this section we simply say that the gamp is congruence $n$-permutable instead of lattice congruence $n$-permutable. We also consider only pregamps and gamps of lattices. The following lemma is immediate, the properties given in Definition~\ref{D:perminlatt} go to quotients.

\begin{lemma}
Let $\bA$ be a congruence $n$-permutable gamp, let $I$ be an ideal of $\bA$, then $\bA/I$ is a congruence $n$-permutable gamp.

Let $P$ be a poset, let $\vec\bA$ be an operational $P$-indexed diagram of algebras. Let $\vec I$ be an ideal of $\vec\bA$. The quotient $\bA/\vec I$ is operational.
\end{lemma}

The following lemma is similar to the Buttress Lemma (cf. \cite{GiWe1}) and to Lemma~\ref{L:LSGamp}, this version is specific to the functor $\PGGL$.

\begin{lemma}\label{L:LSGampOverPregamp}
Let $n\ge 2$. Let~$\cV$ be a variety of lattices. Let $P$ be a lower finite poset. Let $(\bA_p)_{p\in P}$ be a family of finite \pregamps. Let $\bB$ be a gamp in~$\cV$ such that~$B$ is a congruence $n$-permutable lattice. Let $(\bpi_p)_{p\in P}$ be a family of ideal-induced morphisms of \pregamps\ where $\bpi_p\colon \PGGL \bB\to \bA_p$ for all $p\in P$. Then there exists a diagram $\vec \bB=\famm{\bB_p,\bgg_{p,q}}{p\le q\text{ in }P}$ of finite subgamps of $\bB$ (where $\bgg_{p,q}$ is the canonical embedding for all $p\le q$ in $P$), such that the following assertions hold.
\begin{enumerate}
\item $\bpi_p\res\PGGL(\bB_p)$ is an ideal-induced morphism of \pregamps\ for all $p\in P$.
\item $\bB_p$ is strong and congruence $n$-permutable for all $p\in P$.
\item $\bgg_{p,q}$ is operational for all $p<q$ in $P$.
\end{enumerate}
\end{lemma}

\begin{proof}
Let $r\in P$, assume that we have already constructed $\famm{\bB_p,\bgg_{p,q}}{p\le q<r}$ that satisfies the required conditions up to~$r$.

As $\bpi_r$ is ideal-induced, $\pi_r(B^*)=A_r$. Moreover $A_r$ is finite, hence there exists $X$ a finite partial subalgebra of $B^*$ such that $\pi_r(X)=A_r$, put $B_r^*=X\cup\bigcup_{p<r} B_p^*$ with its structure of full partial subalgebra of $B^*$. It follows that $\pi_r(B_r^*)\subseteq\pi_r(B^*)=A_r=\pi_r(X)\subseteq\pi_r(B_r^*)$, hence $\pi_r(B_r^*)=A_r$. Moreover $B_r^*$ is finite.

Let $x_0,x_1,\dots,x_n$ in $B_r^*$. As~$B$ is a congruence $n$-permutable lattice, there exist $y_0<\dots<y_n$ in~$B$ such that $y_0=x_0\wedge x_n$, $y_n=x_0\vee x_n$, and
\[
\delta(y_k,y_{k+1})\le\bigvee\famm{\delta(x_{i},x_{i+1})}{i<n\text{ even}},\quad\text{for all $k<n$ odd},
\]
\[
\delta(y_k,y_{k+1})\le\bigvee\famm{\delta(x_{i},x_{i+1})}{i<n\text{ odd}},\quad\text{for all $k<n$ even}.
\]
Put $X_{x_0,x_1,\dots,x_n}=\set{y_0,\dots,y_n}$. We consider the following finite set with its structure of full partial subalgebra of~$B$ 
\begin{align*}
B_r=&\bigcup\famm{X_{x_0,x_1,\dots,x_n}}{x_0,x_1,\dots,x_n\in B_p^*}\\
&\cup\Setm{x\vee y}{x,y\in B_r^*\cup\bigcup_{p<r}B_p}\\
&\cup\Setm{x\wedge y}{x,y\in B_r^*\cup\bigcup_{p<r}B_p}.
\end{align*}

Put $Y=\setm{\delta(x,y)}{x,y\in B_r}$. As $\widetilde\pi_r$ is ideal-induced, it follows from Proposition~\ref{P:IdeIdu} that there is a finite \jz-subsemilattice $\widetilde B_r$ of $\widetilde B$ such that $Y\subseteq \widetilde B_r$, and $\widetilde\pi_r\res\widetilde B_r$ is ideal-induced. Put $\bB_r=(B_r^*,B_r,\delta,\widetilde B_r)$, denote by $\bgg_{p,r}\colon\bB_p\to\bB_r$ the inclusion morphism for all $p\le r$. The conditions $(1)$-$(3)$ are satisfied. The conclusion follows by induction.
\end{proof}

We apply Lemma~\ref{L:LSGampOverPregamp} and the Armature Lemma (cf. \cite{GiWe1}) to obtain a new, tailor-made version of CLL. Given a diagram~$\vec A$ of finite lattices and a congruence-preserving, congruence $n$-permutable extension of a condensate of~$\vec A$, we obtain a ``congruence-preserving, congruence $n$-permutable extension'' of~$\vec A$.

\begin{lemma}\label{L:CLLCPCnP}
Let~$\cV$ be a variety of lattices. Let $(X,\bX)$ be an~$\aleph_0$-lifter of a poset~$P$, let $\vec A=(A_p,f_{p,q})$ be a diagram of~$\cV$ such that $\Conc A_p$ is finite for all $p\in P^=$. Let~$B$ be a congruence $n$-permutable lattice. If~$B$ is a congruence-preserving extension of $\xF(X)\otimes\vec A$, then there exists an operational diagram $\vec \bB$ of congruence $n$-permutable gamps such that $\PGGL\circ\vec \bB\cong\PGA\circ\vec A$.
\end{lemma}

\begin{proof}
As in the proof of Theorem~\ref{T:CLLGamp}, $\xF(X)\otimes\vec A$ is well-defined. Denote by $\cS$ the category of \pregamps. The functor $\PGA\colon\cV\to\cS$ satisfies (CONT), see Remark~\ref{R:CPGandPGApreslim}. Put $B^*=\xF(X)\otimes\vec A$, as~$B$ is a congruence-preserving extension of $B^*$, we can identify $\Conc B^*$ with $\Conc B$. Put $\bB=(B^*,B,\Theta_B,\Conc B)$, hence $\PGGL(\bB)=\PGA(B^*)$. Put $\bpi_{\bx}=\PGA(\pi^X_{\bx}\otimes\vec A)\colon \PGGL(\bB)\to\PGA(\bA)$ for all $\bx\in\bX$. It follows from Lemma~\ref{L:LSGampOverPregamp}, that there exists a diagram $\famm{\bB_{\bx},\bgg_{\bx,\by}}{\bx\le \by\text{ in }\bX^=}$ of finite subgamps of $\bB$ such that the following assertions hold.
\begin{enumerate}
\item $\bpi_{\bx}\res\PGGL (\bB_{\bx})$ is an ideal-induced morphism of \pregamps\ for all $\bx\in \bX^=$.
\item $\bB_{\bx}$ is strong and congruence $n$-permutable for all $\bx\in \bX^=$.
\item $\bgg_{\bx,\by}$ is operational for all $\bx< \by$ in $\bX^=$.
\end{enumerate}
We complete the diagram with $\bB_{\by}=\bB$ and $\bgg_{\bx,\by}$ the inclusion morphism for all $\by\in\bX-\bX^=$, and $\bx\le \by$. It follows from the Armature Lemma (cf. \cite[Lemma~3-2.2]{GiWe1}) that there exists $\sigma\colon P\to\bX$ such that $\partial\sigma(p)=p$ for all $p\in P$ and $(\bpi_{\sigma(p)}\res\PGGL(\bB_{\sigma(p)}))_{p\in P}$ is a natural transformation from $\famm{\PGGL(\bB_{\sigma(p)}),\PGGL(\bgg_{\sigma(p),\sigma(q)})}{p\le q\text{ in }P}$ to $\PGA\circ\vec A$.

Put $\brho_p=\bpi_{\sigma(p)}\res\PGGL(\bB_{\sigma(p)})$, put $I_p=\ker_0\brho_p$ for all $p\in P$. Put $\vec I=(I_p)_{p \in P}$, it is an ideal of $\famm{\bB_{\sigma(p)},\bgg_{\sigma(p),\sigma(q)}}{p\le q\text{ in }P}$. Put $\vec\bC=\famm{\bB_{\sigma(p)}/I_p,\bgg_{\sigma(p),\sigma(q)}/\vec I}{p\le q\text{ in }P}$. Notice that $\vec\bC$ is an operational diagram of congruence $n$-permutable gamps.

Denote by $\bchi_p\colon \PGGL(\bB_{\sigma(p)})/I_p\to \PGA(A_p)$ the morphism induced by $\brho_p$. It follows from Lemma~\ref{L:IIPG} that $\bchi_p$ is an isomorphism, for all $p\in P$. Hence, from Proposition~\ref{P:projsmpadiag} and Remark~\ref{R:quotassoc}, we obtain that $\vec \bchi=(\bchi_p)_{p\in P}\colon\PGGL\circ\vec \bC\to\PGA\circ\vec A$ is a natural equivalence.
\end{proof}

In the following lemma we construct a square-indexed diagram of lattices with no congruence $n$-permutable, congruence-preserving extension.

\begin{lemma}\label{L:unliftablediag}
Let $n\ge 2$. Let $K$ be a nontrivial, finite, congruence $(n+1)$-permutable lattice, let $x_1,x_2,x_3$ in $K$ such that $x_1\wedge x_2=0$ and $x_3\vee x_2=x_3\vee x_1=1$. There exists a diagram~$\vec A$ of finite congruence $(n+1)$-permutable lattices in $\Var^{0,1}(K)$ indexed by a square, such that $\PGA\circ\vec A\not\cong\PGGL\circ\vec \bB$ for each operational square $\vec \bB$ of congruence $n$-permutable gamps.
\end{lemma}

\begin{proof}
Put $X_0=\set{0,x_3,1}$, put $X_1=\set{0,x_1\wedge x_3,x_1,x_3,1}$, put $X_2=\set{0,x_2\wedge x_3,x_2,x_3,1}$, put $X_3=K$. Notice that $X_k$, $k<4$, are all congruence $(n+1)$-permutable sublattices of $K$.

\begin{figure}[here,top,bottom]
\setlength{\unitlength}{1mm}
\begin{picture}(50,70)(-5,-5)
\put(20,20){\line(1,1){20}}
\put(20,20){\line(-1,1){20}}
\put(40,40){\line(-1,1){20}}
\put(0,40){\line(1,1){20}}
\put(20,0){\line(0,1){20}}

\put(20,20){\circle*{2}}
\put(40,40){\circle*{2}}
\put(0,40){\circle*{2}}
\put(20,60){\circle*{2}}
\put(20,0){\circle*{2}}
\put(-6,39){$x_3$}
\put(42,39){$x_k$}
\put(22,19){$x_3\wedge x_k$}

\put(19,-5){0}
\put(19,62){1}
\end{picture}
\caption{The lattice~$X_k$, for $k\in\set{1,2}$}\label{F:lat-Xk}
\end{figure}
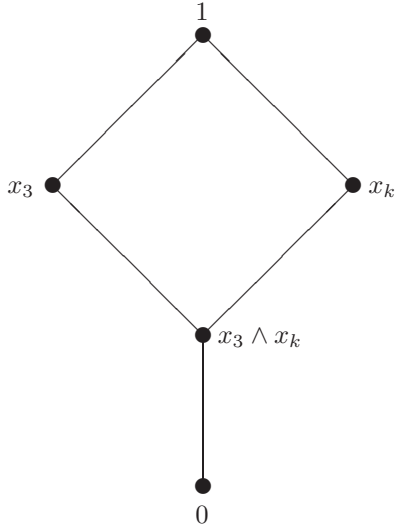

We denote by $h_i\colon X_0\to X_i$ and $h_i'\colon X_i\to X_3$ the inclusion maps, for $i=1,2$. Denote by~$\vec X$ the square on the right hand side of Figure~\ref{F:SQ1}.

\begin{sclaim}
Let $\vec \bB$ be a square of operational gamps \textup(as in Figure~\textup{\ref{F:SQ2}}\textup), let $\vec\bxi\colon\PGA\circ\vec X\to\PGGL\circ\vec\bB$ be a natural equivalence. Let $y\in B_0$ such that $\delta_{\bB_0}(\xi_0(0),y)\le\delta_{\bB_0}(\xi_0(x_3),\xi_0(1))$, $y\wedge\xi_0(1)=y$, and $y\vee\xi_0(0)=y$ in $B_0$, then $y=\xi_0(0)$.
\end{sclaim}

\begin{scproof}
We can assume that $g_1$, $g_2$, $g'_1$, and $g'_2$ are inclusion maps. We can assume that $\PGGL\circ\vec\bB=\PGA\circ\vec X$ and $\vec\bxi$ is the identity. We denote $\delta_k$ instead of $\delta_{\bB_k}$ for all $k\in\set{0,1,2,3}$. Notice that $\delta_k(u,v)$ is a congruence of $X_k$ for all $u,v\in B_k$, moreover if $u,v\in X_k=B_k^*$ then $\delta_k(u,v)=\Theta_{X_k}(u,v)$. Let $y\in B_0$ such that $\delta_{0}(0,y)\subseteq\delta_{0}(x_3,1)$, $y\wedge 1=y$ and $y\vee 0=y$ in~$B_0$.

Let $k\in\set{1,2}$. As $\bgg_k$ is operational, $y\wedge x_k$ is defined in $B_k$. Moreover $0\wedge x_k=0$, hence $\delta_k(0,y\wedge x_k)\subseteq\delta_k(0,y)$. Therefore the following containments hold:
\begin{equation}\label{C:eq1}
\delta_k(y,y\wedge x_k)\subseteq\delta_k(y,0)\vee\delta_k(0,y\wedge x_k)\subseteq\delta_k(0,y)\subseteq\delta_k(x_3,1)=\Theta_{X_k}(x_3,1).
\end{equation}
Moreover, as $y\wedge 1=y$, the following containment holds
\begin{equation}\label{C:eq2}
\delta_k(y,y\wedge x_k)=\delta_k(y\wedge 1,y\wedge x_k)\subseteq\delta_k(1,x_k)=\Theta_{X_k}(1,x_k)
\end{equation}
However $\Theta_{X_k}(1,x_k)\cap \Theta_{X_k}(x_3,1)=\zero_{X_k}$ (see Figure~\ref{F:lat-Xk}), thus it follows from \eqref{C:eq1} and \eqref{C:eq2} that $\delta_k(y,y\wedge x_k)=\zero_{X_k}$, hence the following equality holds
\[
y=y\wedge x_k,\quad\text{for each $k\in\set{1,2}$.}
\]
Therefore $y=(y\wedge x_1)\wedge x_2$ in $B_3$, moreover $x_1\wedge x_2=0$, hence as $\vec\bB$ is operational $y\wedge(x_1\wedge x_2)$ is defined in $B_3$, thus $y\wedge 0=y\wedge(x_1\wedge x_2)=(y\wedge x_1)\wedge x_2=y$.

Moreover as $y\vee 0=y$, it follows that $0=(y\vee 0)\wedge 0=y\wedge 0=y$ (all elements are defined in $B_3$), hence $y=0$.
\end{scproof}

Let $C$ be an $(n+1)$-element chain. Set $T=\setm{t}{t\colon C\tosurj X_0\text{ is order preserving}}$. Put $A_0=C$, put $A_1=X_1^T$, put $A_2=X_2^T$, put $A_3=X_3^T=K^T$. We consider the following morphisms:
\begin{align*}
f_i\colon A_0&\to A_i\\
x&\mapsto (t(x))_{t\in T},\quad\text{for $i=1,2$}.
\end{align*}
We denote by $f_i'\colon A_i\to A_3$ the inclusion maps, for $i=1,2$. We denote by~$\vec A$ the square in the left hand side of Figure~\ref{F:SQ1}.
\begin{figure}[htb]
 \[
  \xymatrix{
   & A_3 & & & X_3 & \\ 
   A_1\ar[ur]^{f_1'} & & A_2\ar[ul]_{f_2'} & X_1\ar[ur]^{h_1'} & & X_2\ar[ul]_{h_2'} \\
   & A_0\ar[ul]^{f_1}\ar[ur]_{f_2} & & & X_0\ar[ul]^{h_1}\ar[ur]_{h_2} &
   } 
 \]\caption{Two squares in $\Var^{0,1}(K)$}\label{F:SQ1}
\end{figure}
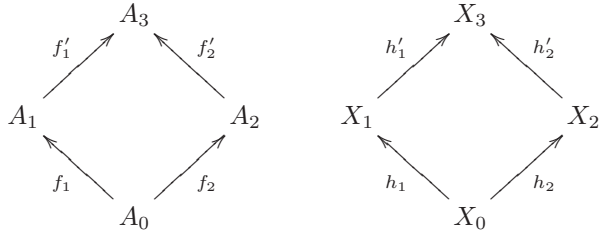

Given $t\in T$, denote $\pi^t_0=t$ and $\pi^t_k\colon A_k\to X_k$, $(a_p)_{p\in T}\mapsto a_t$ the canonical projection, for all $k\in\set{1,2,3}$. It defines a natural transformation $\vec\pi^t$ from~$\vec A$ to~$\vec X$. We denote by $\bpi_k^t=\PGA(\pi_k^t)=(\pi_k^t,\Conc\pi_k^t)$, for all $k\in\set{0,1,2,3}$.

Assume that there exists an operational square $\vec\bB$, as in Figure~\ref{F:SQ2}, of congruence $n$-permutable gamps, and a natural equivalence $\PGA\circ\vec A\to\PGGL\circ\vec\bB$. We can assume that $\PGGL\circ\vec\bB=\PGA\circ\vec A$. Put $\delta_k=\delta_{\bB_k}$, the distance $\delta_{\bA_k}$ is a restriction of $\delta_k$, for all $k\in\set{0,1,2,3}$.
\begin{figure}[htb]
 \[
  \xymatrix{
   & \bB_3 & \\ \bB_1\ar[ur]^{\bgg_1'} & & \bB_2\ar[ul]_{\bgg_2'}\\
   & \bB_0\ar[ul]^{\bgg_1}\ar[ur]_{\bgg_2} &
   } 
 \]\caption{A square in the category~$\GAMP(\cL)$}\label{F:SQ2}
\end{figure}
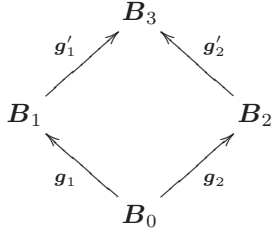

Let $a_0<a_1<\dots<a_n$ be the elements of $C=A_0=B_0^*$. As $\bB_0$ is congruence $n$-permutable, there exist $b_0,\dots,b_n$ in~$B_0$ such that $b_i\wedge b_j=b_j\wedge b_i=b_i$ in $B_0$ for all $i\le j\le n$, $b_0=a_0\wedge a_n=a_0$ in $B_0$, $b_n=a_0\vee a_n=a_n$ in $B_0$, and:
\begin{align}
\delta_0(b_k,b_{k+1})\le\bigvee\famm{\delta_0(a_{i},a_{i+1})}{i<n\text{ even}},\quad\text{for all $k<n$ odd},\label{E:CPext1}\\
\delta_0(b_k,b_{k+1})\le\bigvee\famm{\delta_0(a_{i},a_{i+1})}{i<n\text{ odd}},\quad\text{for all $k<n$ even}.\label{E:CPext2}
\end{align}
In particular the following inequality holds
\begin{equation}\label{E:CPext2b}
\delta_0(b_k,b_{k+1})\le\delta_0(a_0,a_k)\vee\delta_0(a_{k+1},a_n),\quad\text{for all $k<n$.}
\end{equation}
As $b_n=a_n$, an immediate consequence of \eqref{E:CPext2b} is $\delta_0(b_{n-1},a_{n})\le\delta_0(a_0,a_{n-1})$. Assume that $\delta_0(a_{n-2},b_{n-1})\le\delta_0(a_0,a_{n-2})$, it follows that
\[
\delta_0(a_0,a_n)\le\delta_0(a_0,a_{n-2})\vee\delta_0(a_{n-2},b_{n-1})\vee\delta_0(b_{n-1},a_n)\le\delta_0(a_0,a_{n-1})
\]
That is $\Theta_{A_0}(a_0,a_n)\le\Theta_{A_0}(a_0,a_{n-1})$ a contradiction, as $A_0$ is the chain $a_0<\dots<a_n$. It follows that $\delta_0(a_{n-2},b_{n-1})\not\le\delta_0(a_0,a_{n-2})$.

Take $i<n$ minimal such that the following inequality hold 
\begin{equation}\label{E:CPext3}
\delta_0(a_i,b_{i+1})\not\le\delta_0(a_0,a_i).
\end{equation}

If $i=0$, then $a_0=b_0=b_i$, hence $\delta_0(a_0,b_i)\le\delta_0(a_0,a_i)$. If $i>0$, then it follows from the minimality of $i$ that $\delta_0(a_{i-1},b_i)\le\delta_0(a_0,a_{i-1})\le\delta_0(a_0,a_i)$, thus $\delta_0(a_0,b_i)\le\delta_0(a_0,a_{i-1})\vee\delta_0(a_{i-1},b_i)\le\delta_0(a_0,a_i)$. Therefore the following inequality holds
\begin{equation}\label{E:CPextcomp}
\delta_0(a_0,b_i)\le\delta_0(a_0,a_i).
\end{equation}

Assume that $\delta_0(b_i,b_{i+1})\le\delta_0(a_0,a_i)$, this implies with \eqref{E:CPextcomp} the following inequality
\[
\delta_0(a_i,b_{i+1})\le\delta_0(a_i,a_0)\vee \delta_0(a_0,b_i)\vee\delta_0(b_i,b_{i+1})=\delta_0(a_0,a_i),
\]
which contradicts~\eqref{E:CPext3}. Therefore the following statement holds
\begin{equation}\label{E:tt}
\delta_0(b_i,b_{i+1})\not\le\delta_0(a_0,a_i).
\end{equation}

As $\Con A_0$ is a Boolean lattice with atoms $\delta_0(a_k,a_{k+1})$ for $k<n$, it follows from \eqref{E:tt} that there is $j<n$ such that
\begin{equation} \label{E:CPext5}
\delta_0(a_j,a_{j+1})\le \delta_0(b_i,b_{i+1})\text{\ \ and\ \ }\delta_0(a_j,a_{j+1})\not\le \delta_0(a_0,a_i),
\end{equation}
As $\delta_0(a_j,a_{j+1})\not\le \delta_0(a_0,a_i)$, $j\ge i$. It follows from \eqref{E:CPext1}, \eqref{E:CPext2}, and \eqref{E:CPext5} that $i$ and $j$ have distinct parities, therefore $j>i$.

Put 
\begin{align*}
t\colon A_0&\to \set{0,x_3,1}\\
a_k&\mapsto
\begin{cases}
0 &\text{if $k\le i$},\\
x_3 &\text{if $i<k\le j$},\\
1 &\text{if $j<k$,}
\end{cases}\quad\text{for all $k\le n$.}
\end{align*}
As $i<j<n$ the map $t$ is surjective, thus $t\in T$. Put $I_i=\ker_0\bpi_i^t$, for all $i\in\set{0,1,2,3}$. Denote by $\vec\bchi\colon \PGA\circ\vec A/\vec I=\PGGL\circ\vec\bB\to\vec\bX$ the natural transformation induced by~$\vec\bpi$. As $\vec\bpi=\PGA(\vec \pi)$ is ideal-induced, it follows from Lemma~\ref{L:IIPG} that $\vec\bchi$ is a natural equivalence. Put $\vec\bxi=\vec\bchi^{-1}$. Notice that the following inequalities hold
\begin{align*}
\delta_{0}(a_0,b_{i+1})&\le\delta_0(a_0,b_i)\vee\delta_0(b_i,b_{i+1})\\
&\le\delta_0(a_0,a_i)\vee\delta_0(a_{i+1},a_n) && \text{by \eqref{E:CPextcomp} and \eqref{E:CPext2b}.}
\end{align*}
Moreover $\xi_0(0)=a_0/I_0=a_i/I_0$, $\xi_0(x_3)=a_{i+1}/I_0$ and $\xi_0(1)=a_n/I_0$, thus $\delta_{\bB_0/I_0}(\xi_0(0),b_{i+1}/I_0)\le \delta_{\bB_0/I_0}(\xi_0(x_3),\xi_0(1))$. As $b_{i+1}/I_0\wedge\xi_0(1)=(b_{i+1}\wedge a_n)/I_0=b_{i+1}/I_0$ and $b_{i+1}/I_0\vee\xi_0(0)=(b_{i+1}\vee a_0)/I_0=b_{i+1}/I_0$. It follows from the Claim that $b_{i+1}/I_0=\xi_0(0)=a_0/I_0$, that is $\delta_0(a_0,b_{i+1})\in I_0$. Therefore the following inequality holds:
\begin{equation}\label{Eq:consclaim}
\delta_0(a_0,b_{i+1})\le \delta_0(a_0,a_i)\vee\delta_0(a_{i+1},a_j)\vee\delta_0(a_{j+1},a_n)
\end{equation}
Hence we obtain
\begin{align*}
\delta_0(a_j,a_{j+1})&\le \delta_0(b_i,b_{i+1}) &&\text{by \eqref{E:CPext5}.}\\
&\le\delta_0(b_i,a_0)\vee\delta_0(a_0,b_{i+1})\\
&\le \delta_0(a_0,a_i)\vee\delta_0(a_{i+1},a_j)\vee\delta_0(a_{j+1},a_n) &&\text{by \eqref{E:CPextcomp} and \eqref{Eq:consclaim}.}\\
&\le \delta_0(a_0,a_j)\vee\delta_0(a_{j+1},a_n).
\end{align*}
A contradiction, as $A_0$ is the chain $a_0<a_1<\dots<a_n$.
\end{proof}

\begin{figure}[here,bottom,top]
\setlength{\unitlength}{1mm}
\begin{picture}(50,45)(-5,-5)
\put(20,0){\line(1,1){20}}
\put(20,0){\line(-1,1){20}}
\put(20,0){\line(0,1){20}}
\put(19,-4){0}
\put(-5,20){$x_1$}
\put(15,20){$x_2$}
\put(35,20){$x_3$}
\put(19,42){1}

\put(0,20){\line(1,1){20}}
\put(20,20){\line(0,1){20}}
\put(40,20){\line(-1,1){20}}

\put(20,0){\circle*{2}}
\put(0,20){\circle*{2}}
\put(20,20){\circle*{2}}
\put(40,20){\circle*{2}}
\put(20,40){\circle*{2}}
\end{picture}
\quad
\begin{picture}(45,50)(5,-5)
\put(20,0){\line(-1,1){15}}
\put(20,0){\line(1,1){15}}
\put(35,15){\line(0,1){10}}
\put(5,15){\line(0,1){10}}

\put(5,25){\line(1,1){15}}
\put(35,25){\line(-1,1){15}}
\put(20,0){\circle*{2}}
\put(5,15){\circle*{2}}
\put(5,25){\circle*{2}}
\put(35,15){\circle*{2}}
\put(35,25){\circle*{2}}
\put(20,40){\circle*{2}}

\put(20,30){\line(0,1){10}}
\put(5,15){\line(1,1){15}}
\put(35,15){\line(-1,1){15}}
\put(20,30){\circle*{2}}

\put(19,-4){$0$}
\put(19,42){$1$}
\put(19,26){$x_3$}
\put(0,25){$x_1$}
\put(38,25){$x_2$}
\end{picture}\quad
\begin{picture}(50,50)(-5,-5)
\put(12,0){\line(-1,1){12}}
\put(12,0){\line(1,1){24}}
\put(24,12){\line(-1,1){12}}
\put(0,12){\line(1,1){24}}
\put(36,24){\line(-1,1){12}}

\put(12,0){\circle*{2}}
\put(0,12){\circle*{2}}
\put(24,12){\circle*{2}}
\put(12,24){\circle*{2}}
\put(18,18){\circle*{2}}
\put(36,24){\circle*{2}}
\put(24,36){\circle*{2}}

\put(11,-4){$0$}
\put(23,38){$1$}

\put(-5,11){$x_1$}
\put(13,17){$x_2$}
\put(39,23){$x_3$}

\end{picture}\quad
\begin{picture}(50,50)(-5,-5)
\put(20,0){\line(-1,1){20}}
\put(20,0){\line(1,1){15}}
\put(35,15){\line(0,1){10}}
\put(20,0){\line(0,1){10}}
\put(20,10){\line(1,1){15}}
\put(0,20){\line(1,1){20}}
\put(35,25){\line(-1,1){15}}
\put(20,0){\circle*{2}}
\put(0,20){\circle*{2}}
\put(20,10){\circle*{2}}
\put(35,15){\circle*{2}}
\put(35,25){\circle*{2}}
\put(20,40){\circle*{2}}
\put(19,-4){$0$}
\put(19,42){$1$}
\put(-5,19){$x_3$}
\put(38,14){$x_1$}
\put(15,9){$x_2$}
\end{picture}
\caption{The lattices $M_3$, $L_2$, $L_3$, $L_4$.}\label{F:treillis_M3_L2_L4}
\end{figure}
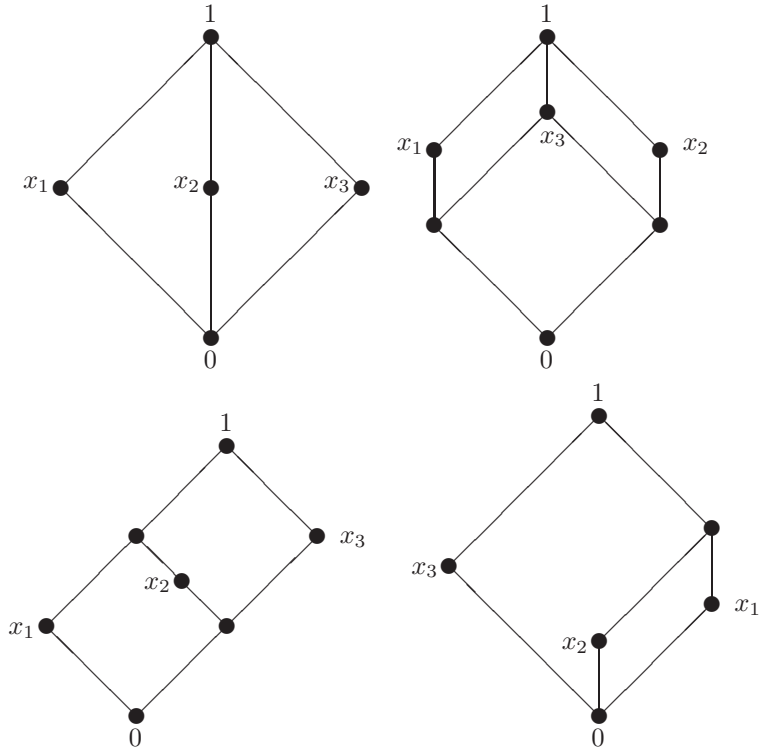 

\begin{theorem}\label{T:NoCPCnP}
Let $n\ge 2$. Let~$\cV$ be a variety of lattices, such that either $M_3\in\cV$, $L_2\in\cV$, $L_3\in\cV$, $L_4\in\cV$, $\dual{L_2}\in\cV$, or $\dual{L_4}\in\cV$. There exists a bounded lattice $L\in\cV$ such that~$L$ is congruence $(n+1)$-permutable, $\card L=\aleph_1$, and~$L$ has no congruence $n$-permutable, congruence-preserving extension in the variety of all lattices.
\end{theorem}

\begin{proof}
Fix $(X,\bX)$ an~$\aleph_0$-lifter of the square such that $\card X=\aleph_1$ (cf. Lemma~\ref{L:squarehaslifter}). Up to changing~$\cV$ to its dual, we can assume that either $M_3\in\cV$, $L_2\in\cV$, $L_3\in\cV$, or $L_4\in\cV$. Let $K$ one of those lattices such that $K\in\cV$, let $x_1,x_2,x_3$ as in Figure~\ref{F:treillis_M3_L2_L4}. The conditions of Lemma~\ref{L:unliftablediag} are satisfied. Denote by~$\vec A$ the diagram constructed in Lemma~\ref{L:unliftablediag}.

Put $L=\xF(X)\otimes\vec A\in\Var^{0,1}(K)\subseteq\cV$ (cf. Remark~\ref{R:DefCondensate}). Notice that~$L$ is a directed colimit of finite products of lattices in $\vec A$ and all lattices in $\vec A$ are congruence $(n+1)$-permutable, thus~$L$ is congruence $(n+1)$-permutable. As $\card X=\aleph_1$ and each lattice in the diagram~$\vec A$ is finite, $\card L=\aleph_1$. Moreover~$L$ cannot have a congruence $n$-permutable, congruence-preserving extension, as the conclusions of Lemma~\ref{L:CLLCPCnP} and Lemma~\ref{L:unliftablediag} contradict each other.
\end{proof}

The following corollary is an immediate consequence of Theorem~\ref{T:NoCPCnP}.

\begin{corollary}
Let~$\cV$ be a variety of lattices such that either $M_3\in\cV$, $L_2\in\cV$, $L_3\in\cV$, $L_4\in\cV$, $\dual{L_2}\in\cV$, or $\dual{L_4}\in\cV$. The free bounded lattice on~$\aleph_1$ generators of~$\cV$ has no congruence $n$-permutable, congruence-preserving extension in the variety of all lattices, for each $n\ge 2$.
\end{corollary}

\section{Acknowledgment}
I wish to thank Friedrich Wehrung for all the corrections and improvements he brought to this paper.

\end{document}